\documentclass[11pt]{amsart}

\newtheorem{thm}{Theorem}[section]
\newtheorem*{thm*}{Theorem}
\newtheorem{lem}[thm]{Lemma}
\newtheorem*{lem*}{Lemma}

\newtheorem{cor}[thm]{Corollary}

\newtheorem{prop}[thm]{Proposition}

\theoremstyle{definition}

\newtheorem*{case*}{Case}

\newtheorem{defn}[thm]{Definition}
\newtheorem*{defn*}{Definition}
\newtheorem{exmp}[thm]{Example}
\newtheorem*{exmp*}{Example}

\newtheorem{hyp}[thm]{Hypothesis}

\newtheorem{step}{Step}\renewcommand{\thestep}{}

\theoremstyle{remark}

\newtheorem{case}{Case}\renewcommand{\thecase}{}

\newtheorem{rmk}[thm]{Remark}
\newtheorem*{rmk*}{Remark}

\makeatletter
\def\alphenumi{
  \def\theenumi{\alph{enumi}}
  \def\p@enumi{\theenumi}
  \def\labelenumi{(\@alph\c@enumi)}}
\makeatother

\makeatletter
\def\thecase{\@arabic\c@case}
\makeatother

\makeatletter
\def\thestep{\@arabic\c@step}
\makeatother

\newcount\hh
\newcount\mm
\mm=\time
\hh=\time
\divide\hh by 60
\divide\mm by 60
\multiply\mm by 60
\mm=-\mm
\advance\mm by \time
\def\hhmm{\number\hh:\ifnum\mm<10{}0\fi\number\mm}

\let\oldmarginpar\marginpar
\renewcommand\marginpar[1]{\-\oldmarginpar[\raggedleft\footnotesize #1]%
{\raggedright\footnotesize #1}}

\renewcommand\emptyset{\varnothing}

\newcommand\HH{\mathbb{H}}

\newcommand\RR{\mathbb{R}}

\newcommand\fa{{\mathfrak{a}}}

\newcommand\fb{{\mathfrak{b}}}

\newcommand\fK{{\mathfrak{K}}}

\newcommand\fw{{\mathfrak{w}}}

\newcommand\sD{{\mathscr{D}}}

\newcommand\sO{{\mathscr{O}}}

\newcommand\sS{{\mathscr{S}}}

\newcommand\sU{{\mathscr{U}}}

\newcommand\eps{\varepsilon}

\newcommand\less{\setminus}

\newcommand\dist{\operatorname{dist}}

\newcommand{\essinf}{\operatornamewithlimits{ess\ inf}}
\newcommand{\esssup}{\operatornamewithlimits{ess\ sup}}

\DeclareMathOperator{\Int}{int}

\newcommand\loc{\operatorname{loc}}

\newcommand\sign{\operatorname{sign}}

\newcommand\supp{\operatorname{supp}}

\newcommand\tr{\operatorname{tr}}

\newcommand\apriori{{\emph{a priori }}}

\numberwithin{equation}{section}

\usepackage{amssymb}
\usepackage{graphicx}
\usepackage{hyperref}
\usepackage{mathrsfs}
\usepackage{paralist}
\usepackage[usenames]{color}
\usepackage{url}
\usepackage{verbatim}

\hypersetup{pdftitle={Maximum principles for boundary-degenerate second-order linear elliptic differential operators}}
\hypersetup{pdfauthor={Paul M. N. Feehan}}

\begin{document}

\title[Maximum principles for elliptic operators]{Maximum principles for boundary-degenerate second-order linear elliptic differential operators}

\author[Paul M. N. Feehan]{Paul M. N. Feehan}
\address{Department of Mathematics, Rutgers, The State University of New Jersey, 110 Frelinghuysen Road, Piscataway, NJ 08854-8019, United States}
\email{feehan@math.rutgers.edu}

\date{September 9, 2013. Incorporates final galley proof corrections corresponding to published version. To appear in Communications in Partial Differential Equations, doi:10.1080/03605302.2013.831446.}

\begin{abstract}
We prove weak and strong maximum principles, including a Hopf lemma, for $C^2$ subsolutions to equations defined by second-order, linear elliptic partial differential operators whose principal symbols vanish along a portion of the domain boundary. The boundary regularity property of the $C^2$ subsolutions along this boundary vanishing locus ensures that these maximum principles hold irrespective of the sign of the Fichera function. Boundary conditions need only be prescribed on the complement in the domain boundary of the principal symbol's vanishing locus. We obtain uniqueness and \apriori maximum principle estimates for $C^2$ solutions to boundary value and obstacle problems defined by these boundary-degenerate elliptic operators with partial Dirichlet or Neumann boundary conditions. We also prove weak maximum principles and uniqueness for $W^{1,2}$ solutions to the corresponding variational equations and inequalities defined with the aide of weighted Sobolev spaces. The domain is allowed to be unbounded when the operator coefficients and solutions obey certain growth conditions.
\end{abstract}

%
%
%
%



\thanks{The author was partially supported by NSF grant DMS-1059206.}

\maketitle

\tableofcontents
\listoffigures

\section{Introduction}
The classical maximum principles of Fichera \cite{Fichera_1956, Fichera_1960} and Ole{\u\i}nik and Radkevi{\v{c}} \cite{Oleinik_Radkevic, Radkevich_2009a, Radkevich_2009b} provide uniqueness theorems for degenerate elliptic and parabolic boundary value problems which do not take into account a more modern view of the appropriate function spaces in which uniqueness is sought, such as \cite{Daskalopoulos_Feehan_statvarineqheston, DaskalHamilton1998, Daskalopoulos_Rhee_2003, Ekstrom_Tysk_bcsftse, Epstein_Mazzeo_annmathstudies, Feehan_Pop_regularityweaksoln, Feehan_Pop_mimickingdegen_pde, Koch}. Indeed, their maximum principles lead to the imposition of additional Dirichlet boundary conditions which are not necessarily motivated by the underlying application, whether in biology, finance, or physics. These additional Dirichlet boundary conditions, usually for certain ranges of parameters defining the partial differential equation, are often less natural than the physically-motivated regularity properties suggested by choices of appropriate weighted H\"older spaces \cite{DaskalHamilton1998, Daskalopoulos_Rhee_2003, Epstein_Mazzeo_annmathstudies, Feehan_Pop_mimickingdegen_pde} or Sobolev spaces \cite{Daskalopoulos_Feehan_statvarineqheston, Feehan_Pop_regularityweaksoln, Koch}, which automatically encode special regularity or integrability up to portions of the domain boundary where the operator becomes degenerate. In the case of weighted H\"older spaces, these boundary regularity properties may be viewed as a type of `second-order' or Ventcel boundary condition \cite{Anderson_1976a, Anderson_1976b, Taira_2004} (see Section  \ref{subsec:Boundary_degenerate_elliptic_operators_and_second_order_boundary_regularity} for further discussion).

However, the question of exactly how regular the solution should be near these boundary portions is delicate. If we ask for too much regularity, such as $C^2$ up to the boundary, the boundary value problem may have no solution or be supported by any existence theory, such as the $C^{2+\alpha}_s$ Schauder theory of Daskalopoulos and Hamilton \cite{DaskalHamilton1998}, which was further developed by the author and Pop \cite{Feehan_classical_perron_elliptic, Feehan_Pop_elliptichestonschauder}. If we ask for too little regularity, such as $C^0$ up to the boundary, we may need to require an unphysical Dirichlet boundary condition to ensure the problem is well-posed, with the unintended consequence that the solutions thus selected can be no more than continuous up to the boundary. An illustration of this point for an ordinary differential equation is provided by Example \ref{exmp:CIR} and a more extended discussion for partial differential equation on open subset of the half-plane is provided in Appendix \ref{sec:FicheraAndHeston}; see also Section  \ref{subsec:Boundary_degenerate_elliptic_operators_and_second_order_boundary_regularity}. Our Theorems \ref{thm:Weak_maximum_principle_C2_connected_domain}, \ref{thm:Weak_maximum_principle_C2} show that a useful intermediate concept of regularity up to the portion of the domain boundary where the principal symbol vanishes, namely $C^2_s$, is given by our Definition \ref{defn:Second_order_boundary_regularity}.

In Part \ref{part:BoundaryValueObstacleProblems} of our article, we prove weak and strong maximum principles for $C^2_s$ solutions to equations defined by linear, second-order, elliptic partial differential operators which are \emph{boundary-degenerate} in the sense that their principal symbols vanish along a portion, $\partial_0\sO$, of the topological boundary, $\partial\sO$, of an open subset $\sO\subset\RR^d$. (We use the term `boundary-degenerate' in this article to help clarify the distinction with the term `degenerate elliptic' as used by Crandall et al. in \cite{Crandall_Ishii_Lions_1992}; see Section  \ref{subsec:Boundary_degenerate_elliptic_operators_and_second_order_boundary_regularity} for further discussion.)
The $C^2_s$ boundary regularity property of the solutions along $\partial_0\sO$ ensures that our maximum principles hold irrespective of the sign of the Fichera function \cite{Fichera_1956, Fichera_1960}. In particular, we only require boundary comparisons for subsolutions and supersolutions along $\partial_1\sO=\partial\sO\less\overline{\partial_0\sO}$ and not $\partial\sO$. In Part \ref{part:VariationalEquationInequality} of our article, we prove weak maximum principles and \apriori maximum principle estimates for $W^{1,2}$ solutions to the corresponding variational equations and inequalities defined using weighted Sobolev spaces.

Although the boundary-degenerate elliptic operators discussed in this article are degenerate along a portion $\partial_0\sO\subset\partial\sO$ of the boundary of the open subset $\sO\subset\RR^d$, it is possible to prove \emph{existence} of solutions which are $C^{2+\alpha}_s$ up to $\partial_0\sO$ but possibly only $C^0$ up to $\partial_1\sO=\partial\sO\less\overline{\partial_0\sO}$. Indeed, a Schauder approach employing Daskalopoulos-Hamilton weighted H\"older spaces to such an existence result is described by the author in \cite{Feehan_classical_perron_elliptic}, building on earlier results of the author and C. Pop \cite{Feehan_Pop_elliptichestonschauder}, while a variational approach employing weighted Sobolev spaces, due to P. Daskalopoulos, the author, and C. Pop can be found in our articles \cite{Daskalopoulos_Feehan_statvarineqheston, Feehan_Pop_regularityweaksoln, Feehan_Pop_higherregularityweaksoln}. Under suitable hypotheses on the coefficients of $A$, the regularity of the boundary $\partial\sO$ of the open subset $\sO$, and the geometry of the intersection involving the boundary portions $\partial_0\sO$ and $\partial_1\sO$, which often meet at a domain corner as illustrated in Figure \ref{fig:domain}, it should be possible to prove that solutions are actually smooth up to the whole boundary, $\partial\sO$. However, this appears to be a difficult problem (see \cite{Feehan_classical_perron_elliptic, Feehan_Pop_higherregularityweaksoln} for a discussion of the issues) which has not been addressed in the literature thus far (to the best of our knowledge) and remains one we plan to address in future articles. As far as \emph{uniqueness} of solutions is concerned, however, that is the topic addressed by our present article.

\subsection{Boundary value and obstacle problems for boundary-degenerate linear, second-order partial differential operators}
\label{subsec:Boundary_value_problems}
We now describe a class of boundary-degenerate operators which we shall consider in this article, along with degenerate and non-degenerate boundary portions and boundary value and obstacle problems with partial boundary data.

Let $\sO\subset\RR^d$, where $d\geq 1$, be an open, possibly unbounded, subset with boundary $\partial\sO$ and suppose that $\Sigma\subseteqq\partial\sO$ is a relatively open subset. Given $1\leq p\leq\infty$, we let $W^{2,p}(\sO)$ (respectively, $W^{2,p}_{\loc}(\sO)$) denote the Sobolev space of measurable functions, $u$ on $\sO$, such that $u$ and its weak first and second derivatives, $u_{x_i}$ and $u_{x_ix_j}$ for $1\leq i,j\leq d$, belong to $L^p(\sO)$  (respectively, $L^p_{\loc}(\sO)$) \cite[Section 3.1]{Adams_1975}. (We summarize frequently-used notation in Section \ref{subsec:Notation}.) In Part \ref{part:BoundaryValueObstacleProblems} of this article, we consider the question of uniqueness and \apriori maximum principle estimates for solutions in $C^2(\sO)$ or $W^{2,p}_{\loc}(\sO)$ to the elliptic equation,
\begin{equation}
\label{eq:Elliptic_equation}
Au = f \quad \hbox{(a.e.) on  } \sO,
\end{equation}
or solutions in $W^{2,p}_{\loc}(\sO)$ to the elliptic obstacle problem,
\begin{equation}
\label{eq:Elliptic_obstacle_problem}
\min\{Au-f,u-\psi\} = 0 \quad\hbox{a.e. on }\sO,
\end{equation}
both with \emph{partial} Dirichlet boundary condition,
\begin{equation}
\label{eq:Elliptic_boundary_condition}
u =g \quad\hbox{on } \partial\sO\less\bar\Sigma.
\end{equation}
Here, $\psi$ is compatible with $g$ in the sense that
\begin{equation}
\label{eq:Elliptic_obstacle_boundarydata_compatibility}
\psi\leq g \quad\hbox{on } \partial\sO\less\bar\Sigma.
\end{equation}
In typical applications to mathematical finance \cite{Bensoussan_Lions, Shreve2}, where one can have a non-empty boundary portion $\Sigma\subsetneqq\partial\sO$, the partial boundary condition \eqref{eq:Elliptic_boundary_condition} represents a barrier condition in option valuation problems; the obstacle condition, $u\geq\psi$ on $\sO$, in \eqref{eq:Elliptic_obstacle_problem} arises in all American-style option valuation problems; the absence of a boundary condition along $\Sigma$ is natural in valuation problems for options contingent on an asset modeled by stochastic volatility processes such as the Heston process \cite{Heston1993}. See Example \ref{exmp:HestonPDE} and Appendix \ref{sec:FicheraAndHeston} for further discussion involving the generator of the Heston process.

Part \ref{part:VariationalEquationInequality} of our article is concerned with the question of uniqueness and \apriori maximum principle estimates for solutions in $W^{1,2}_{\loc}(\sO)$ to variational equations or inequalities corresponding to \eqref{eq:Elliptic_equation} or \eqref{eq:Elliptic_obstacle_problem}.

Let $\sS(d)\subset\RR^{d\times d}$ denote the subset of symmetric matrices and $\sS^+(d)\subset\sS(d)$ denote the subset of non-negative definite, symmetric matrices. The operator,
\begin{equation}
\label{eq:Generator}
Au := -\tr(aD^2u) - \langle b,Du\rangle + cu, \quad u \in C^\infty(\sO),
\end{equation}
is defined by coefficients,
\begin{subequations}
\label{eq:A_coefficients}
\begin{align}
\label{eq:a_nonnegative_symmetric_matrix_valued}
  {}& a:\sO \to \sS^+(d),
  \\
\label{eq:b_coefficient}
  {}& b:\sO \to \RR^d,
  \\
\label{eq:c_coefficient}
  {}& c:\sO \to \RR,
\end{align}
\end{subequations}
which may be everywhere-defined or measurable on $\sO$, depending on whether we consider solutions in $C^2(\sO)$ or $W^{2,p}_{\loc}(\sO)$, respectively. We use $D^2u$ and $Du$ to denote the Hessian matrix and gradient of $u$, respectively, so $\tr(aD^2u) = a^{ij}u_{x_ix_j}$ and $\langle b,Du\rangle = b^iu_{x_i}$, where Einstein's summation convention is used throughout this article.

For the sake of clarity, we shall confine our attention to linear operators in this article, although many of the results can be seen to extend to \emph{semilinear} operators,
$$
S(u) = -\tr(aD^2u) - \langle b,Du\rangle + c(\cdot,u), \quad u \in C^2(\sO),
$$
or \emph{quasilinear} operators \cite[Section 10]{GilbargTrudinger},
$$
Q(u) = -\tr(aD^2u) - b(\cdot,u,Du), \quad u \in C^2(\sO),
$$
as well as linear and quasilinear \emph{parabolic} operators with non-negative definite characteristic form. We shall discuss these applications in future articles.

We call
\begin{equation}
\label{eq:Degeneracy_locus_elliptic}
\partial_0\sO := \Int\left\{x\in\partial \sO: \lim_{\sO\ni x'\to x}a(x') = 0\right\},
\end{equation}
the \emph{degenerate boundary} defined by \eqref{eq:a_nonnegative_symmetric_matrix_valued}, where $\Int S$ denotes the interior of a subset $S$ of a topological space. (Our definition is a variant of that used by Fichera \cite{Fichera_1960} and Ole{\u\i}nik and Radkevi{\v{c}} \cite{Oleinik_Radkevic}, \cite[p. 308]{Radkevich_2009a}. We follow the definition of Ole{\u\i}nik, and Radkevi{\v{c}} \cite[p. 308]{Radkevich_2009a} rather than Tricomi \cite[p. 298]{Radkevich_2009a}, which requires in addition that $\langle a\eta, \eta\rangle>0$ on $\sO$ for all $\eta\in\RR^d\less\{0\}$.)

When the matrix $a$ in \eqref{eq:a_nonnegative_symmetric_matrix_valued} is continuous on $\partial\sO$, then the definition of $\partial_0\sO$ in \eqref{eq:Degeneracy_locus_elliptic} is equivalent to the definition
$$
\partial_0\sO = \Int\{x\in\partial\sO: \langle a(x)\eta,\eta\rangle = 0, \forall\, \eta\in\RR^d\},
$$
used by Fichera \cite{Fichera_1960} and Ole{\u\i}nik and Radkevi{\v{c}} \cite{Oleinik_Radkevic}, \cite[p. 308]{Radkevich_2009a} (where $\partial_0\sO$ is denoted by $\Sigma^0$ and the topological boundary of $\sO$ is denoted by $\Sigma$). In \cite[Section 1.1]{Radkevich_2009a}, an operator, $A$, as in \eqref{eq:Generator} is said to have a \emph{non-negative definite characteristic form} when the coefficient matrix, $a$, is non-negative definite on $\sO$, as assumed in \eqref{eq:a_nonnegative_symmetric_matrix_valued}.

In this article, we allow $\partial_0\sO$ to be non-empty and denote
\begin{equation}
\label{eq:Domain_plus_degenerate_boundary_elliptic}
\underline \sO := \sO\cup\partial_0\sO.
\end{equation}
We also call
\begin{equation}
\label{eq:Elliptic_nondegeneracy_locus}
\partial_1\sO := \partial \sO\less\overline{\partial_0\sO},
\end{equation}
the \emph{non-degenerate boundary} defined by \eqref{eq:a_nonnegative_symmetric_matrix_valued} and observe that
\begin{equation}
\label{eq:Degenerate_boundary_decomposition}
\partial\sO = \overline{\partial_0\sO}\cup\partial_1\sO = \partial_0\sO\cup\overline{\partial_1\sO},
\end{equation}
where $\overline{T}$ indicates closure of a subset $T\subset\partial\sO$ with respect to the topological boundary, $\partial\sO$. For the Kummer operator in Example \ref{exmp:CIR}, if $\sO=(0,\ell)$ and $0<\ell<\infty$, then
$$
\partial\sO = \{0, \ell\}, \quad \partial_0\sO = \{0\}, \quad\hbox{and}\quad  \partial_1\sO = \{\ell\}.
$$
In certain contexts, we may require that\footnote{It is likely that $C^1$ would suffice, but an assumption that $\partial_0\sO$ is a boundary portion of class $C^{1,\alpha}$ simplifies certain proofs --- see \cite[Lemma B.1]{Feehan_perturbationlocalmaxima}.}
\begin{equation}
\label{eq:C1alpha_degenerate_boundary}
\partial_0\sO \quad\hbox{is } C^{1,\alpha},
\end{equation}
and let $\vec n$ denote the \emph{inward}-pointing unit normal vector field along $\partial_0\sO$. The partial boundary condition \eqref{eq:Elliptic_boundary_condition} may arise, for example, by choosing
$$
\Sigma = \partial_0\sO.
$$
See Figure \ref{fig:domain}.

\begin{figure}
 \centering
 \begin{picture}(210,210)(0,0)
 \put(0,0){\includegraphics[scale=0.6]{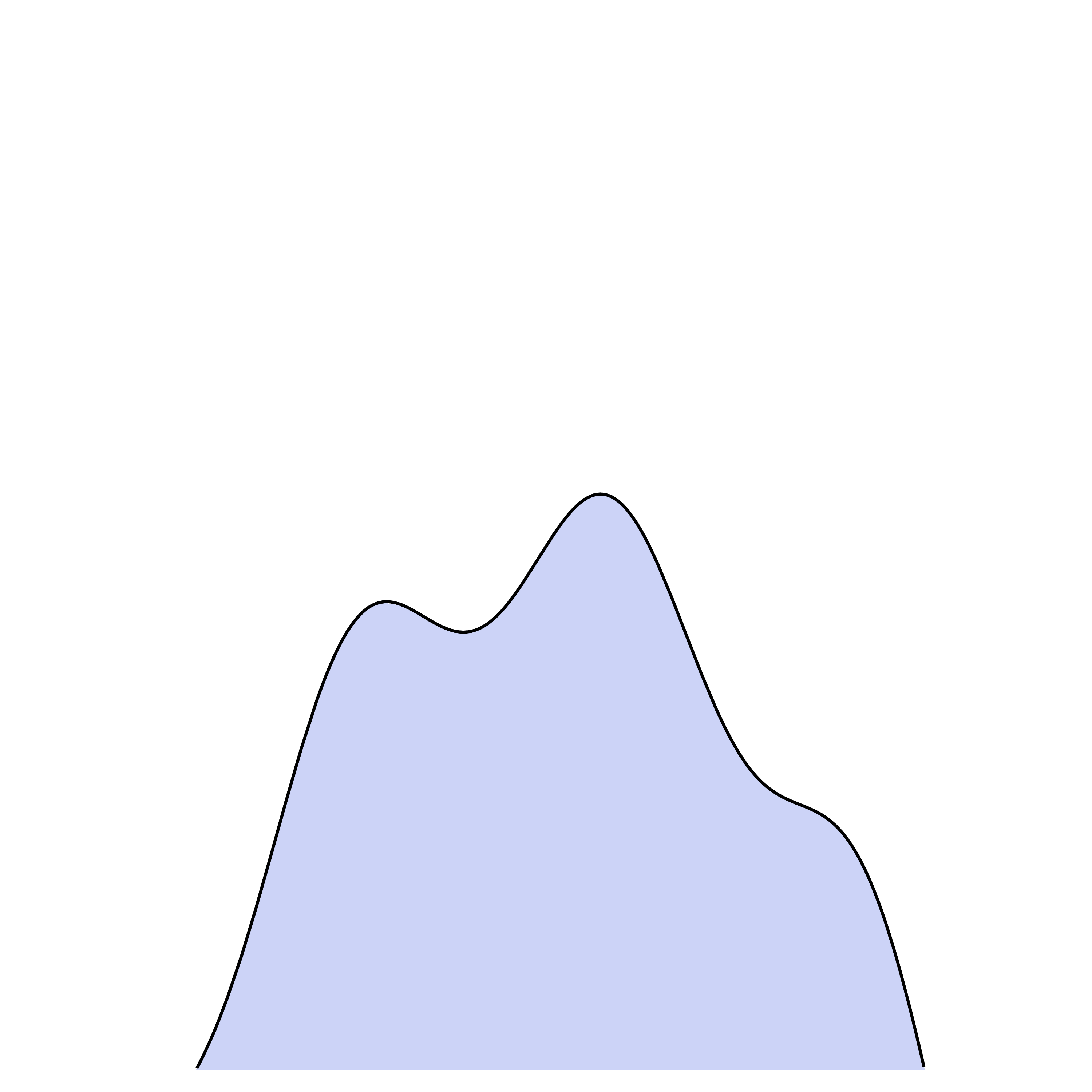}}
 \put(90,10){$\partial_0\sO$}
 \put(105,50){$\sO$}
 \put(110,95){$\partial_1\sO$}
 \put(150,80){$\RR^d$}
 \end{picture}
 \caption[A domain and its `degenerate' and `non-degenerate' boundaries.]{A subdomain, $\sO\subset\RR^d$, and its `degenerate' and `non-degenerate' boundaries, $\partial_0\sO$ and $\partial_1\sO$. In maximum principles, the degenerate boundary portion, $\partial_0\sO$, plays the same role as the interior of the domain, $\sO$.}
 \label{fig:domain}
\end{figure}
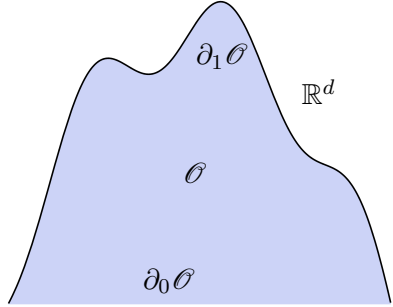

However, when deriving \apriori maximum principle estimates (as in Section \ref{sec:Applications_weak_maximum_principle_property_boundary_value_problems} or  \ref{sec:Applications_weak_maximum_principle_property_obstacle_problems}) it is convenient to allow for possibly greater generality. For example, if the condition $\lim_{\sO\ni x'\to x}a(x')=0$ in the definition \eqref{eq:Degeneracy_locus_elliptic} of $\partial_0\sO$ is replaced by $\lim_{\sO\ni x'\to x}\det a(x')=0$, selected results from this article might still be expected to hold, as suggested by Example \ref{exmp:Keldys} due to M. V. Keld\v{y}s.

\subsection{Examples}
\label{subsec:Examples}
The conditions on the coefficients of $A$ in \eqref{eq:Generator} assumed in this article are mild enough that they allow for many examples of partial differential operators, $A$, with non-negative definite characteristic form and which are of interest in mathematical biology, finance, and physics. Before proceeding to a discussion of our main results, we shall first provide some specific examples of operators to which our results apply. While the examples in this section often discuss operators, $A$, defined on a half-space, $\sO=\HH$, so $\partial_0\sO=\partial\sO=\partial\HH$, applications may require us to consider subdomains $\sO\subsetneqq\HH$ where $\partial_0\subsetneqq\partial\sO$. Such situations arise frequently in mathematical finance due to barrier conditions in option contracts \cite{Shreve2}.

We begin by describing a family of examples which includes certain stochastic volatility models occurring in mathematical finance \cite{Heston1993} and the linearization of the porous medium equation \cite{DaskalHamilton1998}.

\begin{exmp}[Affine coefficients and degeneracy on the boundary of a half-space]
\label{exmp:ScalarAffine}
Suppose the coefficients of $A$ in \eqref{eq:Generator} are affine functions of $x\in\RR^d$, with $a(x)$ positive definite for all $x\in\HH$, where $\HH=\RR^{d-1}\times\RR_+$ is a half-space and $\RR_+=(0,\infty)$, while $a(x)=0$ if $x\in\partial\HH$. Then
\begin{equation}
\label{eq:ScalarAffinePDE}
Au = -\tr(x_da_1D^2u) - \langle b_0+b_1x,Du\rangle + (c_0+\langle c_1,x\rangle)u, \quad u\in C^\infty(\HH),
\end{equation}
where $a_1,b_1\in\RR^{d\times d}$ and $b_0, c_1\in\RR^d$ and $c_0\in\RR$. Thus, $A$ is an elliptic partial differential operator on $C^\infty(\HH)$ which becomes degenerate along the boundary $\partial_0\sO=\partial\HH=\{x_d=0\}$ of the half-space $\HH=\{x_d>0\}$.

When $a_1$ is symmetric, the operator $-A$ is the generator of a degenerate diffusion process. Imposing the condition $\langle b,\vec n\rangle \geq 0$ along $\partial\HH$, where $\vec n$ is the \emph{inward}-pointing unit normal vector field, ensures that the diffusion process remains in the half-space $\bar\HH=\{x_d\geq 0\}$ if started in $\bar\HH$. Since $\vec n = e_d$, this translates to the requirement that $b^d(x)\geq 0$ for all $x\in\partial\HH$, and thus $b_0^d\geq 0$ and $b_1^{dj}=0$ for $1\leq j\leq d-1$ and $b_1^{dd}\geq 0$.

To ensure uniqueness of solutions to \eqref{eq:Elliptic_equation}, \eqref{eq:Elliptic_boundary_condition} on $\HH$ via our maximum principle, it is necessary (though not sufficient) to impose the condition $c\geq 0$ on $\HH$, and hence $c_0\geq 0$ and $c_1^i=0$ for $1\leq i\leq d-1$ and $c_1^d\geq 0$.

Examples of this kind occur frequently in mathematical finance, such as the stochastic volatility process defined by Heston \cite{Heston1993}, where $a(x)=x_da_1$ for $x\in\HH$, and $a_1\in\RR^{d\times d}$ is positive definite, and $d=2$. See Example \ref{exmp:HestonPDE} for a description of the Heston process generator and an important example of a degenerate affine process. The linearization of the porous medium operator is another important example of this type; see Example \ref{exmp:GeneralizedPorousMediumEquation} for details.
\end{exmp}

\begin{exmp}[Elliptic Heston operator]
\label{exmp:HestonPDE}
The generator of the \emph{Heston process} \cite{Heston1993} provides a well-known example in mathematical finance of the operator \eqref{eq:ScalarAffinePDE} when $d=2$:
\begin{equation}
\label{eq:HestonPDE}
Au := -\frac{x_2}{2}\left(u_{x_1x_1} + 2\varrho\sigma u_{x_1x_2} + \sigma^2 u_{x_2x_2}\right) - \left(r-q-\frac{x_2}{2}\right)u_{x_1} - \kappa(\theta-x_2)u_{x_2} + ru,
\end{equation}
where $q\in\RR, r\geq 0, \kappa>0, \theta>0, \sigma\neq 0$, and $\varrho\in (-1,1)$ are constants (their financial interpretation is provided in \cite{Heston1993}), and $u\in C^\infty(\HH)$, with $\HH=\RR\times\RR_+$. The variables $x_1$ and $x_2$ represent the log price and stochastic variance, respectively, of a financial asset while $u$ represents the option value. Finite maturity options will also involve the time variable, $t$, and yield parabolic boundary value or obstacle problems, though perpetual, American-style option pricing problems are stationary and yield elliptic obstacle problems such as \eqref{eq:Elliptic_obstacle_problem}.

When pricing an American or European-style option contingent on an asset modeled by the Heston process, without a barrier condition, one would have $\sO=\RR\times\RR_+$, so $\Sigma=\partial_0\sO=\partial\sO=\RR\times\{0\}$. For a single (upper) barrier option with barrier at $x_1$, one would take $\sO = (-\infty,x_1)\times\RR_+$, so $\partial_0\sO=(-\infty,x_1)\times\{0\}$ and $\partial_1\sO=\{x_1\}\times\RR_+$. For a double barrier option with barriers at $x_1<x_2$, one would take $\sO = (x_1,x_2)\times\RR_+$, so $\partial_0\sO=(x_1,x_2)\times\{0\}$ and $\partial_1\sO=\{x_1,x_2\}\times\RR_+$. The text \cite{Shreve2} by Shreve provides an introduction to the concepts of mathematical finance mentioned in this example.
\end{exmp}

\begin{exmp}[Generator of the Feller square-root or Cox-Ingersoll-Ross process]
\label{exmp:CIR}
The (negative of the) generator of the \emph{Feller square-root} process \cite{Feller1951}, known as the \emph{Cox-Ingersoll-Ross} process in mathematical finance \cite{CoxIngersollRoss1985}, \cite[Example 6.5.2]{Shreve2}, provides a simple example of the operator \eqref{eq:ScalarAffinePDE} when $d=1$,
\begin{equation}
\label{eq:CIRODE}
Au := -\frac{\sigma^2}{2}xu_{xx} - \kappa(\theta-x)u_x + ru,
\end{equation}
where $u\in C^\infty(0,\infty)$, and which takes the form, after a change of variables, of the \emph{Kummer} operator,
\begin{equation}
\label{eq:KummerODE}
\tilde Av := -xv_{xx} - (\beta-x)v_x + \alpha v,
\end{equation}
where $v\in C^\infty(0,\infty)$ and $\beta := 2\kappa\theta/\sigma^2 > 0$ and $\alpha := r/\kappa \geq 0$. The homogeneous Kummer equation, $\tilde Av=0$ on $(0,\infty)$, has two independent
solutions, the \emph{confluent hypergeometric function of the first kind\/} $M(\alpha,\beta;x)$ (or Kummer function) and the  \emph{confluent hypergeometric function of the second kind} $U(\alpha,\beta;x)$ (or Tricomi function) \cite[Sections 13.1.2 and 13.1.3]{AbramStegun}. The solution $M(\alpha,\beta;x)$ is in $C^\infty[0,\infty)$, with $M(\alpha,\beta;0)=1$, $M_x(\alpha,\beta;0)=\alpha/\beta$, and $M_{xx}(\alpha,\beta;0)=\alpha(\alpha+1)/(\beta(\beta+1))$ \cite[Section 13.4.9]{AbramStegun}. Near $x=0$, the solution $U(\alpha,\beta;x)$ is comparable to $x^{1-\beta}$ when $\beta\neq 1$ and $\ln x$ when $\beta=1$  \cite[Sections 13.5.6--12]{AbramStegun}, and thus is in $C[0,\infty)$ for $0<\beta<1$, with $U_x(\alpha,\beta;x)$ comparable to $x^{-\beta}$ and $U_{xx}(\alpha,\beta;x)$ comparable to $x^{-\beta-1}$ near $x=0$, for any $\beta>0$ \cite[Section 13.4.22]{AbramStegun}, and thus not even in $C^1[0,\infty)$ when $\beta>0$.

In the context of Definition \ref{defn:Second_order_boundary_regularity}, we see that $\sO=(0,\infty)$, $\partial_0\sO=\{0\}$, and $M\in C^2_s[0,\infty)$ while $U\notin C^2_s[0,\infty)$.
\end{exmp}

\begin{exmp}[Linearization of the porous medium operator]
\label{exmp:GeneralizedPorousMediumEquation}
In a landmark article, Daskalopoulos and Hamilton \cite{DaskalHamilton1998} proved existence and uniqueness of $C^\infty$ solutions, $u$, to the Cauchy problem for the porous medium equation \cite[p. 899]{DaskalHamilton1998} (when $d=2$),
\begin{equation}
\label{eq:PorousMediumEquation}
-u_t + \sum_{i=1}^d (u^m)_{x_ix_i} = 0 \quad\hbox{on }(0,T)\times\RR^d, \quad u(\cdot, 0) = g \quad\hbox{on }\RR^d,
\end{equation}
with constant $m>1$ and initial data, $g\geq 0$, compactly supported on $\RR^d$, together with $C^\infty$-regularity of its free boundary, $\partial\{u>0\}$, provided the initial pressure function is non-degenerate (that is, $D u^{m-1}  \geq a >0$) on the boundary of its support at $t=0$. Their analysis is based on their development of existence, uniqueness, and regularity results for the linearization of the porous medium equation near the free boundary and, in particular, their \emph{model linear degenerate operator} \cite[p. 901]{DaskalHamilton1998} (generalized from $d=2$ in their article),
\begin{equation}
\label{eq:DHModel}
Au := -x_d\sum_{i=1}^d u_{x_ix_i} - \beta u_{x_d}, \quad u\in C^\infty(\HH),
\end{equation}
where $\beta$ is a positive constant, analogous to the combination of parameters $2\kappa\theta/\sigma^2$ in \eqref{eq:HestonPDE}, and $\HH=\RR^{d-1}\times\RR_+$, following a suitable change of coordinates \cite[p. 941]{DaskalHamilton1998}. The porous medium equation and the same model linear degenerate operator (for $d\geq 2$) were studied independently by Koch \cite[Equation (4.43)]{Koch} and, in a remarkable Habilitation thesis, he obtained existence, uniqueness, and regularity results for solutions to \eqref{eq:PorousMediumEquation} which complement those of Daskalopoulos and Hamilton \cite{DaskalHamilton1998}.
\end{exmp}

Example \ref{exmp:ScalarAffine} describes a class of elliptic differential operators which become degenerate along the boundary of a half-space, $\RR^{d-1}\times\RR_+$; their coefficients are affine functions of $x\in\RR^d$ and $\RR^{d-1}\times\bar\RR_+$ is a state space for the corresponding Markov process when the coefficient $a(x)$ is symmetric. More generally, mathematical finance and biology provide examples of elliptic differential operators which become degenerate along the boundary of a `quadrant', $\RR^{d-m}\times\RR_+^m$, and $\RR^{d-m}\times\bar\RR_+^m$ is a state space for the corresponding Markov process. Examples primarily motivated by mathematical finance include affine processes \cite{Abdelkoddousse_Aurelien_2013, Bru_1991, Cuchiero_Filipovic_Mayerhofer_Teichmann_2011, Duffie_2005, Duffie_Filipovic_Schachermayer_2003, DuffiePanSingleton2000, Filipovic_2005, Filipovic_Mayerhofer_2009}, which may be viewed as extensions of geometric Brownian motion (see, for example, \cite{Shreve2}), the Heston stochastic volatility process \cite{Heston1993}, and the Wishart process \cite{Bru_1991, DaFonseca_Grasselli_Tebaldi_2008, Gnoatto_Grasselli_2014, Grasselli_Tebaldi_2007, Grasselli_Tebaldi_2008}. Examples of this kind which arise in mathematical biology include the multi-dimensional Kimura diffusions and their local model processes \cite[Equations (1.5) and (1.20)]{Epstein_Mazzeo_annmathstudies}. Another example along these lines is provided by the articles of Athreya et al. \cite{Athreya_Barlow_Bass_Perkins_2002, Bass_Perkins_2003, Bass_Lavrentiev_2007} concerning generators of super-Markov chains.

For many of the elliptic operators in the references just cited, the degeneracy is of the kind $\det a(x)=0$ rather than $a(x)=0$ at boundary points $x\in\partial\sO$. Two well-known operators where only part of the principal symbol vanishes at the boundary are described in Examples \ref{exmp:SABR} and \ref{exmp:Keldys} and which also illustrate an order of vanishing which is not linear. While the order of vanishing does not impact proofs of our maximum principles for $C^2$ functions (see Theorems \ref{thm:Strong_maximum_principle}, \ref{thm:Weak_maximum_principle_C2_connected_domain}, and \ref{thm:Weak_maximum_principle_C2}), arbitrarily high order of vanishing is not permitted by the proofs of some our weak maximum principles for weak solutions to variational equations (see Theorem \ref{thm:WeakMaximumPrincipleVarEquationBoundedDomain}).

\begin{exmp}[SABR model in interest rate derivative modeling]
\label{exmp:SABR}
The generator, $-A$, of the two-dimensional `SABR process' due to Hagan et al. \cite{Hagan_Kumar_Lesniewski_Woodward_2002}, in suitable coordinates, can be shown to be
\begin{equation}
\label{eq:SABR}
Au = -\frac{1}{2}\left(x_2^{2\beta} e^{2x_1}u_{x_2x_2} + 2\varrho\alpha x_2^\beta e^{x_1}u_{x_1x_2} + \alpha^2u_{x_1x_1}\right) + \frac{\alpha^2}{2}u_{x_1}, \quad u \in C^\infty(\HH),
\end{equation}
where $\alpha > 0$ and $0<\beta<1$ and $-1<\varrho<1$ are constants. The SABR process is widely used by financial engineers for interest rate modeling \cite{Henry-Labordere_2009}, but the results of our present article do not immediately extend to cover operators such as \eqref{eq:SABR}.
\end{exmp}

\begin{exmp}[Keldy{\v{s}} operator]
\label{exmp:Keldys}
Keldy\v{s} \cite{Keldys_1951} provided another important early example, related to the SABR partial differential operator described in Example \ref{exmp:SABR}, which is also not covered by the results of this article. Let $a(x_1,x_2)$, $b(x_1,x_2)$, $c(x_1,x_2)$ be analytic functions of their real variables, with $c\leq 0$ on $\sO\subset\RR^2$, and\footnote{For ease of comparison with  \cite{Keldys_1951}, we keep the notation and sign conventions of Keldy{\v{s}} in our example, though they differ from our usual choice in \eqref{eq:Generator}.}
\begin{equation}
\label{eq:Keldys}
Au := x_2^m u_{x_2x_2} + u_{x_1x_1} + au_{x_2} + bu_{x_1} + cu,
\end{equation}
where $\sO$ is a simply-connected domain bounded by the segment $(0,1)$ of the $x_1$-axis and a smooth curve $\Gamma$ situated in the upper half-plane. Two problems are considered for the equation $Au=0$ on $\sO$: (D) $u$ must assume given continuous values on the whole boundary $\partial\sO$; (E) $u$ must assume given continuous values on $\Gamma$ and be bounded in $\sO$. Keldy\v{s} proves the following results. (D) is uniquely solvable if $m<1$, or $m=1$ and $a(x_1,0)<1$, or $1<m<2$ and $a(x_1,0)\leq 0$, or $m\geq 2$ and $a(x_1,0)<0$. (E) is uniquely solvable if $m=1$, $a(x_1,0)\geq 1$, or $1<m<2$ and $a(x_1,0)>0$, or $m\geq 2$ and $a(x_1,0)\geq 0$. See \cite[Section 6.6 and Problem 6.10]{GilbargTrudinger} for a related discussion.
\end{exmp}

In a sequel to this article, we shall consider maximum principles which allow us to include operators such as those described in Examples \ref{exmp:SABR} and \ref{exmp:Keldys} and mentioned in the preceding discussion, where $A$ may become degenerate along a \emph{stratified space} $\Sigma\subseteqq\partial\sO$, in the broader sense that
$$
\Sigma = \{x\in\partial\sO:\det a(x) = 0 \},
$$
rather than $\Sigma = \Int\{x\in\partial\sO: a(x)=0\}$.

\subsection{Properties of coefficients of boundary-degenerate elliptic operators}
\label{subsec:Properties_coefficients}
Before summarizing our main results, it is convenient to collect here the main properties, along with their variants, for the coefficients of the operator $A$ used in Part \ref{part:BoundaryValueObstacleProblems} of this article which are motivated by the examples discussed in Section \ref{subsec:Examples}. The reader should keep in mind that unless stated otherwise, a coefficient property is \emph{only} assumed to hold when explicitly invoked in the statement of a lemma, proposition, or theorem.

Given $a$ as in \eqref{eq:a_nonnegative_symmetric_matrix_valued} (everywhere-defined or measurable), let $\lambda(x)$ denote the smallest eigenvalue of the matrix, $a(x)$, for $x\in\sO$, and let
$$
\lambda_*:\sO\to[0,\infty)
$$
be the lower semi-continuous envelope\footnote{When $f:X\to[0,\infty)$ is a measurable function on a measure space $(X,\Sigma,\mu)$, then $f_*:X\to[0,\infty)$ is the largest lower-semicontinuous function on $X$ such that $f_*\leq f$ $\mu$-a.e. on $X$.} of the resulting least eigenvalue function, $\lambda:\sO\to[0,\infty)$, for $a:\sO\to\sS^+(d)$. We may require that $a:\sO\to\sS^+(d)$ be \emph{locally strictly elliptic on the interior}, $\sO$, in the sense that
\begin{equation}
\label{eq:a_locally_strictly_elliptic_interior_domain}
\lambda_* > 0 \quad\hbox{on } \sO.
\end{equation}
The vector field, $\vec n:\partial_0\sO\to\RR^d$, may be extended to a tubular neighborhood $N(\partial_0\sO)$ of $\partial_0\sO \subset \underline\sO$. We can then split the vector field, $b:N(\partial_0\sO)\to\RR^d$, into its normal and tangential components, with respect to the extended vector field, $\vec n:N(\partial_0\sO)\to\RR^d$, so
\begin{equation}
\label{eq:b_splitting_elliptic}
b^\perp := \langle b, \vec n \rangle \quad\hbox{and}\quad b^\parallel := b - b^\perp \vec n \quad\hbox{on } N(\partial_0\sO).
\end{equation}
We may require that the vector field $b^\perp$ obey one of the following conditions,
\begin{align}
\label{eq:b_perp_nonnegative_boundary_elliptic}
b^\perp &\geq 0 \quad \hbox{on } \partial_0\sO, \quad\hbox{or}
\\
\label{eq:b_perp_positive_boundary_elliptic}
\tag{\ref*{eq:b_perp_nonnegative_boundary_elliptic}$'$}
b^\perp &> 0 \quad \hbox{on } \partial_0\sO.
\end{align}
The coefficient $c$ in \eqref{eq:Generator} may obey one of
\begin{align}
\label{eq:c_nonnegative_domain}
c &\geq 0 \quad\hbox{(a.e.) on }\sO, \quad\hbox{or}
\\
\label{eq:c_positive_domain}
\tag{\ref*{eq:c_nonnegative_domain}$'$}
c &> 0 \quad\hbox{(a.e.) on }\sO, \quad\hbox{or}
\\
\label{eq:c_lower*_positive_domain}
\tag{\ref*{eq:c_nonnegative_domain}$''$}
c_* &> 0 \quad \hbox{on } \sO, \quad\hbox{or}
\\
\label{eq:c_positive_lower_bound_domain}
\tag{\ref*{eq:c_nonnegative_domain}$'''$}
c &\geq c_0 \quad\hbox{(a.e.) on }\sO,
\end{align}
for some constant $c_0>0$ and where $c_*:\sO\to[0,\infty)$ is the lower semicontinuous envelope of $c$. We may also require that $c$ obey one of the conditions,
\begin{align}
\label{eq:c_nonnegative_boundary}
c &\geq 0 \quad \hbox{on } \partial_0\sO, \quad\hbox{or}
\\
\label{eq:c_positive_boundary}
\tag{\ref*{eq:c_nonnegative_boundary}$'$}
c &> 0 \quad \hbox{on } \partial_0\sO.
\end{align}
We may require that $b$ or $c$ be (essentially) locally bounded in the following senses (note the distinction between $\sO$ and $\underline\sO$)
\begin{subequations}
\label{eq:bc_locally_bounded}
\begin{align}
\label{eq:b_locally_bounded_on_domain}
b &\in L^\infty_{\loc}(\sO;\RR^d), \quad\hbox{or}
\\
\label{eq:c_locally_bounded_on_domain_plus_degenerate_boundary}
c &\in L^\infty_{\loc}(\underline\sO),
\end{align}
\end{subequations}
where we slightly abuse notation by, for example, writing $w\in L^\infty_{\loc}(\underline\sO)$ as an abbreviation for saying that $w$ is a locally bounded function on $\underline\sO$, irrespective of whether $w$ is measurable or everywhere-defined.

We may also require that the coefficient $b$ be continuous along $\partial_0\sO$,
\begin{equation}
\label{eq:b_continuous_on_degenerate_boundary}
b \in C(\partial_0\sO;\RR^d).
\end{equation}
When the open subset $\sO$ is \emph{unbounded}, we may need a growth condition,
\begin{equation}
\label{eq:Quadratic_growth}
\tr a(x) + \langle b(x),x\rangle \leq K(1+|x|^2), \quad\forall\, x \in \underline\sO,
\end{equation}
for some positive constant $K$.

\subsection{Summary of main results and outline of our article}
\label{subsec:Summary}
We shall leave detailed statements of our main results to the body of our article and simply provide a short outline of our article here to facilitate the reader seeking a particular conclusion of interest. Part \ref{part:BoundaryValueObstacleProblems} of our article (Sections \ref{sec:Applications_weak_maximum_principle_property_boundary_value_problems}--\ref{sec:Weak_maximum_principle_elliptic_C2}) develops weak and strong maximum principles for operators on smooth functions and applications to boundary value and obstacle problems, while Part \ref{part:VariationalEquationInequality} of our article (Sections \ref{sec:ApplicationsWeakMaxPrinPropertyVarEq}--\ref{sec:WeakMaxPrincipleSobolev}) develops weak maximum principles for bilinear maps and operators on functions in Sobolev spaces and applications to variational equations and inequalities.

\subsubsection{Weak and strong maximum principles for operators on smooth functions and applications to boundary value and obstacle problems}
We review some terminology from \cite{Feehan_perturbationlocalmaxima} to help summarize our results. Given a real-valued function $u$ on an open subset $\sO\subset\RR^d$, we let $u^*:\bar\sO\to[-\infty,\infty]$ denote its upper semicontinuous envelope on $\bar\sO$.

\begin{defn}[Generalized subharmonic functions]
\label{defn:A_subharmonic}
Given an open subset $\sO\subset\RR^d$ and $1\leq p\leq \infty$ and a linear, second-order, partial differential operator, $A$, as in \eqref{eq:Generator}, we shall say that a function $u \in C^2(\sO)$ (respectively, $W^{2,p}_{\loc}(\sO)$) is \emph{(strictly) $A$-subharmonic} if $Au\leq 0$ (respectively, $Au < 0$) (a.e.) on $\sO$.
\end{defn}

Recall from the Sobolev Embedding Theorem \cite[Theorem 5.6, Part I (C)]{Adams_1975} that $W^{2,d}_{\loc}(\sO) \hookrightarrow C(\sO)$; for most applications involving Definition \ref{defn:Weak_maximum_principle_property_elliptic_C2_W2d}, it would make no difference if we replaced $W^{2,d}_{\loc}(\sO)$ by $W^{2,p}_{\loc}(\sO)\cap C(\sO)$ with $p\geq 1$ but we keep $W^{2,d}_{\loc}(\sO)$ for consistency with \cite[Theorem 9.1]{GilbargTrudinger}. Partly motivated by an abstract description of weak maximum principle properties due to Trudinger \cite[p. 292]{Trudinger_1977}, we make the

\begin{defn}[Weak maximum principle property for $A$-subharmonic functions in $C^2(\sO)$ or $W^{2,d}_{\loc}(\sO)$]
\label{defn:Weak_maximum_principle_property_elliptic_C2_W2d}
Let $\sO\subset\RR^d$ be an open subset, let $\Sigma \subseteqq \partial\sO$ be an open subset, and let $\fK\subset C^2(\sO)$ (respectively, $W^{2,d}_{\loc}(\sO)$) be a convex cone. We say that an operator $A$ in \eqref{eq:Generator} obeys the \emph{weak maximum principle property on $\sO\cup\Sigma$ for $\fK$} if whenever $u\in \fK$ obeys
$$
Au \leq 0 \quad \hbox{(a.e.) on } \sO \quad\hbox{and}\quad u^* \leq 0 \quad\hbox{on } \partial\sO\less\bar\Sigma,
$$
then $u \leq 0$ on $\sO$.
\end{defn}

Examples in this article of operators, $A$, with the weak maximum principle property on $\sO\cup\Sigma$ arise as follows:
\begin{inparaenum}[\itshape a\upshape)]
\item Theorems \ref{thm:Weak_maximum_principle_C2} and \ref{thm:Weak_maximum_principle_C2_unbounded_opensubset}, where $\Sigma = \partial_0\sO$ and $\fK=\{u\in C^2_s(\underline\sO):\sup_\sO u<\infty\}$ (see Definition \ref{defn:Second_order_boundary_regularity} for a description of $C^2_s(\underline\sO)$);
\item Theorem \ref{thm:Weak_maximum_principle_C2_W2d_unbounded_function} provides a weak maximum principle property for functions $u$ in $C^2(\sO)$ or $W^{2,d}_{\loc}(\sO)$ such that $u^+$ obeys a growth condition on a possibly unbounded domain, $\sO$;
\item Theorem 3.1 and Corollary 3.2 in \cite{GilbargTrudinger}, where $\Sigma=\emptyset$ and $\fK=C^2(\sO)\cap C(\bar\sO)$;
\item Theorem 9.1  in \cite{GilbargTrudinger}, where $\Sigma=\emptyset$ and $\fK=W^{2,d}_{\loc}(\sO)\cap C(\bar\sO)$.
\end{inparaenum}
See \cite{Feehan_perturbationlocalmaxima} for results allowing $C^2_s(\underline\sO)$ to be relaxed to $C^2(\sO)\cap C^1(\underline\sO)$ and $W^{2,d}_{\loc}(\sO)$ to be replaced by $W^{2,d}_{\loc}(\sO)\cap C^1(\underline\sO)$, with $\Sigma=\partial_0\sO$ non-empty. While Trudinger considers linear subspaces in \cite[p. 292]{Trudinger_1977} rather than convex cones as in our Definition \ref{defn:Weak_maximum_principle_property_elliptic_C2_W2d}, we prefer to use convex cones since, for example, it is only the fact that a subharmonic function is bounded above (rather than bounded) which is relevant to the weak maximum principle property.

The condition \eqref{eq:c_nonnegative_domain}, namely $c\geq 0$ on $\sO$, is not explicitly required in Definition \ref{defn:Weak_maximum_principle_property_elliptic_C2_W2d}. However, simple counterexamples to the weak maximum principle exist when this condition is relaxed in general \cite[p. 33]{GilbargTrudinger}; the condition $c\geq 0$ on $\sO$ is a necessary condition for the strong maximum principle \cite[Exercise 2.1]{Pucci_Serrin_2007book}.

We shall find it very convenient to cleanly separate a discussion of when Definition \ref{defn:Weak_maximum_principle_property_elliptic_C2_W2d} holds, which is provided in Section \ref{sec:Weak_maximum_principle_elliptic_C2} for $\Sigma = \partial_0\sO$ non-empty and $\fK\subset C^2(\sO)$ --- the primary motivation for our article --- from the applications to boundary value and obstacle problems which flow from this abstract property and which are discussed in Sections \ref{sec:Applications_weak_maximum_principle_property_boundary_value_problems} and \ref{sec:Applications_weak_maximum_principle_property_obstacle_problems}.

In Section \ref{sec:Applications_weak_maximum_principle_property_boundary_value_problems}, for functions in $C^2(\sO)$ or $W^{2,d}_{\loc}(\sO)$, we consider applications of the weak maximum principle property to boundary value problems, including a comparison principle for subsolutions and supersolutions and uniqueness for solutions to the Dirichlet boundary problem (Proposition \ref{prop:Comparison_principle_elliptic_boundary_value_problem_C2_W2d}), \apriori maximum principle estimates (Proposition \ref{prop:Elliptic_weak_maximum_principle_apriori_estimates_C2_W2d}), and an extension to the case of functions which obey a growth condition on unbounded open subsets (Theorem \ref{thm:Weak_maximum_principle_C2_W2d_unbounded_function}).

Section \ref{sec:Applications_weak_maximum_principle_property_obstacle_problems}, for functions in $W^{2,d}_{\loc}(\sO)$,  contains applications of the weak maximum principle property to obstacle problems, including a comparison principle (Theorem \ref{prop:Comparison_principle_uniqueness_obstacle_problem}) and \apriori maximum principle estimates (Proposition \ref{prop:Elliptic_weak_maximum_principle_apriori_estimates_obstacle_problem}) for supersolutions and uniqueness for solutions to the obstacle problem.

Having discussed simple applications of the weak maximum principle property in the context of partial Dirichlet boundary conditions, we now turn to a discussion of when $A$ as in \eqref{eq:Generator} has the weak maximum principle property with a partial Dirichlet boundary condition if we choose $\Sigma=\partial_0\sO$. For this purpose, we shall need the

\begin{defn}[Second-order boundary condition and boundary regularity]
\label{defn:Second_order_boundary_regularity}
Let $\sO\subset\RR^d$ be an open subset and let $a:\sO\to\sS^+(d)$ be a function. We say that $u\in C^2(\sO)\cap C^1(\underline\sO)$ obeys a \emph{second-order boundary condition} along the boundary portion, $\partial_0\sO\subset\partial\sO$, defined by $a:\sO\to\sS^+(d)$ if
\begin{subequations}
\label{eq:Second_order_boundary_regularity}
\begin{align}
\label{eq:Second_order_boundary_continuity}
\tr(a D^2u) &\in C(\underline\sO),
\\
\label{eq:Second_order_boundary_vanishing}
\tr(a D^2u) &= 0 \quad\hbox{on }\partial_0\sO,
\end{align}
\end{subequations}
and write $u \in C^2_s(\underline \sO)$ if $u\in C^2(\sO)\cap C^1(\underline \sO)$ obeys \eqref{eq:Second_order_boundary_regularity}.
\end{defn}

Definition \ref{defn:Second_order_boundary_regularity} is motivated by the following observations. The property \eqref{eq:Second_order_boundary_continuity} is a mild boundary regularity condition one can impose on a function $u\in C^2(\sO)\cap C^1(\underline\sO)$ which ensures that $Au$ will be continuous up to $\partial_0\sO$ (but not necessarily up to the whole boundary, $\partial\sO$). The second-order vanishing condition \eqref{eq:Second_order_boundary_vanishing} ensures that consideration of the normal and tangential components of $Du$ along $\partial_0\sO$ permit a proof of the weak maximum principle property on $\underline\sO$ for a boundary-degenerate operator, $A$ (see Theorem \ref{thm:Weak_maximum_principle_C2} and its proof). Moreover, \eqref{eq:Second_order_boundary_vanishing} is a property of functions in weighted H\"older spaces defined by Daskalopoulos and Hamilton \cite{DaskalHamilton1998} and which provide a framework for \emph{existence} of solutions to equations such as \eqref{eq:Elliptic_equation} defined by boundary-degenerate elliptic operators, $A$, with partial Dirichlet boundary condition \eqref{eq:Elliptic_boundary_condition}, as demonstrated by \cite{Daskalopoulos_Feehan_statvarineqheston, Feehan_classical_perron_elliptic, Feehan_Pop_regularityweaksoln, Feehan_Pop_higherregularityweaksoln, Feehan_Pop_elliptichestonschauder}. See Section \ref{subsec:Boundary_degenerate_elliptic_operators_and_second_order_boundary_regularity} for a further discussion of the examples which motivate Definition \ref{defn:Second_order_boundary_regularity}.

In Section \ref{sec:Strong_maximum_principle_elliptic}, we prove a strong maximum principle for $A$-subharmonic functions in $C^2_s(\underline \sO)$ and develop its applications to boundary value problems with Neumann boundary conditions. We first prove a Hopf boundary point lemma (see Lemma \ref{lem:Degenerate_hopf_lemma}) for operators $A$ in \eqref{eq:Generator} which may become degenerate along $\partial_0\sO$ using a novel choice of barrier function. We then apply our version of the Hopf lemma to prove a strong maximum principle suitable for operators $A$ in \eqref{eq:Generator} (Theorem \ref{thm:Strong_maximum_principle}) and corresponding uniqueness results for solutions to equations with Neumann boundary conditions along $\overline{\partial_1\sO}$ (Theorem \ref{thm:Neumann_boundary_condition_uniqueness_equation} and Corollary \ref{cor:Neumann_boundary_condition_uniqueness_equation}). We use our strong maximum principle to deduce a version of the weak maximum principle, Theorem \ref{thm:Weak_maximum_principle_C2_connected_domain}, for connected open subsets.

Finally, in Section \ref{sec:Weak_maximum_principle_elliptic_C2}, we establish specific conditions on the coefficients which ensure that the operator $A$ in \eqref{eq:Generator} has the weak maximum principle property on $\sO\cup\Sigma$, when $\Sigma=\partial_0\sO$ and $\fK=\{u\in C^2_s(\underline\sO):\sup_\sO u < \infty\}$, initially for bounded $\sO$  (Theorem \ref{thm:Weak_maximum_principle_C2}), and then for unbounded $\sO$ (Theorem \ref{thm:Weak_maximum_principle_C2_unbounded_opensubset}).

\subsubsection{Weak maximum principles for bilinear maps and operators on functions in Sobolev spaces and applications to variational equations and inequalities}
We next consider variational equations and inequalities defined by bilinear maps on weighted Sobolev spaces, $H^1(\sO,\fw)$, defined by a \emph{weight}, $\fw \in C(\sO)\cap L^1(\sO)$ with $\fw > 0$ on $\sO$, and a \emph{degeneracy coefficient}, $\vartheta \in C_{\loc}(\bar\sO)$ with $\vartheta > 0$ on $\sO$, where a measurable function, $u$ on $\sO$, belongs to $H^1(\sO,\fw)$ if
$$
\int_\sO \left(\vartheta|Du|^2 + (1+\vartheta)u^2\right)\fw\,dx < \infty.
$$
Let $H^1_0(\sO\cup\Sigma,\fw)$ be the closure of $C^\infty_0(\sO\cup\Sigma)$ in $H^1(\sO,\fw)$, where $C^\infty_0(\sO\cup\Sigma)\subset C^\infty(\sO)$ is the linear subspace of smooth functions which have compact support in $\sO\cup\Sigma$. We say that $u\leq 0$ on $\partial\sO\less\Sigma$ in the sense of $H^1(\sO,\fw)$ if $u^+ \in H^1_0(\sO\cup\Sigma,\fw)$, where $u^+ = \max\{u,0\}$. We can now state the following analogue of Definition \ref{defn:Weak_maximum_principle_property_elliptic_C2_W2d}.

\begin{defn}[Weak maximum principle property for a bilinear map]
\label{defn:WeakMaxPrinciplePropertyBilinearForm}
Let $\sO\subseteqq\RR^d$ be an open subset, let
$$
\fa: H^1(\sO,\fw) \times H^1(\sO,\fw) \to \RR,
$$
be a bilinear map, let $\Sigma\subseteqq\partial\sO$ be a relatively open subset and let $\fK\subset H^1(\sO,\fw)$ be a convex cone. We say that $\fa$ obeys the \emph{weak maximum principle property on $\sO\cup\Sigma$ for $\fK$} if whenever $u\in \fK$ obeys
$$
\begin{cases}
\fa(u,v) \leq 0, \forall\, v\in H^1_0(\sO\cup\Sigma,\fw)\hbox{ with }v\geq 0\hbox{ a.e. on }\sO,
\\
u\leq 0 \hbox{ on }\partial\sO\less\Sigma\hbox{ in the sense of $H^1(\sO,\fw)$},
\end{cases}
$$
then $u \leq 0$ a.e. on $\sO$.
\end{defn}

See Section \ref{sec:ApplicationsWeakMaxPrinPropertyVarEq} for additional background and details for Definition \ref{defn:WeakMaxPrinciplePropertyBilinearForm}. (We shall assume that $\Sigma\subseteqq\partial\sO$ is relatively open for the sake of consistency with Part \ref{part:BoundaryValueObstacleProblems} of our article, although $\Sigma$ could now be any non-empty subset, not necessarily relatively open. However, the assumption of relative openness involves no loss of generality.) Typically, $\fK = H^1(\sO,\fw)$ and examples of bilinear maps, $\fa$, with the weak maximum principle property on $\sO\cup\Sigma$ are provided by Theorems \ref{thm:WeakMaximumPrincipleVarEquationCompactDomain}, \ref{thm:WeakMaximumPrincipleVarEquationBoundedDomain}, \ref{thm:WeakMaximumPrincipleVarEquationBoundedDomainGeneralWeight}, when $\Sigma\neq\emptyset$ and $\fK = H^1(\sO,\fw)$, and, when $\Sigma=\emptyset$ and $\fK = H^1(\sO)$, by \cite[Theorem 8.1]{GilbargTrudinger}. Theorem \ref{thm:WeakMaximumPrincipleH1UnboundedDomainBoundedSolution} provides an example with $\Sigma\neq\emptyset$ and $\fK = \{u\in H^1(\sO,\fw): \esssup_\sO u < \infty\}$ when $\sO$ is unbounded.

In Section \ref{sec:ApplicationsWeakMaxPrinPropertyVarEq}, we consider applications of the weak maximum principle property, including the comparison principle (Proposition \ref{prop:ComparisonUniquenessVariationalBVProblem}) and \apriori estimates (Proposition \ref{prop:H1WeakMaxPrincipleAprioriEstimates}) for $H^1(\sO,\fw)$ supersolutions and solutions to variational equations. We also show that when a bilinear map $\fa$ on $H^1(\sO,\fw)$ has a weak maximum principle property for subsolutions (on unbounded open subsets) which are bounded above, the property may extend to subsolutions which instead obey a growth condition (Theorem \ref{thm:Weak_maximum_principle_H1_unbounded_function}).

Section \ref{sec:ApplicationsWeakMaxPrinPropertyVarIneq} contains applications of the weak maximum principle property to variational inequalities. We prove uniqueness for solutions to variational inequalities defined by bilinear maps $\fa$ on $H^1(\sO,\fw)$, a comparison principle for supersolutions and uniqueness for solutions to variational inequalities (Theorem \ref{thm:WeakMaximumPrincipleVariationalInequality}), and \apriori estimates (Proposition \ref{prop:Rodrigues}) for $H^1(\sO,\fw)$ supersolutions and solutions to variational inequalities.

Finally, in Section \ref{sec:WeakMaxPrincipleSobolev}, we prove the weak maximum principle property for a class of bilinear maps $\fa:H^1(\sO,\fw)\times H^1(\sO,\fw)\to\RR$ and $\Sigma=\partial_0\sO$, when $\sO$ is bounded (Theorem \ref{thm:WeakMaximumPrincipleVarEquationBoundedDomainGeneralWeight}) or unbounded (Theorem \ref{thm:WeakMaximumPrincipleH1UnboundedDomainBoundedSolution}).

Appendix \ref{sec:WeakMaxPrincipleUnboundedFunction} provides examples illustrating when the weak maximum principle holds for functions obeying growth conditions on unbounded open subsets. In appendix \ref{sec:FicheraAndHeston}, we compare the weak maximum principles and uniqueness theorems provided by our article with those of Fichera, Ole{\u\i}nik, and Radkevi{\v{c}} \cite{Radkevich_2009a} in the case of the elliptic Heston operator, $A$, in Example \ref{exmp:HestonPDE} on an open subset $\sO\subseteqq\HH$ and show that those of Fichera, Ole{\u\i}nik, and Radkevi{\v{c}} are weaker.

\subsection{Boundary-degenerate elliptic operators and the second-order boundary regularity condition}
\label{subsec:Boundary_degenerate_elliptic_operators_and_second_order_boundary_regularity}
We provide some additional background and motivation for Definition \ref{defn:Second_order_boundary_regularity}. Clearly, if $u\in C^2(\bar\sO)$ as assumed by Amano in \cite{Amano_1979, Amano_1981}, then $Au\in C(\bar\sO)$ but, as we hinted in the introduction, when $A$ is boundary-degenerate, the work of Daskalopoulos and Hamilton \cite{DaskalHamilton1998} (where $\partial_0\sO=\partial\sO$) indicates that this is too strong a condition to initially impose on $u$ for the purpose of proving existence of solutions with Dirichlet data only prescribed along $\partial_1\sO$, whereas a condition such as $u\in C(\bar\sO)$ or $C^1(\bar\sO)$ would be too weak. See Example \ref{exmp:CIR} for a discussion in the context of the Kummer ordinary differential equation.

As we noted earlier, the second-order boundary condition \eqref{eq:Second_order_boundary_vanishing} is a property of functions in the weighted H\"older space, $C^{2+\alpha}_s(\underline\sO)$, defined by Daskalopoulos and Hamilton in \cite[pp. 901--902]{DaskalHamilton1998}. Here, $\sO$ is an open subset of the upper half-space $\HH=\{x_d>0\}\subset\RR^d$ equipped with the `cycloidal metric', $ds^2 = x_d^{-1}(dx_1^2+\cdots+dx_d^2)$, where one replaces the usual Euclidean distance function on $\HH$ in the definition of standard H\"older spaces by the cycloidal distance function defined by the metric $ds^2$. See \cite[Proposition I.12.1]{DaskalHamilton1998}, \cite[Lemma 3.1]{Feehan_Pop_mimickingdegen_pde}, and \cite[Lemma C.1]{Feehan_perturbationlocalmaxima} for further discussion.

The second-order boundary condition \eqref{eq:Second_order_boundary_vanishing} may be viewed as a special case of a generalized Ventcel boundary condition \cite[Section 7.1]{Taira_2004} defined by an auxiliary, degenerate second-order operator, $L$, which may be distinct from $A$. Around 1950, W. Feller completely characterized the analytic structure of one-dimensional diffusion processes, giving an intrinsic representation of the infinitesimal generator $\-A$ of a one-dimensional diffusion process and determined all possible boundary conditions which describe the domain, $\sD(A)$. (A parallel analysis from the point of view of Sturm-Liouville operators may be found in \cite{Zettl}.) In 1959, A. D. Ventcel studied the problem of determining all possible boundary conditions for multi-dimensional diffusion processes and a generalization of Ventcel's results to Feller processes, by Taira and many others, is described in \cite[Chapter 7]{Taira_2004}, while Amano has developed associated maximum principles in \cite{Amano_1981}, albeit by requiring $u\in C^2(\bar\sO)$ and a type of Fichera decomposition to $\partial\sO$, neither of which we require in our present article.

As noted in our Introduction, we use the term `boundary-degenerate elliptic' in this article because of the different meaning of the term `degenerate elliptic' in the context of viscosity solutions \cite{Crandall_Ishii_Lions_1992}. Consider a fully non-linear, second-order partial differential equation,
$$
F(x,u,Du,D^2u) = 0, \quad\forall\, x\in\sO.
$$
The map, $F:\sO\times\RR\times\RR^d\times \sS(d) \to \RR$, is called \emph{degenerate elliptic} (in the sense of viscosity solutions) \cite[p. 2]{Crandall_Ishii_Lions_1992} if it is increasing with respect to its matrix argument,
$$
F(x,r,\eta,X) \leq F(x,r,\eta,Y), \quad\forall\, X\leq Y\hbox{ and } (x,r,\eta) \in \sO\times\RR\times\RR^d,
$$
where $X\leq Y$ in the sense of $\sS(d)$ if $Y-X \in \sS^+(d)$. The linear operators $A$ in \eqref{eq:Generator} and equations considered in this article will always define degenerate elliptic maps, $F(x,u,Du,D^2u) = Au-f$ or $\min\{Au-f,u-\psi\}$, since the coefficient matrix $a(x)$ of $D^2u$ belongs to $\sS^+(d)$ by \eqref{eq:A_coefficients} and the coefficient of $u$ always obeys $c\geq 0$ on $\sO$ by \eqref{eq:c_nonnegative_domain}. Of course, the converse is not true: when $F$ is degenerate elliptic in the sense of  \cite[p. 2]{Crandall_Ishii_Lions_1992} --- which includes the case where $a$ is strictly elliptic, so $a\geq \lambda_0 I_d$ on $\sO$ for some positive constant $\lambda_0$ --- that does not imply that $A$ is boundary degenerate along any subset of $\partial\sO$.

If $u \in C^2_s(\underline\sO)$ and $Au\leq f$ on $\sO$, with $A$ as in \eqref{eq:Generator} and $f\in C(\underline\sO)$, then the (linear) second-order boundary condition \eqref{eq:Second_order_boundary_vanishing} is equivalent to the \emph{non-linear, first-order, oblique boundary condition},
\begin{equation}
\label{eq:EquivalentFirstOrderBC}
-\langle b,Du\rangle + cu \leq f \quad\hbox{on }\partial_0\sO.
\end{equation}
Indeed, when we have $Au=f$ on $\sO$, and thus equality in \eqref{eq:EquivalentFirstOrderBC}, one obtains the elliptic analogue of the boundary condition proposed by Heston \cite[Equation (9)]{Heston1993} for the parabolic terminal/boundary problem corresponding to the elliptic boundary value problem \eqref{eq:Elliptic_equation}, \eqref{eq:Elliptic_boundary_condition}:
\begin{equation}
\label{eq:FinancialEngineerFirstOrderBC}
-\langle b,Du\rangle + cu = f \quad\hbox{on }\partial_0\sO.
\end{equation}
The parabolic version of the condition \eqref{eq:FinancialEngineerFirstOrderBC} (normally when $f=0$) is often used in the numerical solution of parabolic boundary value or obstacle problems in mathematical finance \cite[Equation (22.19)]{DuffyFDM}, \cite[Equation (15)]{ZvanForsythVetzal}.

\subsection{Comparison with previous research}
A detailed comparison between our weak maximum principle for $A$-subharmonic functions in $C^2(\sO)$ (Theorems \ref{thm:Weak_maximum_principle_C2} and \ref{thm:Weak_maximum_principle_C2_unbounded_opensubset}) and that of Fichera \cite{Fichera_1956, Fichera_1960, Oleinik_Radkevic, Radkevich_2009a} is provided in Appendix \ref{sec:FicheraAndHeston} in the case of Example \ref{exmp:HestonPDE}. To describe one of the principal differences between the two weak maximum principles, we recall the definition of the \emph{Fichera function} \cite[Equations (1.1.2) and (1.1.3)]{Radkevich_2009a} (taking into account our sign convention in \eqref{eq:Generator} for the coefficients $a,b,c$ of the non-divergence-form operator $A$),
$$
\fb = \left(b^k - a^{kj}_{x_j}\right)n_k \quad\hbox{on }\partial\sO,
$$
where $\vec n$ denotes \emph{inward}-pointing unit normal vector field along $\partial\sO$ (now assumed, for example, to be $C^1$). By using the Fichera weak maximum principle to decide whether to impose a Dirichlet condition for $u$ along $\partial_0\sO$ to achieve uniqueness of solutions to \eqref{eq:Elliptic_equation} or \eqref{eq:Elliptic_obstacle_problem}, one finds that a Dirichlet boundary condition is required for $u$ on $\partial_0\sO$ when $\fb < 0$ on $\partial_0\sO$, in addition to the usual Dirichlet condition \eqref{eq:Elliptic_boundary_condition} on $\partial\sO\less\overline{\partial_0\sO}$, whereas only a Dirichlet boundary condition for $u$ on $\partial\sO\less\overline{\partial_0\sO}$ is required when $\fb\geq 0$ on $\partial_0\sO$.

However, by instead imposing a Dirichlet boundary condition for $u$ on $\partial\sO\less\overline{\partial_0\sO}$ and the second-order boundary regularity condition, $u\in C^2_s(\underline\sO)$, we achieve uniqueness \emph{independent} of the sign of the Fichera function, $\fb$, on $\partial_0\sO$. Since the second-order boundary condition is automatically obeyed by, for example, functions in the weighted H\"older spaces introduced by Daskalopoulos, Hamilton, and Koch \cite{DaskalHamilton1998, Koch}, we are naturally led to a more convenient and powerful framework for establishing existence, uniqueness, and regularity of solutions.

Similar remarks apply in the case of solutions to variational equations and inequalities and the framework of weighted Sobolev spaces introduced by Daskalopoulos and the author in \cite{Daskalopoulos_Feehan_statvarineqheston} and by Koch in \cite{Koch}.

Uniqueness results for solutions to the parabolic porous medium equation (and its linearization) were established by Daskalopoulos, Hamilton, Rhee, and Koch in \cite{DaskalHamilton1998, Daskalopoulos_Rhee_2003, Koch} but, unlike the linearization of the porous medium, for the coefficients $a,b$ in \eqref{eq:Generator} we permit $a(x)$ and $x\cdot b(x)$ to have quadratic growth in $x$ as $x\to\infty$ and, even when the coefficient vector field, $b$, is constant, we do not require that $b^\parallel=0$. A weak maximum principle for the parabolic (model) Kimura diffusion operator is given by Epstein and Mazzeo in \cite[Proposition 4.1.1]{Epstein_Mazzeo_annmathstudies}, who also employ a form of second-order boundary condition, together with a Hopf lemma and a strong maximum principle in \cite[Lemmas 4.2.4 and 4.2.5]{Epstein_Mazzeo_annmathstudies}. Related uniqueness results and weak maximum principles for classical (sub-)solutions to second-order, linear, degenerate elliptic and parabolic operators are proved by Pozio et al. in \cite{Pozio_Punzo_Tesei_2008, Punzo_Tesei_2009a, Punzo_Tesei_2009b}, but they do not make use of second-order boundary conditions.

A weak maximum principle for variational (sub-)solutions to second-order, nonlinear, degenerate elliptic operators is established by Bonafede \cite{Bonafede_1996}, under complex hypotheses, while a weak maximum principle for variational (sub-)solutions to \eqref{eq:Elliptic_equation} is proved by Monticelli and Payne \cite{Monticelli_Payne_2009} when the coefficient matrix, $a(x)$, of $A$ in \eqref{eq:Generator} has a uniformly elliptic direction and $\sO$ is bounded (conditions, among others, which we do not impose in this article). Borsuk obtains a weak maximum principle, a Hopf lemma, and a strong maximum principle \cite[Theorem 6.2.1, Lemma 6.2.1 and Theorem 6.2.2]{Borsuk_2005} for (sub-)solutions to variational equations defined by certain degenerate elliptic quasilinear operators on open subsets with non-smooth boundaries using weighted Sobolev spaces. Our maximum principle for variational (sub-)solutions to \eqref{eq:Elliptic_equation}, \eqref{eq:Elliptic_boundary_condition} (Theorems \ref{thm:WeakMaximumPrincipleVarEquationBoundedDomainGeneralWeight} and \ref{thm:WeakMaximumPrincipleH1UnboundedDomainBoundedSolution}) generalizes that of Trudinger \cite[Theorem 1]{Trudinger_1977}.

\subsection{Notation and conventions}
\label{subsec:Notation}
For $x, y\in\RR$, we denote $x\wedge y := \min\{x,y\}$ and $x\vee y := \max\{x,y\}$, while $x^+ := x\vee 0$ and $x^- := -(x\wedge 0)$. We let $B(x^0,r)\subset\RR^d$ denote the open ball with radius $r>0$ and center $x^0\in\RR^d$.

If $X$ is a subset of a topological space, we let $\bar X$ denote its closure and let $\partial X := \bar X\less X$ denote its topological boundary. If $V\subset U\subset \RR^d$ are open subsets, we write $V\Subset U$ when $U$ is bounded with closure $\bar U \subset V$.

For an open subset of a topological space, $U\subset X$, we let $u^*:\bar U\to[-\infty,\infty]$ (respectively, $u_*:\bar U\to[-\infty,\infty]$) denote the upper (respectively, lower) semicontinuous envelope of a function $u:U\to[-\infty,\infty]$; when $u$ is continuous on $U$, then $u_* = u = u^*$ on $U$.

In the definition and naming of function spaces, we follow Adams \cite{Adams_1975} and alert the reader to occasional differences in definitions between \cite{Adams_1975} and standard references such as Gilbarg and Trudinger \cite{GilbargTrudinger}. Since $\sO\subseteqq\RR^d$ may often denote an unbounded open subset in this article, we distinguish between
\begin{inparaenum}[\itshape a\upshape)]
\item $C_{\loc}(\bar\sO)$, the vector space of functions, $u$, such that, for any precompact open subset $U\Subset \bar\sO$, we have $u \in C(\bar U)$; and
\item $C(\bar\sO)$, the Banach space of functions which are \emph{uniformly continuous} and \emph{bounded} on $\sO$.
\end{inparaenum}

For an integer $k\geq 0$, we let $C^k(\sO)$ denote the vector space of functions whose derivatives up to order $k$ are continuous on $\sO$ and let $C^k(\bar\sO)$ denote the Banach space of functions whose derivatives up to order $k$ are uniformly continuous and bounded on $\sO$, and thus have unique bounded, continuous extensions to $\bar \sO$ \cite[Sections 1.25 and 1.26]{Adams_1975}.

If $T \subseteqq \partial\sO$ is a relatively open set, we let $C^k_{\loc}(\sO\cup T)$ denote the vector space of functions, $u$, such that, for any precompact open subset $U\Subset \sO \cup T$, we have $u \in C^k(\bar U)$. We adopt the convention that $C(\sO\cup T) = C_{\loc}(\bar\sO)$ even when $T=\partial\sO$.

\subsection{Acknowledgments} I am grateful to Panagiota Daskalopoulos and Camelia Pop for many engaging discussions on degenerate partial differential equations. I am especially grateful to the anonymous referee for many helpful suggestions and comments. I was very saddened to learn of the recent death of Peter Laurence, with whom I had many conversations on mathematics closely related to this article and who had alerted me to several important references. I will always be grateful for his kindness to me and for generously sharing his ideas, insights, and questions.

\part{Weak and strong maximum principles for operators on smooth functions and applications to boundary value and obstacle problems}
\label{part:BoundaryValueObstacleProblems}
In this part of our article (sections \ref{sec:Applications_weak_maximum_principle_property_boundary_value_problems}, \ref{sec:Applications_weak_maximum_principle_property_obstacle_problems}, \ref{sec:Strong_maximum_principle_elliptic}, and \ref{sec:Weak_maximum_principle_elliptic_C2}), we develop weak and strong maximum principles for operators on smooth functions and their applications to boundary value and obstacle problems.

\section{Applications of the weak maximum principle property to boundary value problems}
\label{sec:Applications_weak_maximum_principle_property_boundary_value_problems}
We shall encounter many different situations (for example, operators on bounded or unbounded open subsets, bounded functions or unbounded functions with prescribed growth, and so on) where a basic maximum principle holds for a linear, second-order, partial differential operator acting on a convex cone of functions in $C^2(\sO)$ or $W^{2,d}_{\loc}(\sO)$. In order to unify our treatment of applications, we find it useful to isolate a key `weak maximum principle property' (Definition \ref{defn:Weak_maximum_principle_property_elliptic_C2_W2d}) and then derive the consequences which necessarily follow in an essentially formal manner. In this section, we consider applications to boundary value problems. In Section \ref{subsec:Applications_weak_maximum_principle_property_boundary_value_problems}, for functions in $C^2(\sO)$ or $W^{2,d}_{\loc}(\sO)$, we establish a comparison principle for subsolutions and supersolutions and uniqueness for solutions to the Dirichlet boundary problem (Proposition \ref{prop:Comparison_principle_elliptic_boundary_value_problem_C2_W2d}) and \apriori weak maximum principle estimates (Proposition \ref{prop:Elliptic_weak_maximum_principle_apriori_estimates_C2_W2d}). In Section \ref{subsec:Weak_maximum_principle_C2_W2d_unbounded_function}, we show that when an operator has the weak maximum principle property for functions which are bounded above, the property may also hold for unbounded functions which obey a growth condition (Theorem \ref{thm:Weak_maximum_principle_C2_W2d_unbounded_function}).

\subsection{Applications of the weak maximum principle property to boundary value problems}
\label{subsec:Applications_weak_maximum_principle_property_boundary_value_problems}
The weak maximum principle property (Definition \ref{defn:Weak_maximum_principle_property_elliptic_C2_W2d}) immediately yields a comparison principle and thus uniqueness for solutions to the equation \eqref{eq:Elliptic_equation} with partial Dirichlet condition \eqref{eq:Elliptic_boundary_condition}.

\begin{prop}[Comparison principle and uniqueness for solutions to the Dirichlet boundary value problem]
\label{prop:Comparison_principle_elliptic_boundary_value_problem_C2_W2d}
Let $\sO\subset\RR^d$ be an open subset and $A$ in \eqref{eq:Generator} have the weak maximum principle property on $\sO\cup\Sigma$ in the sense of Definition \ref{defn:Weak_maximum_principle_property_elliptic_C2_W2d}, for a convex cone $\fK\subset C^2(\sO)$ (respectively, $W^{2,d}_{\loc}(\sO)$) and open subset $\Sigma\subseteqq\partial\sO$. Suppose that $u, -v\in \fK$. If $Au \leq Av$ (a.e.) on $\sO$ and $u^* \leq v_*$ on $\partial\sO\less\bar\Sigma$, then $u \leq v$ on $\sO$. If $Au = Av$ (a.e.) on $\sO$ and $u^* = v_*$ on $\partial\sO\less\bar\Sigma$, then $u = v$ on $\sO$ and $u=v\in C(\sO\cup \partial\sO\less\bar\Sigma)$.
\end{prop}

\begin{proof}
Since $u$ is a subsolution and $v$ a supersolution, then $u-v$ is a subsolution with $u^*-v_*\leq 0$ on $\partial\sO\less\bar\Sigma$ and thus
$$
u-v \leq 0 \quad\hbox{on }\sO,
$$
because $A$ has  weak maximum principle property on $\sO\cup\Sigma$. When $u, v$ are both solutions, then we also obtain $v-u \leq 0$ on $\sO$ and so $v = u$ on $\sO$. Since $u^*=u_*$ on $\partial\sO\less\bar\Sigma$, we must have that $u$ is continuous on $\partial\sO\less\bar\Sigma$ and thus on $\sO\cup \partial\sO\less\bar\Sigma$.
\end{proof}

We can now proceed to give the expected \apriori estimates.

\begin{prop}[\emph{A priori} weak maximum principle estimates for functions in $C^2(\sO)$ or $W^{2,d}_{\loc}(\sO)$]
\label{prop:Elliptic_weak_maximum_principle_apriori_estimates_C2_W2d}
Let $\sO\subset\RR^d$ be an open subset and $A$ in \eqref{eq:Generator} have the weak maximum principle property on $\sO\cup\Sigma$ in the sense of Definition \ref{defn:Weak_maximum_principle_property_elliptic_C2_W2d}, for a convex cone $\fK\subset C^2(\sO)$ (respectively, $W^{2,d}_{\loc}(\sO)$) containing the constant function $1$ and open subset $\Sigma\subseteqq\partial\sO$. Suppose that $u, -v\in \fK$.
\begin{enumerate}
\item\label{item:Subsolution_Au_leq_zero} If $c \geq 0$ on $\sO$ and $Au\leq 0$ on $\sO$, then
$$
u\leq 0 \vee \sup_{\partial\sO\less\Sigma}u^* \quad\hbox{on } \sO.
$$
\item\label{item:Subsolution_Au_arb_sign} If $Au$ has arbitrary sign and there is a constant $c_0>0$ such that
$c\geq c_0$ on $\sO$, then
$$
u\leq 0 \vee \frac{1}{c_0}\sup_\sO Au \vee \sup_{\partial\sO\less\Sigma}u^* \quad\hbox{on } \sO.
$$
\item\label{item:Supersolution_Au_geq_zero} If $c \geq 0$ on $\sO$ and $Av\geq 0$ on $\sO$, then
$$
v\geq 0 \wedge \inf_{\partial\sO\less\Sigma}v_* \quad\hbox{on } \sO.
$$
\item\label{item:Supersolution_Au_arb_sign} If $Av$ has arbitrary sign and $c\geq c_0$ on $\sO$ for a positive constant $c_0$, then
$$
v\geq 0 \wedge \frac{1}{c_0}\inf_\sO Av \wedge \inf_{\partial\sO\less\Sigma}v_* \quad\hbox{on } \sO.
$$
\item\label{item:Solution_Au_zero} If $c \geq 0$ on $\sO$ and $Au=0$ on $\sO$ and $u \in C(\sO\cup\partial\sO\less\bar\Sigma)$ and $u \in \fK\cap -\fK$, then
$$
\|u\|_{C(\bar\sO)} \leq \|u\|_{C(\partial\sO\less\Sigma)}.
$$
\item\label{item:Solution_Au_arb_sign} If $Au$ has arbitrary sign and $c\geq c_0$ on $\sO$ for a positive constant $c_0$ and $u \in C(\sO\cup\partial\sO\less\bar\Sigma)$ and $u \in \fK\cap -\fK$, then
$$
\|u\|_{C(\bar\sO)} \leq \frac{1}{c_0}\|Au\|_{C(\bar \sO)}\vee\|u\|_{C(\partial\sO\less\Sigma)}.
$$
\end{enumerate}
The terms $\sup_{\partial\sO\less\Sigma} u^*$, and $\inf_{\partial\sO\less\Sigma}v_*$, and $\|u\|_{C(\partial\sO\less\Sigma)}$ in the preceding items are omitted when $\Sigma=\partial\sO$. When $\fK\subset W^{2,d}_{\loc}(\sO)$, then inequalities involving $c$ and $Au$ or $Av$ may hold a.e. on $\sO$ and we write $\esssup_\sO Au$ and $\essinf_\sO Av$ and $\|Au\|_{L^\infty(\sO)}$ in place of $\sup_\sO Au$ and $\inf_\sO Av$ and $\|Au\|_{C(\bar \sO)}$.
\end{prop}

The \apriori estimate in Item \eqref{item:Solution_Au_arb_sign} may be compared with \cite[Theorem 1.1.2]{Radkevich_2009a} (in the case of $C^2$ functions) and \cite[Theorem 1.5.1 and 1.5.5]{Radkevich_2009a} and \cite[Lemma 2.8]{Troianiello} (in the case of $H^1$ functions). However, because the coefficient matrix $a$ of $A$ in \eqref{eq:Generator} is zero along $\Sigma\subseteqq\partial\sO$ in applications considered in this article, there is no analogue of \cite[Theorem 3.7 and Corollary 3.8]{GilbargTrudinger}.

\begin{proof}[Proof of Proposition \ref{prop:Elliptic_weak_maximum_principle_apriori_estimates_C2_W2d}]
For Items \eqref{item:Subsolution_Au_leq_zero} and \eqref{item:Subsolution_Au_arb_sign}, we describe the proof when $\Sigma\subsetneqq\partial\sO$; the proof for the case $\Sigma=\partial\sO$ is the same except that the suprema on the right-hand side are replaced by zero. When $Au$ has arbitrary sign, choose
$$
M :=  0 \vee \frac{1}{c_0}\sup_\sO Au \vee \sup_{\partial\sO\less\Sigma}u^*,
$$
while if $Au\leq 0$, choose
$$
M :=  0 \vee \sup_{\partial\sO\less\Sigma}u^*.
$$
We may assume without loss of generality that $M<\infty$. We have $M\geq u^*$ on $\partial\sO\less\bar\Sigma$ and
$$
AM = cM \geq c_0M \geq Au \quad\hbox{on }\sO \quad\hbox{(by \eqref{eq:c_positive_lower_bound_domain})},
$$
when $Au$ has arbitrary sign and, when $Au\leq 0$, we have
$$
AM = cM \geq 0 \geq Au \quad\hbox{on }\sO \quad\hbox{(by \eqref{eq:c_nonnegative_domain})}.
$$
Thus, $u\leq M$ on $\sO$ by Proposition \ref{prop:Comparison_principle_elliptic_boundary_value_problem_C2_W2d}, which gives Items \eqref{item:Subsolution_Au_leq_zero} and \eqref{item:Subsolution_Au_arb_sign}. Items \eqref{item:Supersolution_Au_geq_zero} and \eqref{item:Supersolution_Au_arb_sign} follow from Items \eqref{item:Subsolution_Au_leq_zero} and \eqref{item:Subsolution_Au_arb_sign} by choosing $u = -v$. Item \eqref{item:Solution_Au_zero} follows by combining Items \eqref{item:Subsolution_Au_leq_zero} and \eqref{item:Supersolution_Au_geq_zero}, while Item \eqref{item:Solution_Au_arb_sign} follows by combining Items \eqref{item:Subsolution_Au_arb_sign} and \eqref{item:Supersolution_Au_arb_sign}.
\end{proof}

\subsection{Applications of the weak maximum principle property to unbounded subharmonic functions}
\label{subsec:Weak_maximum_principle_C2_W2d_unbounded_function}
If an operator has the weak maximum principle property for subsolutions which are bounded above, we obtain an extension for subsolutions which instead obey a growth condition.

\begin{thm}[Weak maximum principle for unbounded $A$-subharmonic functions in $C^2(\sO)$ or $W^{2,d}_{\loc}(\sO)$ on unbounded open subsets]
\label{thm:Weak_maximum_principle_C2_W2d_unbounded_function}
Let $\sO\subseteqq\RR^d$ be a possibly unbounded open subset and $\varphi\in C^2(\sO)$ obey $0<\varphi\leq 1$ on $\sO$. Let $A$ be an operator as in \eqref{eq:Generator} and
\begin{equation}
\label{eq:First_order_operator}
Bv := -[A, \varphi](\varphi^{-1}v), \quad\forall\, v\in C^2(\sO),
\end{equation}
and suppose that the differential operator,
\begin{equation}
\label{eq:Defn_hatA_operator}
\widehat A := (A+B)v, \quad\forall\, v\in C^2(\sO),
\end{equation}
has the weak maximum principle property on $\sO\cup\Sigma$ in the sense of Definition \ref{defn:Weak_maximum_principle_property_elliptic_C2_W2d}, for a convex cone $\fK\subset C^2(\sO)$ (respectively, $W^{2,d}_{\loc}(\sO)$) and open subset $\Sigma\subseteqq\partial\sO$, for functions $u\in\fK$ which are \emph{bounded above}, so $\sup_\sO u<\infty$. Then $A$ has the weak maximum principle property on $\sO\cup\Sigma$ for functions $u\in\fK$ which obey the growth condition,
\begin{equation}
\label{eq:Upper_bound_subsolution}
u \leq C\left(1+\varphi^{-1}\right) \quad\hbox{on }\sO.
\end{equation}
\end{thm}

\begin{proof}
Suppose $u\in \fK$ obeys \eqref{eq:Upper_bound_subsolution} and $Au\leq 0$ (a.e.) on $\sO$. Clearly, we have
\begin{align*}
\widehat A(\varphi u) &= A\varphi u + B\varphi u \quad\hbox{(by \eqref{eq:Defn_hatA_operator})}
\\
&= \varphi Au + [A, \varphi]u + B\varphi u
\\
&= \varphi Au \quad\hbox{(by \eqref{eq:First_order_operator})}
\\
&\leq 0 \quad\hbox{(a.e.) on }\sO.
\end{align*}
Since $\varphi u\leq C(\varphi+1) \leq 2C$ on $\sO$ by \eqref{eq:Upper_bound_subsolution}, then $\varphi u \leq 0$ on $\sO$ since $\widehat A$ has the weak maximum principle property on $\sO\cup\Sigma$ for functions $u$ on $\sO$ which are bounded above. Thus, $u \leq 0$ on $\sO$.
\end{proof}

\section{Applications of the weak maximum principle property to obstacle problems}
\label{sec:Applications_weak_maximum_principle_property_obstacle_problems}
We now turn to the application of the weak maximum principle property to obstacle problems. The application is complicated by the fact that the optimal interior regularity of solutions to obstacle problems is $C^{1,1}(\sO)$, rather than $C^2(\sO)$, and indeed classical existence theory (see, for example, \cite[Theorem 1.3.2]{Friedman_1982}) initially only yields solutions in $W^{2,p}(\sO)$, for $1< p<\infty$. For this reason, we consider functions in $W^{2,d}_{\loc}(\sO)$ and prove a comparison principle for supersolutions and uniqueness for solutions to the obstacle problem (Proposition \ref{prop:Comparison_principle_uniqueness_obstacle_problem}) and then derive \apriori maximum principle estimates for supersolutions and solutions to the obstacle problem (Proposition \ref{prop:Elliptic_weak_maximum_principle_apriori_estimates_obstacle_problem}). We begin with the

\begin{defn}[Solution and supersolution to an obstacle problem]
\label{defn:Solution_supersolution_obstacle_problem}
Let $\sO\subset\RR^d$ be an open subset, $p\geq 1$, and $A$ as in \eqref{eq:Generator}. Given $f\in L^p_{\loc}(\sO)$ and $\psi\in L^p_{\loc}(\sO)$, we call $u\in W^{2,p}_{\loc}(\sO)$ a \emph{solution} (respectively, \emph{supersolution}) to the obstacle problem \eqref{eq:Elliptic_obstacle_problem} if
$$
\min\{Au-f,u-\psi\} = 0\ (\geq 0)  \quad\hbox{a.e. on }\sO.
$$
Furthermore, given $g \in C(\partial\sO\less\bar\Sigma)$ and $\psi$ also belonging to $C(\partial\sO\less\bar\Sigma)$ and obeying the compatibility condition \eqref{eq:Elliptic_obstacle_boundarydata_compatibility}, that is, $\psi\leq g$ on $\partial\sO\less\bar\Sigma$, we call $u$ a \emph{solution} to the obstacle problem with partial Dirichlet condition if in addition $u$ belongs to $C(\partial\sO\less\bar\Sigma)$ and is a \emph{solution} (respectively, \emph{supersolution}) to \eqref{eq:Elliptic_boundary_condition}, so
$$
u = g\ (\geq g)  \quad\hbox{a.e. on } \partial\sO\less\bar\Sigma.
$$
\end{defn}

Proposition \ref{prop:Comparison_principle_uniqueness_obstacle_problem} below for solutions to the obstacle problem is an analogue of Theorem \ref{thm:WeakMaximumPrincipleVariationalInequality}, which applies to $H^1(\sO,\fw)$ solutions to a variational inequality, and Theorem \ref{thm:WeakMaximumPrincipleH2ObstacleProblem}, which applies to $H^2(\sO,\fw)$, that is, strong solutions to the obstacle problem. We may also compare Propositions \ref{prop:Comparison_principle_uniqueness_obstacle_problem} and \ref{prop:Elliptic_weak_maximum_principle_apriori_estimates_obstacle_problem} with \cite[Theorems 4.5.1, 4.6.1, 4.6.6, and 4.7.4, and Corollary 4.5.2]{Rodrigues_1987} for the case of variational inequalities.

\begin{prop}[Comparison principle for $W^{2,d}_{\loc}$ supersolutions and uniqueness for $W^{2,d}_{\loc}$ solutions to the obstacle problem]
\label{prop:Comparison_principle_uniqueness_obstacle_problem}
Let $\sO\subset\RR^d$ be an open subset, $\fK\subset W^{2,d}_{\loc}(\sO)$ a convex cone
and $\Sigma\subseteqq\partial\sO$ an open subset. For every open subset $\sU\subset\sO$, let $A$ in \eqref{eq:Generator} have the weak maximum principle property on $\sU\cup\Sigma$ in the sense of Definition \ref{defn:Weak_maximum_principle_property_elliptic_C2_W2d} for the cone $\fK$. Let $f\in L^d_{\loc}(\sO)$ and $\psi\in L^d_{\loc}(\sO)$. Suppose $u\in \fK$ (respectively, $v\in -\fK$) is a solution (respectively, supersolution) to the obstacle problem,
$$
\min\{Au-f, \ u-\psi\} = 0 \ (\geq  0) \quad\hbox{a.e. on } \sO.
$$
If $v_*\geq u^*$ on $\partial\sO\less\bar\Sigma$, then $v \geq u$ on $\sO$; if $u, v$ are solutions and $v_* = u^*$ on $\partial\sO\less\bar\Sigma$, then $u = v$ on $\sO$.
\end{prop}

Note that the weak maximum principle property hypothesis on $A$ in Propositions \ref{prop:Comparison_principle_uniqueness_obstacle_problem} and \ref{prop:Elliptic_weak_maximum_principle_apriori_estimates_obstacle_problem} is stronger than that in Propositions \ref{prop:Comparison_principle_elliptic_boundary_value_problem_C2_W2d} and \ref{prop:Elliptic_weak_maximum_principle_apriori_estimates_C2_W2d}.

\begin{proof}[Proof of Proposition \ref{prop:Comparison_principle_uniqueness_obstacle_problem}]
The proof is similar to the argument for the argument for the parabolic version \cite[Proposition 3.2]{Feehan_parabolicmaximumprinciple} of Proposition \ref{prop:Comparison_principle_uniqueness_obstacle_problem}, but we include the details for completeness. Suppose $\sU := \sO\cap\{u>v\}$ is non-empty. We have
$$
\partial\sU = \left(\sO\cap\partial\{u>v\}\right)\cup\left(\{u>v\}\cap \partial\sO\right)\cup\left(\partial\{u>v\}\cap \partial\sO\right),
$$
and hence we see that
$$
\partial\sU\less\bar\Sigma = \left(\sO\cap\partial\{u>v\}\less\bar\Sigma\right)\cup\left(\{u>v\}\cap \partial\sO\less\bar\Sigma\right)\cup\left(\partial\{u>v\}\less\bar\Sigma\cap \partial\sO\less\bar\Sigma\right).
$$
Because $u\leq v$ on $\partial\sO\less\bar\Sigma$ (when non-empty) and $u=v$ on $\partial\sU$, so $u=v$ on $(\sO\cup\partial\sO\less\bar\Sigma)\cap\partial\{u>v\}\less\bar\Sigma$, we must have
$$
u-v \leq 0 \quad\hbox{on } \partial\sU\less\bar\Sigma.
$$
We have $u-v\in \fK$ and $A(u-v)\leq 0$ a.e on $\sU$ by hypothesis, so $u-v\leq 0$ on $\sU$ since $A$ has the weak maximum principle property on $\sU\cup(\Sigma\cap \partial\sU)$ for $\fK\cap W^{2,d}_{\loc}(\sU)$ in the sense of Definition \ref{defn:Weak_maximum_principle_property_elliptic_C2_W2d}, contradicting our assertion that $\sU$ is non-empty. Hence, $u\leq v$ on $\sO$.

If both $u$ and $v$ are solutions to the obstacle problem then, since any solution is also a supersolution by Definition \ref{defn:Solution_supersolution_obstacle_problem}, we may reverse the roles of $u$ and $v$ in the preceding argument to give $v\leq u$ on $\sO$ and thus $u=v$ on $\sO$.
\end{proof}

We then have the

\begin{prop}[Weak maximum principle and \apriori estimates for supersolutions and solutions to obstacle problems]
\label{prop:Elliptic_weak_maximum_principle_apriori_estimates_obstacle_problem}
Let $\sO\subset\RR^{d}$ be an open subset, $\fK\subset W^{2,d}_{\loc}(\sO)$ be a convex cone containing the constant function $1$, and
$\Sigma\subseteqq\partial\sO$ be an open subset. For every open subset $\sU\subset\sO$, let $A$ in \eqref{eq:Generator} have the weak maximum principle property on $\sU\cup\Sigma$ in the sense of
Proposition \ref{prop:Comparison_principle_uniqueness_obstacle_problem}. Assume that $c \geq 0$ a.e. on $\sO$. Let $f\in L^{d}_{\loc}(\sO)$, and $g\in C(\partial\sO\less\bar\Sigma)$, and $\psi\in C(\sO\cup\partial\sO\less\bar\Sigma)$ with $\psi\leq g$ on $\partial\sO\less\bar\Sigma$.
Suppose $u\in \fK\cap -\fK$ is a solution and $v\in -\fK$ is a supersolution to the obstacle problem in the sense of Definition \ref{defn:Solution_supersolution_obstacle_problem} for $f$ and $g$ and $\psi$.
\begin{enumerate}
\item \label{item:Obstacle_supersolution_f_geq_zero} If $f\geq 0$ a.e. on $\sO$, then
$$
v\geq 0 \wedge \inf_{\partial\sO\less\Sigma}g \quad\hbox{on }\sO.
$$
\item\label{item:Obstacle_supersolution_f_arb_sign} If $f$ has arbitrary sign but there is a constant $c_0>0$ such that $c\geq c_0$ a.e. on $\sO$, then
$$
v\geq 0 \wedge \frac{1}{c_0}\essinf_\sO f \wedge \inf_{\partial\sO\less\Sigma}g \quad\hbox{on }\sO.
$$
\item \label{item:Obstacle_solution_f_leq_zero} If $f\leq 0$ a.e on $\sO$, then
$$
u\leq 0 \vee \sup_{\partial\sO\less\Sigma}g \vee \sup_\sO\psi \quad\hbox{on }\sO.
$$
\item\label{item:Obstacle_solution_f_arb_sign} If $f$ has arbitrary sign but $c\geq c_0$ a.e. on $\sO$, then
$$
u\leq 0 \vee \frac{1}{c_0}\esssup_\sO f \vee \sup_{\partial\sO\less\Sigma}g \vee \sup_\sO\psi \quad\hbox{on }\sO.
$$
\item\label{item:Obstacle_comparison} If $u_1$ and $u_2$ are solutions, respectively, for $f_1\geq f_2$ a.e. on $\sO$ and $\psi_1\geq \psi_2$ on $\sO$, and $g_1\geq g_2$ on $\partial\sO\less\bar\Sigma$, then
$$
u_1\geq u_2 \quad\hbox{on }\sO.
$$
\item\label{item:Obstacle_stability} If $u_i$ is a solution for $f_i,\psi_i$ on $\sO$ and $g_i$ on $\partial\sO\less\bar\Sigma$ with $\psi_i\leq g_i$ on $\partial\sO\less\bar\Sigma$ for $i=1,2$, and $c\geq c_0$ a.e. on $\sO$, then
$$
\|u_1-u_2\|_{C(\bar \sO)} \leq \frac{1}{c_0}\|f_1-f_2\|_{L^\infty(\sO)} \vee \|g_1-g_2\|_{C(\partial\sO\less\Sigma)} \vee \|\psi_1-\psi_2\|_{C(\bar \sO)},
$$
and if $f_1=f_2$ and $c\geq 0$ a.e. on $\sO$, then
$$
\|u_1-u_2\|_{C(\bar \sO)} \leq \|g_1-g_2\|_{C(\partial\sO\less\Sigma)} \vee \|\psi_1-\psi_2\|_{C(\bar \sO)}.
$$
\end{enumerate}
The terms $\sup_{\partial\sO\less\Sigma} g$, and $\inf_{\partial\sO\less\Sigma}g$, and $\|g_1-g_2\|_{C(\partial\sO\less\Sigma)}$ in the preceding items are omitted when $\Sigma=\partial\sO$.
\end{prop}

\begin{proof}
Consider Items \eqref{item:Obstacle_supersolution_f_geq_zero} and \eqref{item:Obstacle_supersolution_f_arb_sign}. Since $u$ is a supersolution to the obstacle problem \eqref{eq:Elliptic_obstacle_problem}, then it is also a supersolution to the boundary value problem \eqref{eq:Elliptic_equation} (where $\psi$ plays no role) and so Items \eqref{item:Obstacle_supersolution_f_geq_zero} and \eqref{item:Obstacle_supersolution_f_arb_sign} here just restate Items \eqref{item:Supersolution_Au_geq_zero} and \eqref{item:Supersolution_Au_arb_sign} in Proposition \ref{prop:Elliptic_weak_maximum_principle_apriori_estimates_C2_W2d}.

Consider Items \eqref{item:Obstacle_solution_f_leq_zero} and \eqref{item:Obstacle_solution_f_arb_sign} here. When $f\leq 0$ and $c\geq 0$ a.e. on $\sO$, let
$$
M := 0\vee \sup_{\partial\sO\less\Sigma}g \vee\sup_\sO\psi,
$$
while if $f$ has arbitrary sign and $c\geq c_0$ a.e. on $\sO$, let
$$
M := 0\vee \frac{1}{c_0}\esssup_\sO f \vee \sup_{\partial\sO\less\Sigma}g \vee\sup_\sO\psi.
$$
We may assume without loss of generality that $M<\infty$. Then $M\geq \psi$ on $\sO$ and $M\geq g$ on $\partial\sO\less\Sigma$, while
$$
AM = cM \geq 0 \geq f \quad\hbox{a.e. on }\sO,
$$
when $f\leq 0$ a.e. on $\sO$ and
$$
AM = cM \geq c_0M \geq f \quad\hbox{a.e. on }\sO,
$$
when $f$ has arbitrary sign. Hence, $M$ is a supersolution and so Proposition \ref{prop:Comparison_principle_uniqueness_obstacle_problem} implies that $u\leq M$ on $\sO$, which establishes Items \eqref{item:Obstacle_solution_f_leq_zero} and \eqref{item:Obstacle_solution_f_arb_sign}. For Item \eqref{item:Obstacle_comparison}, observe that $u_1$ is a supersolution for the obstacle problem in Definition \ref{defn:Solution_supersolution_obstacle_problem} given by $f_2,g_2,\psi_2$ and thus $u_1\geq u_2$ on $\sO$ by Proposition \ref{prop:Comparison_principle_uniqueness_obstacle_problem}. For Item \eqref{item:Obstacle_stability}, define
$$
m := \frac{1}{c_0}\|f_1-f_2\|_{L^\infty(\sO)} \vee \|g_1-g_2\|_{C(\partial\sO\less\Sigma)} \vee \|\psi_1-\psi_2\|_{C(\bar\sO)} \quad\hbox{and}\quad u := u_2 + m.
$$
Then
$$
Au = Au_2 + Am \geq f_2 + cm \geq f_2 + \esssup_\sO(f_1 - f_2) \geq f_1 \quad\hbox{a.e. on }\sO,
$$
while
$$
u \geq \psi_2 + \sup_\sO(\psi_1 - \psi_2) \geq \psi_1 \quad\hbox{on }\sO,
$$
and
$$
u \geq g_2 + \sup_{\partial\sO\less\Sigma}(g_1 - g_2) \geq g_1 \quad\hbox{on }\partial\sO\less\Sigma.
$$
Therefore, $u$ is a supersolution for $f_1, g_1, \psi_1$ and so Proposition \ref{prop:Comparison_principle_uniqueness_obstacle_problem} implies that $u\geq u_1$ on $\sO$, and thus
$$
u_1-u_2 \leq m \quad\hbox{on }\sO.
$$
By interchanging the roles of $u_1, u_2$ in the preceding argument, the conclusion follows for the case $c\geq c_0>0$. For the case $c\geq 0$, we now define
$$
m := \|\psi_1-\psi_2\|_{C(\bar\sO)} \vee \|g_1-g_2\|_{C(\partial\sO\less\Sigma)} \quad\hbox{and}\quad u := u_2 + m,
$$
so that
$$
Au = Au_2 + Am \geq f + cm \geq f \quad\hbox{a.e. on }\sO,
$$
and the remainder of the argument is identical.
\end{proof}

\section{Strong maximum principle and applications to boundary value problems}
\label{sec:Strong_maximum_principle_elliptic}
The usual statements of the Hopf boundary point lemma \cite[Lemma 3.4]{GilbargTrudinger} require that $A$ in \eqref{eq:Generator} be strictly and uniformly elliptic, but a more careful analysis shows that it holds under much weaker hypotheses. We exploit our version of the Hopf lemma (see Lemma \ref{lem:Degenerate_hopf_lemma}) to prove a strong maximum principle suitable for boundary-degenerate elliptic operators (Theorem \ref{thm:Strong_maximum_principle}) and corresponding uniqueness results for solutions to equations with partial Neumann boundary conditions (Theorem \ref{thm:Neumann_boundary_condition_uniqueness_equation} and Corollary \ref{cor:Neumann_boundary_condition_uniqueness_equation}). Finally, we use our strong maximum principle to prove a version, Theorem \ref{thm:Weak_maximum_principle_C2_connected_domain}, of the weak maximum principle which complements our alternative version, Theorem \ref{thm:Weak_maximum_principle_C2}.

\begin{figure}
 \centering
 \begin{picture}(460,180)(0,0)
 \put(20,0){\includegraphics[scale=0.5]{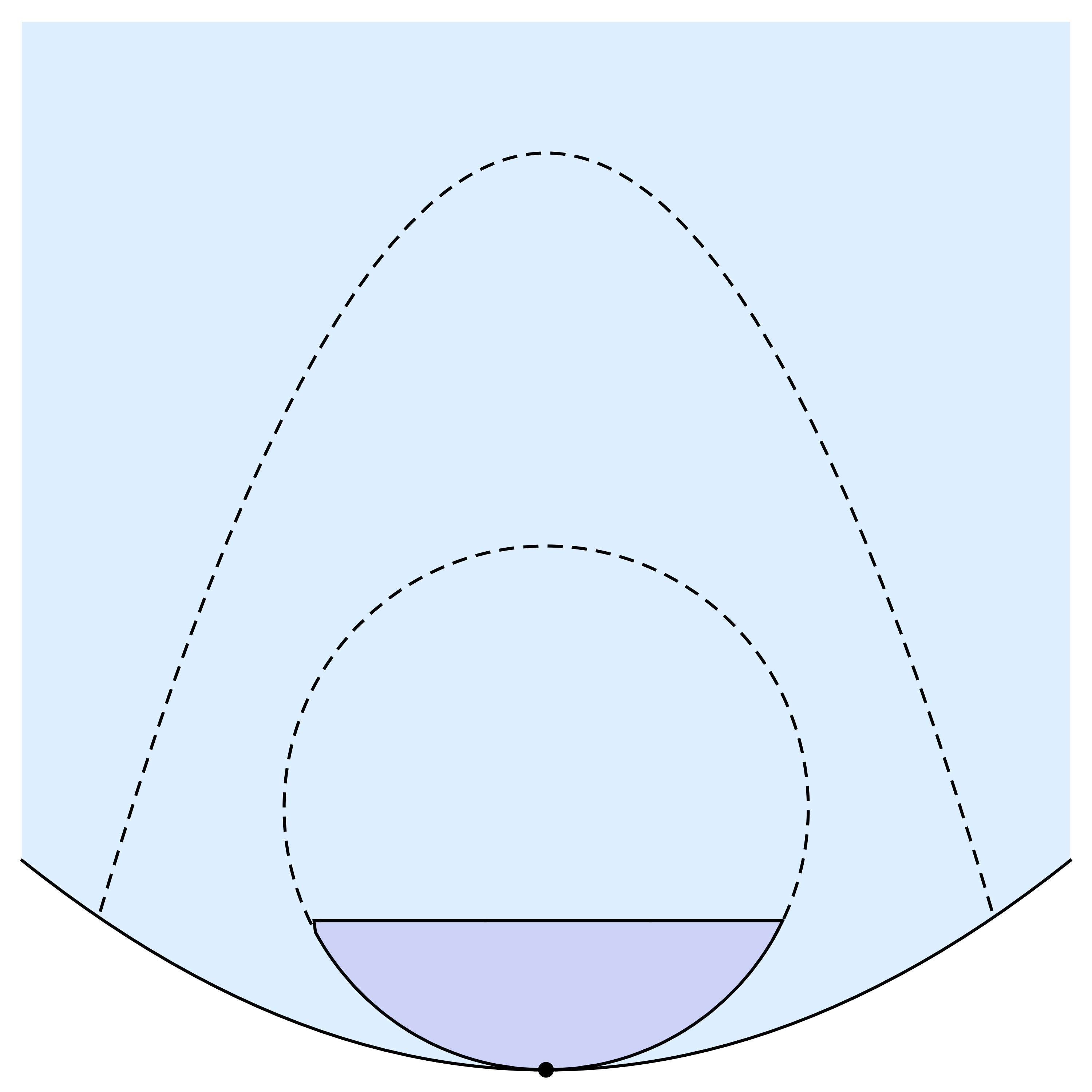}}
 \put(40,130){$\scriptstyle\sO$}
 \put(50,7){$\scriptstyle \partial\sO$}
 \put(102,-6){$\scriptstyle x^0$}
 \put(120,15){$\scriptstyle D$}
 \put(102,130){$\scriptstyle N$}
 \put(90,70){$\scriptstyle E$}
 \put(260,0){\includegraphics[scale=0.5]{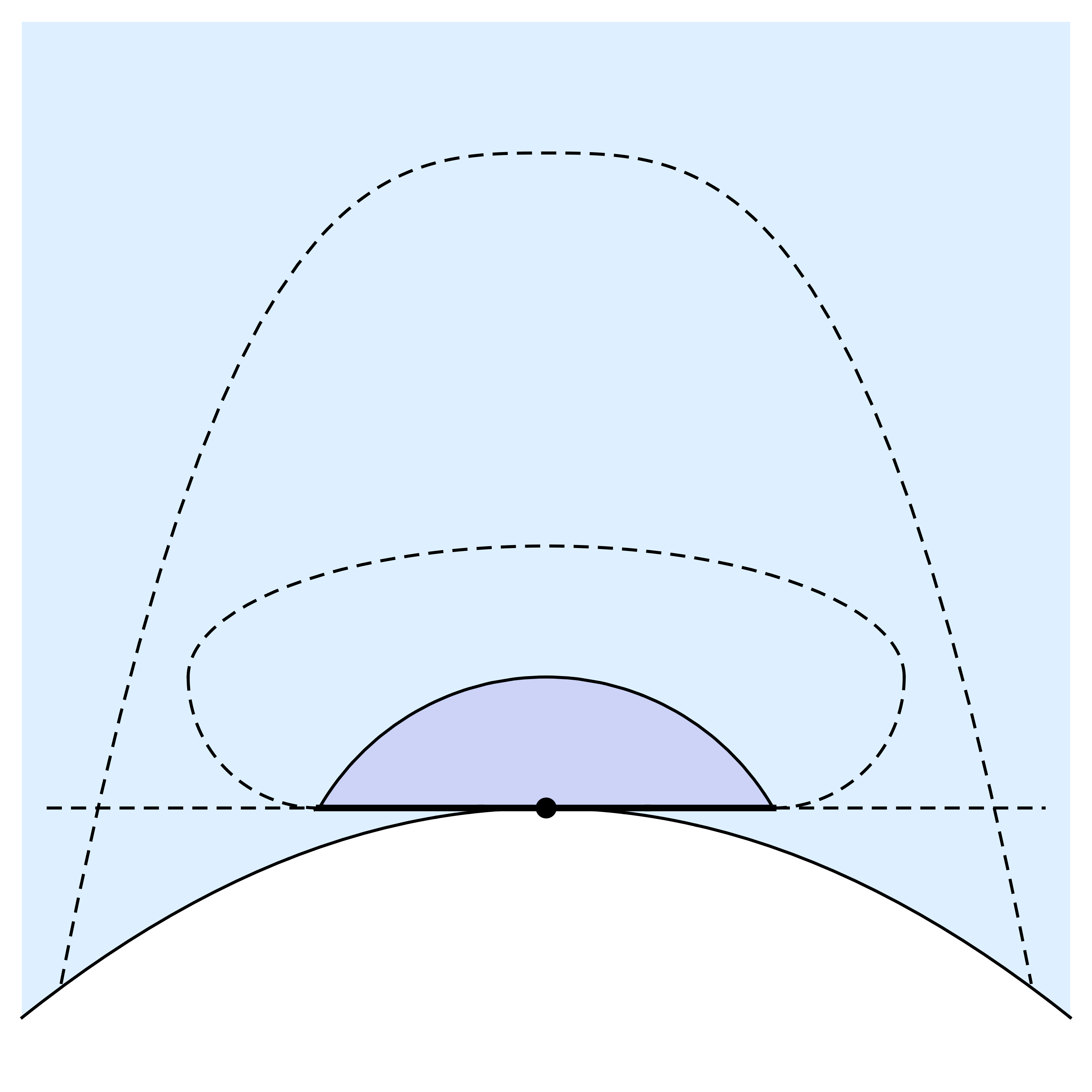}}
 \put(277,41){$\scriptstyle x_d=0$}
 \put(345,38){$\scriptstyle O$}
 \put(345,71){$\scriptstyle C_1$}
 \put(325,50){$\scriptstyle C_2$}
 \put(358,53){$\scriptstyle \widetilde D$}
 \put(280,130){$\scriptstyle \widetilde\sO$}
 \put(290,22){$\scriptstyle \partial\widetilde\sO$}
 \put(345,130){$\scriptstyle \widetilde N$}
 \put(310,70){$\scriptstyle \widetilde E$}
 \put(405,35){$\scriptstyle x_d<0$}
 \put(405,160){$\scriptstyle x_d>0$}
 \end{picture}
 \caption[A quarter-ball and its deformation.]{A quarter-ball and its deformation, denoting $E=B(x^*,R)$ and $\widetilde E = \Phi(E)$ and the quarter-ball by $D$ and $\widetilde D = \Phi(D)$.}
 \label{fig:elliptic_domain_interior_ball_and_deformation}
\end{figure}
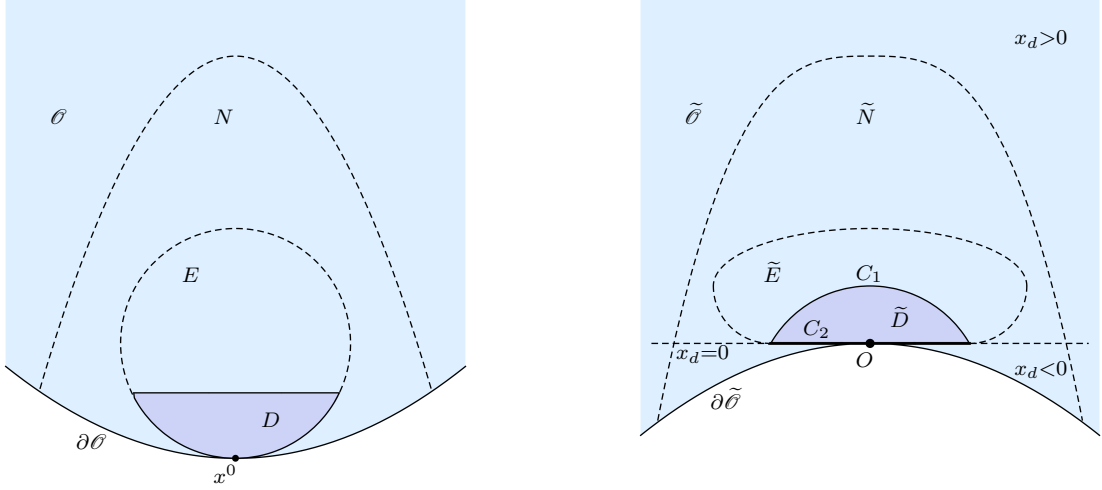

\subsection{A generalization of the Hopf boundary point lemma to linear, second-order differential operators with non-negative definite characteristic form}
\label{subsec:HopfLemma}
We first recall refinements of the statements of the classical weak maximum principle for $A$-subharmonic functions in $C^2(\sO)$ or $W^{2,d}_{\loc}(\sO)$, where a Dirichlet boundary condition is imposed along the full boundary, $\partial\sO$, but the usual strict ellipticity requirement on $\sO$ for the coefficient matrix, $a$, in \eqref{eq:Generator} is relaxed.

We begin with a simple extension of the classical maximum weak maximum principle for $A$-subharmonic functions in $C^2(\sO)$ \cite[Theorem 3.1 and Corollary 3.2]{GilbargTrudinger} using elliptic regularization (see, for example, the proof of \cite[Theorem 6.5]{Crandall_Ishii_Lions_1992}; other versions of the extension are noted in a remark immediately following the statement of \cite[Theorem 3.1]{GilbargTrudinger} and in \cite[p. 33, top of page]{GilbargTrudinger}.

\begin{thm}[Classical weak maximum principle for $A$-subharmonic functions in $C^2(\sO)$ and nonnegative definite characteristic form]
\label{thm:Classical_weak_maximum_principle_C2_elliptic_relaxed}
\cite[Theorem 2.18]{Feehan_perturbationlocalmaxima}
Let $\sO\subset\RR^d$ be a bounded, open subset and $A$ as in \eqref{eq:Generator} with $b$ locally bounded on $\sO$, and $c$ obeying \eqref{eq:c_lower*_positive_domain}, that is, $c_*>0$ on $\sO$. Suppose $u\in C^2(\sO)$ and $\sup_\sO u < \infty$. If $Au \leq 0$ on $\sO$ and $u_*\leq 0$ on $\partial\sO$, then $u\leq 0$ on $\sO$.
\end{thm}

We also have an analogue of Theorem \ref{thm:Classical_weak_maximum_principle_C2_elliptic_relaxed} for functions in $W^{2,d}_{\loc}(\sO)$.

\begin{thm}[Classical weak maximum principle for $A$-subharmonic functions in $W^{2,d}_{\loc}(\sO)$ and nonnegative definite characteristic form]
\label{thm:Classical_weak_maximum_principle_W2d_elliptic_relaxed}
Assume the hypotheses of Theorem \ref{thm:Classical_weak_maximum_principle_C2_elliptic_relaxed} on $\sO$ and $A$, except that the coefficients of $A$ are now required to be measurable. Suppose $u\in W^{2,d}_{\loc}(\sO)$ and $\sup_\sO u < \infty$. If $Au \leq 0$ a.e. on $\sO$ and $u_*\leq 0$ on $\partial\sO$, then $u\leq 0$ on $\sO$.
\end{thm}

\begin{proof}
This follows from the classical weak maximum principle \cite[Theorem 9.1]{GilbargTrudinger} (with $f=0$ on $\sO$) for a full Dirichlet boundary condition along $\partial\sO$ and $A$-subharmonic functions in $W^{2,d}_{\loc}(\sO)$ and our \apriori weak maximum principle estimates, Proposition \ref{prop:Elliptic_weak_maximum_principle_apriori_estimates_C2_W2d}, using the method of proof of \cite[Theorem 2.18]{Feehan_perturbationlocalmaxima}.
\end{proof}

Next, we have the crucial

\begin{lem}[Hopf boundary point lemma for a linear, second-order partial differential operator with non-negative definite characteristic form]
\label{lem:Degenerate_hopf_lemma}
Suppose that $\sO\subset\RR^d$ is an open subset and that $\sO$ obeys an interior sphere condition at $x^0\in\partial\sO$, with an open ball $B(x^*,R)\subset\sO$ such that $x^0\in\partial B(x^*,R)$. Require that the operator $A$ in \eqref{eq:Generator} have
\begin{equation}
\label{eq:c_nonnegative_ball}
c\geq 0 \quad\hbox{(a.e.) on }B(x^*,R),
\end{equation}
and that at least \emph{one} of the of the following conditions hold, where $h=\vec n(x^0)$ denotes the \emph{inward}-pointing unit normal vector at $x^0$:
\begin{align}
\label{eq:LocallyBoundedabcratioBoundary}
{}&\langle ah,h\rangle > 0
\quad\hbox{and}\quad
\frac{\langle b,h\rangle}{\langle ah,h\rangle} \geq -2K_1
\quad\hbox{and}\quad
\frac{c}{\langle ah,h\rangle} \leq K_0  \quad\hbox{(a.e) on } B(x^*,R), \quad\hbox{or}
\\
\label{eq:PositivebcRatio}
\tag{\ref*{eq:LocallyBoundedabcratioBoundary}$'$}
{}& \langle b,h\rangle \geq \frac{b_0}{2} \quad\hbox{and}\quad c \leq C_0 \quad\hbox{(a.e) on } B(x^*,R),
\end{align}
for some positive constants, $b_0, C_0, K_0, K_1$. Finally, when the coefficients of $A$ are measurable rather than everywhere-defined on $\sO$, require in addition that $b$ obeys \eqref{eq:b_locally_bounded_on_domain} and $c$ obeys \eqref{eq:c_lower*_positive_domain}. Suppose that $u\in C^2(\sO)$ or $u\in W^{2,d}_{\loc}(\sO)$, obeys
$Au\leq 0$ (a.e.) on $\sO$, and satisfies the conditions,
\begin{enumerate}
\renewcommand{\theenumi}{\roman{enumi}}
\item $u$ is continuous at $x^0$;
\item\label{item:pzeroStrictLocalMax} $u(x^0) > u(x)$, for all $x\in\sO$;
\item $D_{\vec n} u(x^0)$ exists,
\end{enumerate}
where $D_{\vec n} u(x^0)$ is the derivative of $u$ at $x^0$ in the direction of the \emph{inward}-pointing unit normal vector, $\vec n(x^0)$. Then the following hold.
\begin{enumerate}
\item\label{item:Hopf_c_zero} If $c=0$ on $\sO$, then $D_{\vec n} u(x^0)$ obeys the strict inequality,
\begin{equation}
\label{eq:PositiveInwardNormalDerivative}
D_{\vec n} u(x^0) < 0.
\end{equation}
\item\label{item:Hopf_c_geq_zero} If $c\geq 0$ on $\sO$ and $u(x^0)\geq 0$, then \eqref{eq:PositiveInwardNormalDerivative} holds.
\item\label{item:Hopfcnosign} If $u(x^0)=0$, then \eqref{eq:PositiveInwardNormalDerivative} holds irrespective of the sign of $c$.
\end{enumerate}
\end{lem}

The preceding version of the Hopf lemma is `local', in the sense that its hypotheses are given in terms of properties of the coefficients of $A$ in \eqref{eq:Generator} over the ball, $B$. In our application to the proof of our strong maximum principle (Theorem \ref{thm:Strong_maximum_principle}), it will be convenient to have a version of the Hopf lemma with simpler conditions on the coefficients of $A$ in \eqref{eq:Generator} over $\sO$ and $\partial\sO$.

\begin{cor}[Hopf boundary point lemma with simplified hypotheses]
\label{cor:Hopf_very_simple}
Assume the hypotheses of Lemma \ref{lem:Degenerate_hopf_lemma} except that the condition \eqref{eq:c_nonnegative_ball} is replaced by \eqref{eq:c_nonnegative_domain}, condition \eqref{eq:LocallyBoundedabcratioBoundary} is replaced by
\begin{subequations}
\label{eq:abc_hopf_very_simple_1}
\begin{align}
\label{eq:a_hopf_very_simple_1}
{}& \lambda_* > 0 \quad\hbox{on } \sO\cup\partial_1\sO,
\\
\label{eq:b_hopf_very_simple_1}
{}& b \in L^\infty_{\loc}(\sO\cup\partial_1\sO;\RR^d),
\\
\label{eq:c_hopf_very_simple_1}
{}& c \in L^\infty_{\loc}(\sO\cup\partial_1\sO),
\end{align}
\end{subequations}
and condition \eqref{eq:PositivebcRatio} is replaced  by
\begin{subequations}
\label{eq:bc_hopf_very_simple_2}
\begin{align}
\label{eq:b_perp_positive_hopf_very_simple_2}
{}& b^\perp > 0 \quad\hbox{on } \overline{\partial_0\sO},
\\
\label{eq:b_continuous_degenerate_boundary_hopf_very_simple_2}
{}& b \in C_{\loc}(\overline{\partial_0\sO}),
\\
\label{eq:c_hopf_very_simple_2}
{}& c \in L^\infty_{\loc}(\sO\cup \overline{\partial_0\sO}).
\end{align}
\end{subequations}
Then the conclusions of Lemma \ref{lem:Degenerate_hopf_lemma} continue to hold.
\end{cor}

Note that conditions \eqref{eq:c_hopf_very_simple_1} and \eqref{eq:c_hopf_very_simple_2} in Corollary \ref{cor:Hopf_very_simple} could be combined and replaced by the single condition $c \in L^\infty_{\loc}(\bar\sO)$, but the present separation will be more convenient in our application.

\begin{rmk}[Application of Corollary \ref{cor:Hopf_very_simple} to the proof of the strong maximum principle]
\label{rmk:Application_hopf_lemma_to_strong_maximum_principle}
Corollary \ref{cor:Hopf_very_simple} provides a version of Lemma \ref{lem:Degenerate_hopf_lemma} with simplified hypotheses, analogous to the classical Hopf lemma \cite[Lemma 3.4]{GilbargTrudinger}. However, in our application to the proof of our strong maximum principle (Theorem \ref{thm:Strong_maximum_principle}), we need only consider the case $B\Subset\underline\sO$, by analogy with the case $B\Subset\sO$ in the application of the classical Hopf lemma to the proof of the classical strong maximum principle \cite[Theorem 3.5]{GilbargTrudinger}. For that purpose, the group of conditions \eqref{eq:abc_hopf_very_simple_1} can be relaxed to \eqref{eq:a_locally_strictly_elliptic_interior_domain}, $b\in L^\infty_{\loc}(\sO;\RR^d)$, and $c\in L^\infty_{\loc}(\sO)$ and the group of conditions \eqref{eq:bc_hopf_very_simple_2} can be relaxed to \eqref{eq:b_perp_positive_boundary_elliptic}, \eqref{eq:b_continuous_on_degenerate_boundary}, and \eqref{eq:c_locally_bounded_on_domain_plus_degenerate_boundary}.
\end{rmk}

\begin{proof}[Proof of Corollary \ref{cor:Hopf_very_simple}]
First, suppose that $x^0\in \partial B\cap\partial_1\sO$; for this case, we shall invoke the conditions \eqref{eq:abc_hopf_very_simple_1} on the coefficients $a,b,c$. Since $a$ obeys \eqref{eq:a_hopf_very_simple_1} on $\sO\cup \partial_1\sO$ and $\bar B\subset \sO\cup \partial_1\sO$, it follows that
$$
\langle a(x)h,h\rangle \geq \lambda_*(x) \geq \lambda_0,
\quad\hbox{(a.e.) } x\, \in B,
$$
for some constant $\lambda_0=\lambda_0(B)>0$. Moreover, $c \leq C_0$ (a.e.) on $B$, for some constant $C_0=C_0(B)<\infty$ by \eqref{eq:c_hopf_very_simple_1}, and hence we obtain
$$
\frac{c(x)}{\langle a(x)h,h\rangle} \leq \frac{c(x)}{\lambda_*(x)} \leq \frac{C_0}{\lambda_0} < \infty,
\quad\hbox{(a.e.) } x\, \in B.
$$
Because $b$ is locally bounded on $\sO\cup \partial_1\sO$ by \eqref{eq:b_hopf_very_simple_1}, we have $\langle b,h\rangle\leq L_0$ (a.e.) on $B$ for some constant $L_0=L_0(B)<\infty$, and thus
$$
\frac{\langle b(x),h\rangle}{\langle a(x)h,h\rangle} \leq \frac{L_0}{\lambda_0}, \quad\hbox{(a.e.) } x\, \in B,
$$
and therefore
$$
\frac{\langle b(x),h\rangle}{\langle a(x)h,h\rangle} \geq -\frac{L_0}{\lambda_0} > -\infty,
\quad \quad\hbox{(a.e.) } x\, \in B.
$$
Combining the preceding observations yields \eqref{eq:LocallyBoundedabcratioBoundary}.

Second, suppose that $x^0\in \partial B\cap\overline{\partial_0\sO}$; for this case, we shall invoke the conditions \eqref{eq:bc_hopf_very_simple_2} on the coefficients $b,c$, passing without loss of generality to a possibly smaller ball $B'\subset B\subset\sO$ with $x^0\in \partial B'\cap\overline{\partial_0\sO}$. We have that $b_0 := b^\perp(x^0)$ is positive by \eqref{eq:b_perp_positive_hopf_very_simple_2} and $b^\perp(x) = \langle b(x), \vec n(x)\rangle$ for (a.e.) $x \in N(\partial_0\sO)$ by \eqref{eq:b_splitting_elliptic}. Therefore, by \eqref{eq:b_continuous_degenerate_boundary_hopf_very_simple_2} we see that $\langle b(x^0),h\rangle = b^\perp(x^0) = b_0$ and, for small enough $B'$,
$$
\langle b(x),h\rangle \geq \frac{b_0}{2} \quad\hbox{(a.e.) } x \in B'.
$$
Moreover, $c(x) \leq C_0'<\infty$ for a.e. $x\in B'$ by \eqref{eq:c_hopf_very_simple_2}, for some constant $C_0'=C_0'(B')$. Combining the preceding observations yields \eqref{eq:PositivebcRatio} with $B$ replaced by $B'$.
\end{proof}

\begin{rmk}[Application of the Hopf lemma to the case of boundary points where the operator is degenerate]
\label{rmk:Hopf}
The Hopf lemma for points in $\{x_2=0\}$ for the Heston operator in Example \ref{exmp:HestonPDE} was proved independently by Daskalopoulos in an unpublished manuscript using a barrier function similar to that in the proof of \cite[Lemma 3.4]{GilbargTrudinger}. I am grateful to her for suggesting that such results should hold even at boundary points where the operator becomes degenerate. As pointed out to me by Pop, another version of the Hopf lemma for was obtained by Epstein and Mazzeo as \cite[Lemma 4.2.4]{Epstein_Mazzeo_annmathstudies} for their generalized Kimura diffusion operators, but also proved using a barrier function similar to that in the proof of \cite[Lemma 3.4]{GilbargTrudinger}.

The traditional proof of the Hopf lemma, as described by Gilbarg and Trudinger \cite[Lemma 3.4]{GilbargTrudinger}, by Evans \cite[Section 6.4.2]{Evans}, or by Han and Lin \cite[Theorem 2.5]{Han_Lin_2011}, exploits the interior sphere condition by choosing the barrier function
$$
v(x) = e^{-\alpha|x|^2} - e^{-\alpha R^2}, \quad x\in B(x^*,R)\less \bar B(x^*,\rho),
$$
where $\alpha>0$ is a constant which is ultimately depends on bounds on the coefficients of $A$ on the open annulus $B(x^*,R)\less \bar B(x^*,\rho)$, where $\rho\in (0,R)$ is a constant. This is the model for the barrier function chosen in the proof of \cite[Lemma 4.2.4]{Epstein_Mazzeo_annmathstudies}, but the resulting argument is quite difficult. As we shall see in our proof of Lemma \ref{lem:Degenerate_hopf_lemma}, however, a choice of exponential-linear or linear barrier function instead will easily lead to the desired result.

That such a Hopf lemma should hold even at points in $\partial_0\sO$ can be seen by examining the \emph{Kummer equation} \cite[Section 13.1.1]{AbramStegun},
$$
Au(x) := -xu_{xx}(x) - (b-x)u_x(x) + cu(x) = 0, \quad x\in\RR_+,
$$
where $b$ and $c$ are positive constants here. If $u_x(0)$ exists, then $u$ is necessarily a constant multiple of a \emph{confluent hypergeometric function of the first kind}, $u(x)=kM(x)$, by \cite[Sections 13.1.2--4, 13.4.8, 13.4.21, and 13.5.5--10]{AbramStegun} and $u$ is $C^\infty$ on $[0,\infty)$ with $M(0)=1$ and $u(0)=k\in\RR$. (Indeed, when $a=c$ then $M(x)=e^x$.) The continuity of $Au$ on $[0,\infty)$ implies that $u_x(0) = kc/b$. If $k<0$, then $Au=0$ on $(0,\infty)$ and $u$ has a strict local maximum at $x=0$ and $u_x(0)<0$, as predicted by Lemma \ref{lem:Degenerate_hopf_lemma}.
\end{rmk}

\begin{proof}[Proof of Lemma \ref{lem:Degenerate_hopf_lemma}]
We use the strategy of the proof of \cite[Lemma 3.4]{GilbargTrudinger}, with two choices of barrier functions, depending on whether condition \eqref{eq:LocallyBoundedabcratioBoundary} or \eqref{eq:PositivebcRatio} holds but both different from that in the proof of \cite[Lemma 3.4]{GilbargTrudinger} and an approach to exploiting the interior sphere condition which is also different from that in the proof of \cite[Lemma 3.4]{GilbargTrudinger}.

\begin{step}[Geometric set-up and application of a $C^2$ diffeomorphism]
\label{step:DegenerateHopfLemmaProof_C2diffeomorphism}
We may assume without loss of generality, using a translation of $\RR^d$ if needed, that $x^0=0\in\RR^d$ and $h=e_d$ and that $B(x^*,R)$ is contained in the open \emph{upper} half-space $\{x_d>0\}$. We now apply a $C^2$ diffeomorphism, $\Phi:\RR^d\to\RR^d$ with $\Phi(0)=0$, to \emph{flatten} the portion $\{0\leq x_d<R\}\cap\partial B(x^*,R)$ of the boundary of $B(x^*,R)$ by pushing it \emph{downward}, as in Figure \ref{fig:elliptic_domain_interior_ball_and_deformation}, so that
\begin{align*}
\Phi\left(\{0\leq x_d<R\}\cap\partial B(x^*,R)\right) &= \{x\in\RR^d: |x|<R, x_d=0\}
\\
&= B(0,R)\cap\{x_d=0\}.
\end{align*}
If $T\subset\partial\sO$ is a small, relatively open neighborhood of $x^0\in\partial\sO$, the map $\Phi$ pushes $T\less\{x^0\}$ \emph{downward} into the open \emph{lower} half-space, $ \{x_d<0\}$:
\begin{gather*}
\Phi(T\less\{x^0\})\subset \{x_d<0\},
\\
\Phi(\{x_d\geq R\}\cap\partial B(x^*,R)) \Subset \Phi(\sO).
\end{gather*}
Henceforth, \emph{after} applying the preceding diffeomorphism (and now denoting $\Phi(\sO)$ simply by $\sO$), we may assume, without loss of generality, that the open \emph{half-ball}
$$
B^+(0,R):=\{x\in\RR^d:|x|<R, x_d>0\}\subset\sO,
$$
has the property that
$$
\{x_d>0\}\cap\partial B^+(0,R) \Subset\sO.
$$
One can now check (see the proofs of \cite[Theorem 6.3.4]{Evans}, \cite[Lemma 6.5 or Theorem 8.12]{GilbargTrudinger}, \cite[Lemma 6.2.1]{Krylov_LecturesHolder} for similar arguments)
that the conditions \eqref{eq:LocallyBoundedabcratioBoundary} and \eqref{eq:PositivebcRatio} are equivalent to
\begin{align}
\label{eq:LocallyBoundedabcratioBoundary_halfball}
{}&\langle ah,h\rangle > 0
\quad\hbox{and}\quad
\frac{\langle b,h\rangle}{\langle ah,h\rangle} \geq -2K_1
\quad\hbox{and}\quad
\frac{c}{\langle ah,h\rangle} \leq K_0 \quad\hbox{(a.e.) on } B^+(0,R), \quad\hbox{or }
\\
\label{eq:PositivebcRatio_halfball}
\tag{\ref*{eq:LocallyBoundedabcratioBoundary_halfball}$'$}
{}& b_d \geq \frac{b_0}{2} \quad\hbox{and}\quad c \leq C_0 \quad\hbox{(a.e.) on } B^+(0,R),
\end{align}
for a possibly smaller positive constant, $b_0$, and possibly larger positive constants, $C_0, K_0, K_1$.
\end{step}

\begin{step}[Construction of the barrier function when condition \eqref{eq:LocallyBoundedabcratioBoundary_halfball} holds]
\label{step:Degenerate_hopf_lemma_proof_barrier_one}
We choose
$$
v(x) := e^{\alpha x_d} - 1, \quad x\in\RR^d,
$$
where $\alpha>0$ is a constant yet to be determined. Clearly, $v(0)=0$ and $v\geq 0$ on the half-space $\{x_d\geq 0\}$ and, in particular, $v\geq 0$ on $B^+(0,R)$. Moreover,
\begin{align*}
Av &= -\alpha^2a^{dd} e^{\alpha x_d} - b^d\alpha e^{\alpha x_d} + c\left(e^{\alpha x_d} - 1\right)
\\
&\leq -a^{dd}\left(\alpha^2 + \alpha\frac{b^d}{a^{dd}} - \frac{c}{a^{dd}}\right)e^{\alpha x_d}
\\
&\leq -a^{dd}\left(\alpha^2 + 2K_1\alpha - K_0\right)e^{\alpha x_d} \quad\hbox{(a.e.) on }B^+(0,R),
\end{align*}
using the hypothesis \eqref{eq:c_nonnegative_ball} to obtain $c\geq 0$ (a.e.) on $B^+(0,R)$ and noting that $a^{dd}> 0$ (a.e.) on $B^+(0,R)$ by hypothesis \eqref{eq:LocallyBoundedabcratioBoundary_halfball} and that $a, b, c$ obey \eqref{eq:LocallyBoundedabcratioBoundary_halfball}. But
$$
\alpha^2 + 2K_1\alpha - K_0 = (\alpha-K_1)^2 - K_1^2-K_0 > 0
$$
provided $\alpha$ obeys
$$
\alpha > K_1+\sqrt{K_0+K_1^2}.
$$
We fix such an $\alpha$ and thus obtain $Av < 0$ (a.e.) on $B^+(0,R)$.
\end{step}

\begin{step}[Construction of the barrier function when condition \eqref{eq:PositivebcRatio_halfball} holds]
\label{step:Degenerate_hopf_lemma_proof__barrier_two}
We choose
$$
v(x) := x_d, \quad x\in\RR^d,
$$
and observe that
$$
Av = -b^dv_{x_d} + cv = -b^d + cx_d.
$$
Since $b_d \geq b_0/2$ and $c\leq C_0$ (a.e.) on $B^+(0,R)$ by \eqref{eq:PositivebcRatio_halfball}, we obtain
$$
Av = -b^d + cx_d \leq -\frac{b_0}{2} + C_0x_d \leq -\frac{b_0}{2} + C_0R' < 0 \quad\hbox{(a.e.) on }B^+(0,R'),
$$
provided $R'\in(0,R]$ is chosen small enough. Since the size of $R'$ is immaterial in the remainder of the proof, for notational simplicity we shall simply write $Av<0$ (a.e.) on $B^+(0,R)$ in this case as well.
\end{step}

\begin{step}[Verification that the weak maximum principle holds for $A$ on $B^+(0,R)$]
\label{step:Degenerate_hopf_lemma_proof_weak_max_principle}
We consider separately the cases where the coefficients of $A$ are everywhere-defined on $\sO$ (with $u\in C^2(\sO)$) and the coefficients of $A$ are measurable on $\sO$ (with $u\in W^{2,d}_{\loc}(\sO)$).

First, consider the case where the coefficients of $A$ are everywhere-defined on $\sO$. From Steps \ref{step:Degenerate_hopf_lemma_proof_barrier_one} and \ref{step:Degenerate_hopf_lemma_proof__barrier_two}, we obtain $Av<0$ on $B^+(0,R)$ for either choice of barrier function, $v$. Therefore, the proof of \cite[Theorem 3.1]{GilbargTrudinger}, with the role of $e^{\gamma x_1}$ in \cite[p. 32]{GilbargTrudinger} replaced by $v$, shows that the conclusions of  the classical weak maximum principle \cite[Theorem 3.1 and Corollary 3.2]{GilbargTrudinger} (that is, with full boundary comparison) hold for the operator $A$ on $B^+(0,R)$ and $A$-subharmonic functions\footnote{Alternatively, if we had assumed \eqref{eq:c_lower*_positive_domain} for this case too, Theorem \ref{thm:Classical_weak_maximum_principle_C2_elliptic_relaxed} would yield the same conclusion.}
$$
w \in C^2(B^+(0,R))\cap C(\bar B^+(0,R)).
$$
Second, consider the case where the coefficients of $A$ are measurable on $\sO$ and we assume in addition that \eqref{eq:c_lower*_positive_domain} holds. Then Theorem \ref{thm:Classical_weak_maximum_principle_W2d_elliptic_relaxed} implies that the classical weak maximum principle holds for $A$ on $B^+(0,R)$ and $A$-subharmonic functions,
$$
w \in W^{2,d}_{\loc}(B^+(0,R))\cap C(\bar B^+(0,R)),
$$
concluding this step.
\end{step}

\begin{step}[Application of the weak maximum principle]
\label{step:DegenerateHopfLemmaProof_AppWeakMaxPrin}
Since $u-u(0)<0$ on $\sO$ and $u\in C_{\loc}(\bar\sO)$ and $\{x_d>0\}\cap\partial B^+(0,R)\Subset\sO$, we obtain
$$
u(x) - u(0) \leq -m_0 < 0, \quad\forall\, x\in \{x_d>0\}\cap\partial B^+(0,R),
$$
for some positive constant, $m_0$, depending on $R$ and $u$. If \eqref{eq:LocallyBoundedabcratioBoundary_halfball} holds, then
$$
v(x) = e^{\alpha x_d}-1 \leq e^{\alpha R}-1, \quad\forall\, x\in \{x_d>0\}\cap\partial B^+(0,R),
$$
while if \eqref{eq:PositivebcRatio_halfball} holds, then
$$
v(x) = x_d \leq R, \quad\forall\, x\in \{x_d>0\}\cap\partial B^+(0,R).
$$
Hence, recalling that $R>0$, there is a positive constant $m_1 := (e^{\alpha R}-1)\vee R$ such that
$$
v(x) \leq m_1, \quad\forall\, x\in \{x_d>0\}\cap\partial B^+(0,R).
$$
Consequently,
$$
u(x) - u(0) + \eps v(x) \leq -m_0 + \eps m_1 \leq 0, \quad \forall\, x\in \{x_d>0\}\cap\partial B^+(0,R),
$$
provided we fix $\eps$ in the range $0<\eps\leq m_0/m_1$, while
$$
u(x) - u(0) + \eps v(x) = u(x) - u(0) \leq 0, \quad \forall\, x\in \{x_d=0\}\cap\partial B^+(0,R),
$$
since for either choice of barrier function we have $v(x)=0$ when $x_d=0$ and our hypothesis \eqref{item:pzeroStrictLocalMax} (with $x^0=0$) implies that $u(x) \leq u(0)$ on $\partial B^+(0,R)\subset\sO\cup\{0\}$. But
$$
A(u - u(0) + \eps v) = Au - cu(0) +\eps Av \leq -cu(0) \leq 0 \quad\hbox{(a.e.) on } B^+(0,R),
$$
where the last inequality holds if $c=0$ (as in Conclusion \eqref{item:Hopf_c_zero}), or $c\geq 0$ and $u(0)\geq 0$  (as in Conclusion \eqref{item:Hopf_c_geq_zero}), or $u(0)=0$ (as in Conclusion \eqref{item:Hopfcnosign}). (For the case $u(0)=0$, we simply note as in the proof of \cite[Lemma 3.4]{GilbargTrudinger} that we can replace $A$ by $A+c^-$, where we write $c=c^+-c^-$.)

The weak maximum principle (from Step \ref{step:Degenerate_hopf_lemma_proof_weak_max_principle}) therefore yields
$$
u - u(0) + \eps v \leq 0 \quad\hbox{on } B^+(0,R),
$$
by virtue of Step \ref{step:Degenerate_hopf_lemma_proof_weak_max_principle}.
\end{step}

\begin{step}[Sign of the directional derivative of the subsolution at the boundary]
\label{step:DegenerateHopfLemmaProof_SignDirectionalDeriv}
We have
$$
\frac{u(x)-u(0)}{x_d} \leq -\eps\frac{v(x)}{x_d} = -\eps\frac{v(x)-v(0)}{x_d}, \quad\forall\, x\in B^+(0,R).
$$
If $v(x) = e^{\alpha x_d}-1$, we have $v_{x_d}=\alpha e^{\alpha x_d}$ and $v_{x_d}(0) = \alpha > 0$, while if $v(x) = x_d$, we have $v_{x_d}=1$. Taking the limit as $x_d\downarrow 0$ and noting that
$$
v_{x_d}(0)
=
\begin{cases}
\alpha &\hbox{if $a,b,c$ obey \eqref{eq:LocallyBoundedabcratioBoundary_halfball},}
\\
1 &\hbox{if $b,c$ obey \eqref{eq:PositivebcRatio_halfball},}
\end{cases}
$$
yields $v_{x_d}(0) \geq \alpha\wedge 1$ and
$$
u_{x_d}(0) \leq -\eps v_{x_d}(0) \leq -\eps(\alpha\wedge 1) < 0,
$$
and thus \eqref{eq:PositiveInwardNormalDerivative} holds.
\end{step}
This completes the proof.
\end{proof}

\begin{rmk}[Application to the elliptic Heston operator]
\label{rmk:EllipticHestonDegenerateHopf}
The hypotheses of Lemma \ref{lem:Degenerate_hopf_lemma} on the coefficients of $A$ are obeyed in the case of the elliptic Heston operator, Example \ref{exmp:HestonPDE}, where $d=2$ and $\sO=\HH$. For example, if $x^0\in\partial\HH$ then $h=\vec n(x^0) = e_2$, while $a^{22} = \sigma^2 x_2/2$ and $b^2 = \kappa(\theta-x_2)$, so
$$
\frac{b^2}{a^{22}} = \frac{2\kappa(\theta-x_2)}{\sigma^2 x_2} \geq -\frac{2\kappa}{\sigma^2}, \quad\forall\, x_2\geq 0.
$$
Thus, condition \eqref{eq:LocallyBoundedabcratioBoundary} is obeyed when $r=0$, noting that $c=r$, while if $r>0$, then
$$
\frac{b^2}{c} = \frac{\kappa(\theta-x_2)}{r} \geq \frac{\kappa\theta}{2r}, \quad 0\leq x_2 < \theta/2,
$$
and thus condition \eqref{eq:PositivebcRatio} is obeyed.
\end{rmk}

\begin{rmk}[Application to linear, second-order, strictly and uniformly elliptic operators]
\label{rmk:DegenerateImpliesGTHopfLemma}
The classical Hopf boundary point lemma \cite[Lemma 3.4]{GilbargTrudinger} requires that the coefficients of $A$ in \eqref{eq:Generator} obey a uniformly ellipticity condition on $\sO$ \cite[p. 31]{GilbargTrudinger} and the bounds in \cite[Equation (3.2)]{GilbargTrudinger}. Such hypotheses imply that the coefficients of $A$ obey the inequalities \eqref{eq:LocallyBoundedabcratioBoundary} and hence that our Lemma \ref{lem:Degenerate_hopf_lemma} implies \cite[Lemma 3.4]{GilbargTrudinger}, although the converse is not true, as Remark \ref{rmk:EllipticHestonDegenerateHopf} illustrates.
\end{rmk}

\begin{rmk}[Hopf boundary point lemma for open subsets obeying an interior cone condition]
\label{rmk:HopfInteriorCone}
The interior sphere condition can be relaxed in the classical Hopf lemma \cite[Lemma 3.4]{GilbargTrudinger}, as noted in \cite[p. 35 and p. 46]{GilbargTrudinger}, and generalizations to open subsets with non-smooth points are described in \cite{Lieberman_2001b, Miller_1967, Nadirashvili_1981, Oddson_1968}.
\end{rmk}

\subsection{Strong maximum principle}
\label{subsec:Strong_maximum_principle}
Recall that by a `domain' in $\RR^d$, we always mean a \emph{connected}, open subset. We shall now adapt the proof of \cite[Theorem 3.5]{GilbargTrudinger}, applying our Lemma \ref{lem:Degenerate_hopf_lemma} (or more precisely its simpler form, Corollary \ref{cor:Hopf_very_simple}, as discussed in Remark \ref{rmk:Application_hopf_lemma_to_strong_maximum_principle}) instead of \cite[Lemma 3.4]{GilbargTrudinger}, to give

\begin{thm}[Strong maximum principle for $A$-subharmonic functions in $C^2(\sO)$]
\label{thm:Strong_maximum_principle}
Suppose that $\sO\subset\RR^d$ is a domain. Require that the operator $A$ as in \eqref{eq:Generator} have coefficients obeying \eqref{eq:a_locally_strictly_elliptic_interior_domain}, \eqref{eq:b_perp_positive_boundary_elliptic}, \eqref{eq:c_nonnegative_domain}, \eqref{eq:c_nonnegative_boundary}, \eqref{eq:bc_locally_bounded}, and \eqref{eq:b_continuous_on_degenerate_boundary}. Require, in addition, that $\partial_0\sO$ obey \eqref{eq:C1alpha_degenerate_boundary} for some $\alpha\in(0,1)$. If $u\in C^2_s(\underline\sO)$ obeys $Au\leq 0$ on $\sO$, then the following hold.
\begin{enumerate}
\item\label{item:Strong_maximum_principle_c_zero} If $c=0$ on $\underline\sO$ and $u$ attains a global maximum in $\underline\sO$, then $u$ is constant on $\sO$.
\item\label{item:Strong_maximum_principle_c_geq_zero} If $c\geq 0$ on $\underline\sO$ and $u$ attains a \emph{non-negative} global maximum in $\underline\sO$, then $u$ is constant on $\sO$.
\end{enumerate}
\end{thm}

\begin{proof}
Consider Conclusion \eqref{item:Strong_maximum_principle_c_zero}. Assume, to the contrary, that $u$ is non-constant on $\sO$ and achieves a global maximum $M$ at a point in $\underline\sO$. Let $\sO^- := \{x\in\sO:u(x)<M\}$ and observe that $\sO^-$ is non-empty by our assumption that $u$ is non-constant on $\sO$. Let $x^*\in\sO^-$ be such that $\dist(x^*,\partial\sO^-)<\dist(x^*,\partial_1\sO)$ (if $\partial_1\sO=\emptyset$, then any $x^*\in\sO^-$ will do) and let $B\subset\sO^-$ be the largest open ball centered at $x^*$ and contained in $\sO^-$. Then $u(x^0)=M$ for some $x^0\in \partial B \cap \partial\sO^-$ and $u<M$ on $B$. Note that $x^0\in\underline\sO$, since $\dist(x^*,\partial\sO^-)<\dist(x^*,\partial_1\sO)$ by choice of $x^*$.

\setcounter{case}{0}
\begin{case}[$x^0\in \sO$]
We must have $Du(x^0)=0$ since $x^0$ is an interior local maximum. However, by applying Conclusion \eqref{item:Hopf_c_zero} in Lemma \ref{lem:Degenerate_hopf_lemma} to the operator $A$ on the open subset $B$ and boundary point $x^0\in\partial B$, we obtain $Du(x^0)\neq 0$, a contradiction.
\end{case}

\begin{case}[$x^0\in\partial_0\sO$]
If $\vec\tau(x^0)\in\RR^d$ is tangential to $\partial_0\sO$ at $x^0$, then $\langle\tau(x^0), Du(x^0)\rangle = 0$, since $x^0$ is a local maximum for $u$ in $\partial_0\sO$ and $\partial_0\sO$ is $C^{1,\alpha}$ by hypothesis \eqref{eq:C1alpha_degenerate_boundary} and applying the boundary-straightening result \cite[Lemma B.1]{Feehan_perturbationlocalmaxima}. Therefore, by the splitting \eqref{eq:b_splitting_elliptic} of $b=b^\parallel + b^\perp\vec n$ near $\partial_0\sO$ and the property \eqref{eq:Second_order_boundary_vanishing} of functions $u\in C^2_s(\underline\sO)$ and our hypothesis that $c=0$ on $\underline\sO$, we obtain
\begin{align*}
Au(x^0) &= -\tr(aD^2u)(x^0) - \langle b(x^0),Du(x^0)\rangle + c(x^0)u(x^0)
\\
&= -b^\perp(x^0)\langle \vec n(x^0),Du(x^0)\rangle.
\end{align*}
But $b^\perp(x^0)>0$ by hypothesis \eqref{eq:b_perp_positive_boundary_elliptic} and Conclusion \eqref{item:Hopf_c_zero} of Lemma \ref{lem:Degenerate_hopf_lemma} yields $D_{\vec n}u(x^0) < 0$, so we obtain
$$
Au(x^0) > 0,
$$
contradicting the fact that $Au\leq 0$ on $\sO$ by hypothesis and hence $Au\leq 0$ on $\underline\sO$ by the property \eqref{eq:Second_order_boundary_continuity} of functions $u\in C^2_s(\underline\sO)$.
\end{case}

Conclusion \eqref{item:Strong_maximum_principle_c_geq_zero} follows by an identical argument when $u$ achieves a non-negative maximum in $\underline\sO$, except that we now appeal to Conclusion \eqref{item:Hopf_c_geq_zero} in Lemma \ref{lem:Degenerate_hopf_lemma}.
\end{proof}

\begin{rmk}[Strong maximum principle for $A$-subharmonic functions in $W^{2,d}_{\loc}(\sO)$]
\label{rmk:Strong_maximum_principle_W2d}
While the proof of the classical strong maximum principle for functions in  $C^2(\sO)$ \cite[Theorem 3.5]{GilbargTrudinger} may be modified easily to give a version for $A$-subharmonic functions in $W^{2,d}_{\loc}(\sO)$, as in \cite[Theorem 9.6]{GilbargTrudinger}, that is not the case for Theorem \ref{thm:Strong_maximum_principle} because of the need for $Au(x)$ to be defined at each point of $x\in\partial_0\sO$, as we see in the proof. A version of Theorem \ref{thm:Strong_maximum_principle} for $A$-subharmonic functions in $W^{2,d}_{\loc}(\sO)\cap C^1(\underline\sO)$ is developed in \cite{Feehan_perturbationlocalmaxima}, but the proof is considerably more difficult.
\end{rmk}

We next consider the question of uniqueness in the Neumann problem and note here that it is important to distinguish between $\partial\sO$ and $\partial_1\sO$.

\begin{thm}[Uniqueness for the Neumann problem for $C^2$ functions on bounded domains]
\label{thm:Neumann_boundary_condition_uniqueness_equation}
Let $\sO\subset\RR^d$ be a \emph{bounded} domain. Let $A$ as in \eqref{eq:Generator} and $\partial_0\sO\subsetneqq\partial\sO$ in \eqref{eq:Degeneracy_locus_elliptic} obey the hypotheses of Theorem \ref{thm:Strong_maximum_principle} and assume that $\sO$ satisfies an interior sphere condition at each point of $\overline{\partial_1\sO}$. Suppose that $u\in C^2_s(\underline\sO)\cap C(\bar\sO)$ obeys
$$
Au = 0 \quad\hbox{on } \sO.
$$
If the derivative, $D_{\vec n} u$, with respect to the inward-pointing normal vector field, $\vec n$, is defined everywhere on $\overline{\partial_1\sO}$ and
\begin{equation}
\label{eq:Neumann_boundary_condition_homogeneous}
D_{\vec n}u = 0 \quad\hbox{on } \overline{\partial_1\sO},
\end{equation}
then $u$ is constant on $\sO$. In addition, if $c>0$ at some point in $\underline\sO$, then $u\equiv 0$ on $\sO$.
\end{thm}

\begin{proof}
We modify the proof of \cite[Theorem 3.6]{GilbargTrudinger}. If $u$ is not identically constant on $\sO$, then either $u$ or $-u$ achieves a non-negative maximum $M$ on $\bar\sO$. Since $\sO$ is bounded and $u\in C(\bar\sO)$, we may suppose that $u$ achieves a non-negative maximum at some point $x^0\in\bar\sO$, as the argument when $-u$ achieves a non-negative maximum on $\bar\sO$ will be identical. Therefore, $u(x^0)=M$ for some $x^0\in \overline{\partial_1\sO}$ since, because $u$ is not identically constant on $\sO$, Theorem \ref{thm:Strong_maximum_principle} implies that $u<M$ on $\underline\sO$. But then Lemma \ref{lem:Degenerate_hopf_lemma} implies that $D_{\vec n} u(x^0)< 0$, contradicting our hypothesis \eqref{eq:Neumann_boundary_condition_homogeneous}. Thus, we must have $u=M$, a constant, on $\sO$. If $c>0$ at some point of $\underline\sO$, the facts that $Au=cM$ and $Au=0$ force $M=0$.
\end{proof}

As an immediate consequence, we obtain

\begin{cor}[Uniqueness for the Neumann problem for $C^2$ functions on bounded domains]
\label{cor:Neumann_boundary_condition_uniqueness_equation}
Assume the hypotheses of Theorem \ref{thm:Neumann_boundary_condition_uniqueness_equation} on $\sO$, $\partial\sO$, and $A$.
Require in addition that $c>0$ at some point of $\underline\sO$. Let $f\in C(\sO)$ and $h\in C_{\loc}(\overline{\partial_1\sO};\RR^d)$. If $u_1, u_2 \in C^2_s(\underline\sO)\cap C^1(\sO\cup\overline{\partial_1\sO})\cap C(\bar\sO)$ are solutions to the elliptic equation \eqref{eq:Elliptic_equation} with partial Neumann boundary condition,
\begin{equation}
\label{eq:Neumann_boundary_condition}
D_{\vec n}u = h \quad \hbox{on }\overline{\partial_1\sO},
\end{equation}
then $u_1=u_2$ on $\sO$.
\end{cor}

Finally, we note that a version of the weak maximum principle can be deduced from our strong maximum principle, Theorem \ref{thm:Strong_maximum_principle}, using a proof which is identical to that of \cite[Theorem 2.9]{Feehan_perturbationlocalmaxima}; the result complements our alternative version, Theorem \ref{thm:Weak_maximum_principle_C2}, which has different hypotheses.

\begin{thm}[Weak maximum principle on domains for $A$-subharmonic functions in $C^2(\sO)$]
\label{thm:Weak_maximum_principle_C2_connected_domain}
Let $\sO\subset\RR^d$ be a bounded domain. Assume the hypotheses of Theorem \ref{thm:Strong_maximum_principle} for the coefficients of $A$ in \eqref{eq:Generator} and that
\begin{gather}
\label{eq:Degenerate_boundary_proper_subset}
\partial_0\sO \neq \partial\sO, \quad\hbox{\emph{or}}
\\
\label{eq:Degenerate_boundary_is_whole_boundary}
\tag{\ref*{eq:Degenerate_boundary_proper_subset}$'$}
\hbox{$c>0$ at some point in } \underline\sO.
\end{gather}
Suppose $u \in C^2_s(\underline\sO)$ and $\sup_\sO u<\infty$.
If $Au\leq 0$ on $\sO$ and $u^* \leq 0$ on $\partial\sO\less\partial_0\sO$, then $u\leq 0$ on $\sO$.
\end{thm}

\section{Weak maximum principle for smooth functions}
\label{sec:Weak_maximum_principle_elliptic_C2}
Having considered applications of the weak maximum principle property (Definition \ref{defn:Weak_maximum_principle_property_elliptic_C2_W2d}) to Dirichlet boundary value problems in Section \ref{sec:Applications_weak_maximum_principle_property_boundary_value_problems} and obstacle problems in Section \ref{sec:Applications_weak_maximum_principle_property_obstacle_problems}, we now establish conditions under which the operator $A$ in \eqref{eq:Generator} has the weak maximum principle property on $\underline\sO$, that is, when $\Sigma=\partial_0\sO$. In Section \ref{subsec:Weak_maximum_principle_C2_bounded_domain}, we establish a weak maximum principle for bounded $C^2$ functions on bounded open subsets (Theorem \ref{thm:Weak_maximum_principle_C2}), while in Section \ref{subsec:Weak_maximum_principle_C2_unbounded_opensubset}, we extend this result to the case of bounded $C^2$ functions on unbounded open subsets (Theorem \ref{thm:Weak_maximum_principle_C2_unbounded_opensubset}).

Our weak maximum principle (Theorems \ref{thm:Weak_maximum_principle_C2} and \ref{thm:Weak_maximum_principle_C2_unbounded_opensubset}) differs in several aspects from \cite[Theorem 1.1.2]{Radkevich_2009a}, some of which may appear subtle at first glance but which are nonetheless important for applications. For example,
\begin{enumerate}
\item The function $u$ is \emph{not} required to be in $C^2(\underline\sO)$, but rather  $C^2_s(\underline\sO)$, a strictly weaker condition on regularity up to the boundary portion, $\partial_0\sO$;
\item The subdomain $\sO\subset\RR^d$ is allowed to be \emph{unbounded}; and
\item The coefficients of $A$ in \eqref{eq:Generator} are allowed to be \emph{unbounded}.
\end{enumerate}
The significance of these points is illustrated further by the example of the Heston operator discussed in Appendix \ref{sec:FicheraAndHeston}.

\subsection{Bounded $C^2$ functions on bounded open subsets}
\label{subsec:Weak_maximum_principle_C2_bounded_domain}
We begin with the case of bounded open subsets and adapt the proofs of \cite[Theorem 3.1]{GilbargTrudinger} and \cite[Theorem 2.9.1]{Krylov_LecturesHolder}; see also \cite[Theorem I.3.1]{DaskalHamilton1998}, \cite[Section 3.2]{Feehan_Pop_mimickingdegen_pde}. It will be convenient to adopt the following convention. If $\Sigma\subseteqq \partial\sO$ and $g:\partial\sO\less \Sigma\to\RR$ is a function and $m\in\RR$, then
\begin{equation}
\label{eq:MaxOverEmptyBoundary}
m\vee\sup_{\partial\sO\less \Sigma} g
=
\begin{cases}
\displaystyle{\sup_{\partial\sO\less \Sigma} g} &\hbox{if }\Sigma\subsetneqq \partial\sO,
\\
m &\hbox{if }\Sigma = \partial\sO,
\end{cases}
\end{equation}
where we recall that $x\vee y = \max\{x,y\}$, for any $x,y\in\RR$.

\begin{thm}[Weak maximum principle for $A$-subharmonic functions in $C^2(\sO)$ on bounded open subsets]
\label{thm:Weak_maximum_principle_C2}
Suppose that $\sO\subset\RR^d$ is a bounded open subset. Require that the coefficients of the operator $A$ in \eqref{eq:Generator} be defined everywhere on $\underline\sO$ and obey \eqref{eq:b_perp_nonnegative_boundary_elliptic}, \eqref{eq:c_nonnegative_domain}, \eqref{eq:c_nonnegative_boundary}, where $\partial_0\sO$ is assumed to be $C^{1,\alpha}$, that is, \eqref{eq:C1alpha_degenerate_boundary} holds. Assume further that at least \emph{one} of the following holds: $c$ obeys \eqref{eq:c_positive_domain} and \eqref{eq:c_positive_boundary}, that is, $c>0$ on $\underline\sO$, \emph{or}, for \emph{some} fixed $h\in\RR^d$,
\begin{subequations}
\label{eq:global_uniform_elliptic_direction}
\begin{gather}
\label{eq:global_uniform_elliptic_direction_degenerate_boundary_proper_subset}
{} \partial_0\sO \neq \partial\sO, \quad \hbox{and}
\\
\label{eq:global_uniform_elliptic_direction_b_dot_h_positive_degenerate_boundary}
{} \langle b,h\rangle > 0 \quad\hbox{on }\partial_0\sO, \quad \hbox{and}
\\
\label{eq:global_uniform_elliptic_direction_bounded_ab_ratio_domain}
{} \inf_{\sO}\frac{\langle b,h\rangle}{\langle ah,h\rangle} > -\infty.
\end{gather}
\end{subequations}
Suppose that $u \in C^2_s(\underline\sO)$ obeys $\sup_\sO u < \infty$. If $Au \leq 0$ on $\sO$, then
\begin{equation}
\label{eq:Weak_maximum_principle_bound_C2_c_geq_zero}
\sup_\sO u \leq 0\vee\sup_{\partial\sO\less\partial_0\sO} u^*,
\end{equation}
and, if $c=0$ on $\underline\sO$ and $\partial_0\sO\neq\partial\sO$, then
\begin{equation}
\label{eq:Weak_maximum_principle_bound_C2_c_zero}
\sup_\sO u = \sup_{\partial\sO\less\partial_0\sO} u^*.
\end{equation}
Moreover, $A$ has the \emph{weak maximum principle property on $\underline\sO$} in the sense of Definition \ref{defn:Weak_maximum_principle_property_elliptic_C2_W2d}, with $\Sigma=\partial_0\sO$ and convex cone $\fK = \{u\in C^2_s(\underline\sO): \sup_\sO u < \infty\}$.
\end{thm}

The hypothesis in Theorem \ref{thm:Weak_maximum_principle_C2} that \eqref{eq:C1alpha_degenerate_boundary} holds is required by the change-of-coordinates argument used in Step \ref{step:weak_maximum_principle_C2_c_positive} of the proof.

\begin{rmk}[On the hypothesis of a global uniformly elliptic direction]
\label{rmk:global_uniform_elliptic_direction}
The hypothesis of a global `uniformly elliptic direction', $h\in\RR^d$, in the statement of Theorem \ref{thm:Weak_maximum_principle_C2} arises in Step \ref{step:weak_maximum_principle_C2_c_zero} of our proof; compare the proof of \cite[Theorem 3.1]{GilbargTrudinger} and \cite{Monticelli_Payne_2009}. Although unattractive, Remark \ref{rmk:EllipticHestonC2Bounded} indicates that this hypothesis is not unduly restrictive in applications since, typically, we can take $h=\vec n(x^0)$ for some $x^0\in\partial_0\sO$. Nevertheless, versions of Theorem \ref{thm:Weak_maximum_principle_C2} which omit this condition, at the expense of imposing slightly stronger, local conditions on the coefficients of $A$ are given in \cite{Feehan_perturbationlocalmaxima}. Theorem \ref{thm:Weak_maximum_principle_C2_connected_domain}, which is deduced as a consequence of our strong maximum principle, Theorem \ref{thm:Strong_maximum_principle}, illustrates another collection of hypotheses which yields a weak maximum principle.
\end{rmk}

\begin{rmk}[Application to the elliptic Heston operator]
\label{rmk:EllipticHestonC2Bounded}
The hypotheses of Theorem \ref{thm:Weak_maximum_principle_C2} on the coefficients of $A$ are obeyed in the case of the elliptic Heston operator, Example \ref{exmp:HestonPDE}, where $d=2$ and $\sO\subset\HH$ is bounded and $\Sigma=\bar\sO\cap\partial\HH$. Choosing $h=e_2$ in condition \eqref{eq:global_uniform_elliptic_direction}, we have $b^2(x) = \kappa(\theta-x_2)$ and $a^{22}(x) = \sigma^2 x_2/2$, so
$$
\frac{\langle b,e_2\rangle}{\langle ae_2,e_2\rangle} = \frac{b^2(x)}{a^{22}(x)} = \frac{2\kappa(\theta-x_2)}{\sigma^2 x_2} =  \frac{2\kappa\theta}{\sigma^2 x_2} - \frac{2\kappa}{\sigma^2} > -\frac{2\kappa}{\sigma^2}, \quad \forall\, x_2>0,
$$
and condition \eqref{eq:b_perp_nonnegative_boundary_elliptic} is obeyed (with $\vec n=h$) as $\langle b,h\rangle = b^2(x_1,0) = \kappa\theta > 0$. Lastly, $c=r\geq 0$.
\end{rmk}

\begin{proof}[Proof of Theorem \ref{thm:Weak_maximum_principle_C2}]
Since $u^*$ is upper semicontinuous on $\bar\sO$ and by hypothesis $\bar\sO$ is compact, then $u^*$ achieves its maximum value at some point $x^0$ in $\bar\sO$.

\setcounter{step}{0}
\begin{step}[$c>0$ on $\underline\sO$]
\label{step:weak_maximum_principle_C2_c_positive}
Suppose that $u^*$ attains its maximum value, $u^*(x^0)=u(x^0)$, at a point $x^0$ in $\underline\sO$. If $x^0 \in \sO$, then $Du(x^0) = 0$ and $D^2u(x^0) \leq 0$, so that
\begin{align*}
Au(x^0) &= -a^{ij}(x^0)u_{x_ix_j}(x^0) - b^i(x^0)u_{x_i}(x^0) + c(x^0)u(x^0)
\\
&\geq c(x^0)u(x^0),
\end{align*}
since  $a(x^0)$ is non-negative definite by \eqref{eq:a_nonnegative_symmetric_matrix_valued} and $\tr(KL)\geq 0$ whenever $K,L$ are non-negative definite matrices \cite[p. 218]{Lancaster_Tismenetsky}. If $u(x^0)>0$, we would obtain $Au(x^0)>0$, contradicting our hypothesis that $Au \leq 0$ on $\sO$. Therefore, we must have $u(x^0)\leq 0$ and, necessarily, $u(x^0)\leq 0\vee\sup_{\partial\sO\less\partial_0\sO}u^*$ or simply $u(x^0)\leq 0$ if $\partial_0\sO=\partial\sO$.

If $x^0\in \partial_0\sO$ then, possibly after a $C^2$ change of coordinates on $\RR^d$, we may assume without loss of generality by our hypothesis \eqref{eq:C1alpha_degenerate_boundary} and \cite[Lemma B.1]{Feehan_perturbationlocalmaxima} that $B(x^0)\cap\underline\sO=B(x^0)\cap\HH$, where $\HH=\RR^{d-1}\times\RR_+$. Thus, $\vec n(x^0) = e_d$ and the condition \eqref{eq:b_perp_nonnegative_boundary_elliptic} at $x^0$ becomes $b^d(x^0)\geq 0$. We have $\tr(a D^2u(x^0)) = 0$ by \eqref{eq:Second_order_boundary_vanishing} and the fact that $u\in C^2_s(\underline\sO)$, while $u_{x_i}(x^0) = 0$ for $1\leq i\leq d-1$ and $u_{x_d}(x^0) \leq 0$ since $x^0\in\partial_0\sO$ is a local maximum and $u\in C^1(\underline\sO)$, so that
\begin{align*}
Au(x^0) &= -a^{ij}(x^0)u_{x_ix_j}(x^0) - b^i(x^0)u_{x_i}(x^0) + c(x^0)u(x^0)
\\
&= -b^d(x^0)u_{x_d}(x^0) + c(x^0)u(x^0)
\\
&\geq c(x^0)u(x^0).
\end{align*}
Our hypothesis that $Au \leq 0$ on $\sO$ implies, by continuity of $Au$ on $\underline\sO$ due to $u\in C^2_s(\underline\sO)$ and \eqref{eq:Second_order_boundary_continuity}, that $Au \leq 0$ on $\underline\sO$. If $u(x^0)>0$, we would obtain $Au(x^0)>0$, a contradiction. Therefore, we must again have $u(x^0)\leq 0$ and, necessarily, $u(x^0)\leq 0\vee\sup_{\partial\sO\less\partial_0\sO}u^*$ or simply $u(x^0)\leq 0$ if $\partial_0\sO=\partial\sO$.

Finally, if $x^0\in \partial\sO\less\partial_0\sO$, then $u(x^0)=\sup_{\partial\sO\less\partial_0\sO}u^*$ and so by combining the preceding three cases we obtain \eqref{eq:Weak_maximum_principle_bound_C2_c_geq_zero} for this step.
\end{step}

\begin{step}[$c=0$ on $\underline\sO$ and $\partial_0\sO\neq\partial\sO$ and $Au < 0$ on $\underline\sO$]
\label{step:weak_maximum_principle_C2_c_zero_Au_negative}
Suppose that $u^*$ attains its maximum value, $u^*(x^0)=u(x^0)$, at a point $x^0$ in $\underline\sO$. Repeating the argument of Step \ref{step:weak_maximum_principle_C2_c_positive} would then yield
$$
Au(x^0) \geq c(x^0)u(x^0) = 0,
$$
contradicting the assumption that $Au < 0$ on $\underline\sO$. Consequently, we must have $x^0\in\partial\sO\less\partial_0\sO$ and
$$
\sup_\sO u = \sup_{\partial\sO\less\partial_0\sO} u^*,
$$
that is, \eqref{eq:Weak_maximum_principle_bound_C2_c_zero} holds. See additional comments at the end of the proof.
\end{step}

\begin{step}[$c=0$ on $\underline\sO$ and $\partial_0\sO\neq\partial\sO$]
\label{step:weak_maximum_principle_C2_c_zero}
We see that, for a constant $\nu>0$ to be chosen and any $x\in\sO$,
\begin{align*}
Ae^{\nu\langle h,x\rangle} &= \left(-\nu^2a^{ij}(x)h_ih_j - \nu b^i(x)h_i\right)e^{\nu\langle h,x\rangle}
\\
&= -\nu a^{ij}(x)h_ih_j\left(\nu + \frac{b^i(x)h_i}{a^{ij}(x)h_ih_j}\right)e^{\nu\langle h,x\rangle}
\\
&< 0 \quad\hbox{on }\sO, \quad\hbox{if } \inf_{x\in\sO}\frac{b^i(x)h_i}{a^{ij}(x)h_ih_j} > -\nu,
\end{align*}
appealing to \eqref{eq:global_uniform_elliptic_direction_bounded_ab_ratio_domain} for the ability to choose a finite $\nu>0$. For any $x\in\partial_0\sO$, we have
\begin{align*}
Ae^{\nu\langle h,x\rangle} &= -\nu b^i(x)h_i e^{\nu\langle h,x\rangle}
\\
&< 0 \quad\hbox{on }\partial_0\sO, \quad\hbox{if } b^ih_i > 0\quad\hbox{on }\partial_0\sO,
\end{align*}
by appealing to \eqref{eq:global_uniform_elliptic_direction_b_dot_h_positive_degenerate_boundary}. Therefore, we have, for any $\eps>0$,
$$
A(u + \eps e^{\nu\langle h,x\rangle}) < 0 \quad\hbox{on }\underline\sO,
$$
and so, by Step \ref{step:weak_maximum_principle_C2_c_zero_Au_negative}, we obtain
$$
\sup_\sO(u + \eps e^{\nu\langle h,x\rangle}) = \sup_{\partial_1\sO}(u + \eps e^{\nu\langle h,x\rangle}).
$$
Letting $\eps\to 0$ yields \eqref{eq:Weak_maximum_principle_bound_C2_c_zero} for this step.
\end{step}

\begin{step}[$c\geq 0$ on $\underline\sO$ and $\partial_0\sO\neq\partial\sO$]
\label{step:weak_maximum_principle_C2_c_geq_zero}
Let $\sO^+$ denote the open subset $\{x\in \sO: u(x)>0\}\subset \sO$. If $\sO^+$ is empty, then $u\leq 0$ on $\sO$ and so $u^*\leq 0$ on $\bar\sO$, with
$$
\sup_\sO u \leq 0 = 0\vee\sup_{\partial\sO\less\partial_0\sO}u^*.
$$
It remains to consider the case where $\sO^+$ is non-empty. By hypothesis, $Au \leq 0$ on $\sO$ and thus $A_0u \leq -cu \leq 0$ on $\sO^+$, where $A_0 := A-c$. We may write
$$
\partial\sO^+ = \left(\partial_0\sO\cap\partial\sO^+\right) \cup \left(\partial\sO^+\less\partial_0\sO\right).
$$
If $\partial\sO^+\less\partial_0\sO$ were empty, then we would have $\partial\sO^+\subset\partial_0\sO\subset\partial\sO$. Thus, we would necessarily have $\sO^+=\sO$ and hence $\partial\sO^+=\partial_0\sO=\partial\sO$, contradicting our assumption that $\partial_0\sO\neq\partial\sO$. Therefore, $\partial\sO^+\less\partial_0\sO$ must be non-empty and our result \eqref{eq:Weak_maximum_principle_bound_C2_c_zero} for the case $c=0$ on $\underline\sO$ yields
$$
\sup_{\sO^+} u = \sup_{\partial\sO^+\less\partial_0\sO}u.
$$
Now
$$
\partial\sO^+\less\partial_0\sO = \left(\partial\sO^+ \cap \partial\sO\less\partial_0\sO\right) \cup \left(\partial\sO^+ \cap\sO\right).
$$
Since $u=0$ on $\partial\sO^+ \cap\sO$ (because $u$ is continuous on $\sO$ and $u\leq 0$ on $\sO\less\sO^+$), then
$$
\sup_{\partial\sO^+\less\partial_0\sO}u = 0\vee \sup_{\partial\sO^+ \cap \partial\sO\less\partial_0\sO}u = \sup_{\partial\sO\less\partial_0\sO}u,
$$
and so, combining the preceding inequalities and equations and noting that $\sup_\sO u = \sup_{\sO^+} u$, we obtain
$$
\sup_\sO u = \sup_{\partial\sO\less\partial_0\sO}u,
$$
as desired for the case of non-empty $\sO^+$. Combining the preceding cases yields \eqref{eq:Weak_maximum_principle_bound_C2_c_geq_zero} for this step.
\end{step}
This completes the proof.
\end{proof}

If $\partial_0\sO=\partial\sO$, then the conditions $Au < 0$ on $\partial_0\sO$ and $\tr(aD^2u)=0$ on $\partial_0\sO$ via \eqref{eq:Second_order_boundary_vanishing} and $c=0$ on $\partial_0\sO$ imply that $-\langle b, Du\rangle <0$ on $\partial_0\sO$. If in addition we had $b^\parallel=0$ and \eqref{eq:b_perp_nonnegative_boundary_elliptic} were strengthened to $\langle b,\vec n\rangle>0$ on $\partial_0\sO$, then we would obtain $D_{\vec n} u>0$ on $\partial_0\sO$ and consequently $u$ would have a local maximum in $\sO$, contradicting the condition $Au < 0$ on $\sO$. However, in general when $\partial_0\sO=\partial\sO$, one obtains no additional information regarding $\sup_\sO u$ when $c=0$ on $\underline\sO$, beyond the fact that $u$ achieves its maximum at some point of $\underline\sO=\bar\sO$.

\subsection{Bounded $C^2$ functions on unbounded open subsets}
\label{subsec:Weak_maximum_principle_C2_unbounded_opensubset}
Next, we consider the case of bounded $C^2$ functions on \emph{unbounded}  open subsets. We have the following refinement of the maximum principle for bounded $C^2$ functions on unbounded open subsets and elliptic operators with non-negative definite characteristic form \cite[Theorem 2.9.2 and Exercises 2.9.4, 2.9.5]{Krylov_LecturesHolder}.

\begin{thm}[Weak maximum principle for bounded functions in $C^2(\sO)$ on unbounded open subsets]
\label{thm:Weak_maximum_principle_C2_unbounded_opensubset}
Suppose that $\sO\subseteqq\RR^d$ is a possibly unbounded open subset. Assume that the coefficients of $A$ in \eqref{eq:Generator} obey the hypotheses of Theorem \ref{thm:Weak_maximum_principle_C2}, except that the conditions \eqref{eq:c_positive_domain}, \eqref{eq:global_uniform_elliptic_direction} on $a,b,c$ are replaced by the condition \eqref{eq:c_positive_lower_bound_domain} that $c\geq c_0$ on $\sO$ for some positive constant, $c_0$, though we keep \eqref{eq:c_positive_boundary}, that is, $c>0$ on $\partial_0\sO$. In addition, we require that there is a positive constant, $K$, such that \eqref{eq:Quadratic_growth} holds for $a,b$. If $u \in C^2_s(\underline\sO)$ obeys $\sup_\sO u < \infty$ and $u^*\leq 0$ on $\partial\sO\less \partial_0\sO$ (when non-empty), then
$$
\sup_\sO u \leq 0\vee\frac{1}{c_0}\sup_\sO Au.
$$
Moreover, $A$ has the \emph{weak maximum principle property on $\underline\sO$} in the sense of Definition \ref{defn:Weak_maximum_principle_property_elliptic_C2_W2d}, with $\Sigma=\partial_0\sO$ and convex cone $\fK = \{u\in C^2_s(\underline\sO): \sup_\sO u < \infty\}$.
\end{thm}

\begin{proof}
Let
\begin{equation}
\label{eq:Defnvzero}
v_0(x) := 1+|x|^2, \quad\forall\, x \in\RR^d,
\end{equation}
and observe, noting that $c\geq 0$ on $\sO$ by hypothesis \eqref{eq:c_positive_lower_bound_domain}, that
\begin{align*}
Av_0(x) &= -2\tr a(x) - 2\langle b(x),x\rangle + c(x)\left(1+|x|^2\right)
\\
&\geq -2Kv_0(x), \quad\forall\, x\in\sO,
\end{align*}
where $K$ is the constant in \eqref{eq:Quadratic_growth}. Therefore,
\begin{equation}
\label{eq:Aplus2KvgeqZero}
(A+2K)v_0 \geq 0 \quad\hbox{on } \sO.
\end{equation}
Define
$$
M := 0\vee\sup_\sO(A+2K)u.
$$
If $M=+\infty$, there is nothing to prove, so we may assume $0\leq M<\infty$. For $\delta>0$, set
\begin{equation}
\label{eq:Defnw}
w := u - \delta v_0 - (c_0+2K)^{-1}M,
\end{equation}
and observe that, using $c\geq c_0\geq 0$ on $\sO$,
\begin{align*}
(A+2K)w &= (A+2K)u - \delta(A+2K)v_0 - (c+2K)(c_0+2K)^{-1}M
\\
&\leq (A+2K)u - M
\\
&\leq 0 \quad\hbox{on } \sO.
\end{align*}
Denoting $B(R)=\{x\in\RR^d: |x|<R\}$, then for all large enough $R > 0$, we have $w \leq 0$ on $\sO\cap\partial B(R)$ since $u$ is bounded above on $\sO$ by hypothesis. Also, $u^*\leq 0$ on $\partial\sO\less \partial_0\sO$ (when non-empty) by hypothesis and so $w^*\leq 0$ on $\partial\sO\less \partial_0\sO$ (when non-empty). But
$$
\partial(\sO\cap B(R)) = (\bar\sO\cap\partial B(R)) \cup (\bar B(R)\cap\partial\sO),
$$
and therefore,
$$
\partial(\sO\cap B(R))\less\partial_0\sO = (\bar\sO\cap\partial B(R)) \cup (\bar B(R)\cap\partial\sO\less\partial_0\sO),
$$
so that
$$
w^* \leq 0 \quad\hbox{on } \partial(\sO\cap B(R))\less\partial_0\sO.
$$
Consequently, Theorem \ref{thm:Weak_maximum_principle_C2} (with $\sO$ replaced by $\sO\cap B(R)$ and $A$ by $A+2K$) implies that $w \leq 0$ on $\sO\cap B(R)$, for any sufficiently large $R>0$ and thus $w \leq 0$ on $\sO$. By letting $\delta \to 0$, we obtain $u\leq (c_0+2K)^{-1}M$ on $\sO$ and so
\begin{align*}
\sup_\sO u^+ &\leq \frac{M}{c_0+2K}
\\
&= \frac{1}{c_0+2K}\sup_\sO((A+2K)u)^+
\\
&\leq \frac{1}{c_0+2K}\sup_\sO(Au)^+ + \frac{2K}{c_0+2K}\sup_\sO u^+,
\end{align*}
and thus
$$
\frac{c_0}{c_0+2K}\sup_\sO u^+ \leq  \frac{1}{c_0+2K}\sup_\sO(Au)^+,
$$
which, using $c_0>0$, completes the proof.
\end{proof}

\begin{rmk}[Application to the elliptic Heston operator]
\label{rmk:EllipticHestonC2UnboundedDomain}
The hypotheses \eqref{eq:c_positive_lower_bound_domain},  \eqref{eq:c_positive_boundary}, and \eqref{eq:Quadratic_growth} in Theorem \ref{thm:Weak_maximum_principle_C2_unbounded_opensubset} are obeyed in the case of the elliptic Heston operator (Example \ref{exmp:HestonPDE}) with
$$
a(x) = \frac{x_2}{2}\begin{pmatrix}1 & \varrho\sigma \\ \varrho\sigma & \sigma^2 \end{pmatrix},
\quad
b(x) = \begin{pmatrix} r-q-\displaystyle\frac{x_2}{2} & \kappa(\theta-x_2)\end{pmatrix},
\quad\hbox{and}\quad c(x) = r,
$$
provided $r>0$.
\end{rmk}

\part{Weak maximum principles for bilinear maps and operators on functions in Sobolev spaces and applications to variational equations and inequalities}
\label{part:VariationalEquationInequality}
In this part of our article (sections \ref{sec:ApplicationsWeakMaxPrinPropertyVarEq}, \ref{sec:ApplicationsWeakMaxPrinPropertyVarIneq}, and \ref{sec:WeakMaxPrincipleSobolev}), we develop weak maximum principles for bilinear maps and operators on functions in Sobolev spaces and applications to variational equations and inequalities.

\section{Applications of the weak maximum principle property to variational equations}
\label{sec:ApplicationsWeakMaxPrinPropertyVarEq}
Just as in the case of the weak maximum principle for linear, second-order, partial differential operators $A$ in \eqref{eq:Generator} with non-negative definite characteristic form acting on smooth functions, we shall encounter many different situations (bounded or unbounded open subsets $\sO\subset\RR^d$, bounded or unbounded functions $u$ with prescribed growth, and so on) where the basic weak maximum principle holds for bilinear maps $\fa$ on $H^1(\sO,\fw)$ or associated linear operators $A:H^2(\sO,\fw) \to L^2(\sO,\fw)$. Again, we find it useful to isolate that key property and then derive the consequences which necessarily follow in an essentially formal manner. In this section, we consider applications to variational equations. After providing some technical preliminaries, we proceed to the main applications, including the comparison principle (Proposition \ref{prop:ComparisonUniquenessVariationalBVProblem}) and \apriori estimates (Proposition \ref{prop:H1WeakMaxPrincipleAprioriEstimates}) for $H^1(\sO,\fw)$
supersolutions, subsolutions, and solutions to variational equations, and the corresponding results (Proposition \ref{prop:H2WeakMaxPrincipleAprioriEstimates}) for $H^2(\sO,\fw)$ supersolutions, subsolutions, and solutions to the associated boundary value problems. Finally, we show that when a bilinear map $\fa$ on $H^1(\sO,\fw)$ or operator $A$ on $H^2(\sO,\fw)$ has a weak maximum principle property for subsolutions (on unbounded open subsets) which are bounded above, the property may extend to subsolutions which instead obey a growth condition (Theorem \ref{thm:Weak_maximum_principle_H1_unbounded_function} and Corollary \ref{cor:Weak_maximum_principle_H2_unbounded_function}).

\begin{defn}[Weight function]
\label{defn:WeightFunction}
Let $\sO\subseteqq\RR^d$ be an open subset. We call $\fw$ a \emph{weight} function if
\begin{equation}
\label{eq:WeightFunctionCondition}
\fw \in C(\sO)\cap L^1(\sO) \quad\hbox{and}\quad \fw > 0 \quad\hbox{on }\sO.
\end{equation}
\end{defn}

\begin{defn}[Weighted Sobolev spaces]
\label{defn:SobolevSpaces}
Given weight functions $\fw_{k,i}$, for integers $0\leq i\leq k$, we define Hilbert spaces with norms
$$
\|u\|_{H^k(\sO,\fw)}^2 := \sum_{i=0}^k \int_\sO |D^iu|^2\fw_{k,i}\,dx, \quad k\geq 1,
$$
as the completions of the vector space $C^\infty_0(\bar\sO)$ with respect to the preceding norms and denote $L^2(\sO,\fw) := H^0(\sO,\fw)$. Given a relatively open subset $\Sigma\subseteqq\partial\sO$, we define $H^1_0(\sO\cup \Sigma,\fw)$ to be the closure of $C^\infty_0(\sO\cup \Sigma)$ in $H^1(\sO,\fw)$.
\end{defn}

When $\Sigma=\partial\sO$ in Definition \ref{defn:SobolevSpaces}, then $\sO\cup \Sigma=\bar\sO$ and $C^\infty_0(\sO\cup \Sigma)=C^\infty_0(\bar\sO)$ and $H^1_0(\sO\cup \Sigma,\fw) = H^1(\sO,\fw)$.
By analogy with \cite[Section 8.1]{GilbargTrudinger}, for $u\in H^1(\sO,\fw)$, we define
\begin{equation}
\label{eq:BoundarySupH1Function}
\sup_{\partial\sO\less\Sigma} u := \inf\left\{l\in\RR: u \leq l \hbox{ on $\partial\sO\less\Sigma$ in the sense of $H^1(\sO,\fw)$}\right\},
\end{equation}
where we recall that $u \leq l$ on $\partial\sO\less\Sigma$ in the sense of $H^1(\sO,\fw)$ if $(u-l)^+ \in H^1_0(\sO\cup\Sigma,\fw)$. If $g\in H^1(\sO,\fw)$ and $m\in\RR$, then
\begin{equation}
\label{eq:EssSupOverEmptyBoundary}
m\vee\sup_{\partial\sO\less \Sigma} g
:=
\begin{cases}
\displaystyle{\sup_{\partial\sO\less \Sigma} g} &\hbox{if }\Sigma\subsetneqq \partial\sO,
\\
m &\hbox{if }\Sigma = \partial\sO,
\end{cases}
\end{equation}
by analogy with our convention \eqref{eq:MaxOverEmptyBoundary} for everywhere-defined functions.

By analogy with \cite[Section 5.9.1]{Evans}, we let
\begin{equation}
\label{eq:DualSpace}
H^{-1}(\sO\cup\Sigma,\fw) := (H^1_0(\sO\cup\Sigma,\fw))'
\end{equation}
denote the dual space and, as in \cite[Section 5.9.1]{Evans}, observe that
$$
H^1_0(\sO\cup\Sigma,\fw) \subset L^2(\sO,\fw) \subset H^{-1}(\sO\cup\Sigma,\fw).
$$
The proofs of \cite[Theorem 3.8]{Adams_1975} or \cite[Theorem 5.9.1]{Evans} easily adapt to show that every $F\in H^{-1}(\sO\cup\Sigma,\fw)$ has the form
\begin{equation}
\label{eq:DualSpaceRepresentation}
F(v) = (f,v)_{L^2(\sO,\fw)} - (f^i,v_{x_i})_{L^2(\sO,\fw)}, \quad\forall\, v\in H^1_0(\sO\cup\Sigma,\fw),
\end{equation}
and we write $F=(f,f^1,\ldots,f^d)$, where $f, f^i \in L^2(\sO,\fw)$, $1\leq i\leq d$. We say that $F\leq 0$ when
$$
F(v) \leq 0, \quad \forall\, v \in H^1_0(\sO\cup\Sigma,\fw), v\geq 0 \hbox{ a.e. on } \sO,
$$
and, given $F_1,F_2\in H^{-1}(\sO\cup\Sigma,\fw)$, we say that $F_1\leq F_2$ if $F_1-F_2\leq 0$.

\begin{defn}[Variational solution, subsolution, and supersolution]
\label{defn:WeakSolution}
Let $F=(f,f^1,\ldots,f^d) \in H^{-1}(\sO\cup\Sigma,\fw)$ and $g \in H^1(\sO,\fw)$. We define $u \in H^1(\sO,\fw)$ to be a \emph{variational subsolution},
\begin{align*}
\fa(u,\cdot) &\leq F,
\\
u&\leq g \quad\hbox{on $\partial\sO\less\Sigma$ in the sense of } H^1(\sO,\fw),
\end{align*}
if
\begin{align}
\label{eq:VariationalSubsolution}
\fa(u,v) &\leq F(v), \quad \forall\, v \in H^1_0(\sO\cup\Sigma,\fw), v\geq 0 \hbox{ a.e. on } \sO,
\\
\label{eq:VariationalSubsolutionBC}
\quad (u-g)^+ &\in H^1_0(\sO\cup\Sigma,\fw).
\end{align}
We call $u \in H^1(\sO,\fw)$ a \emph{variational supersolution} if $-u$ is a variational subsolution and call $u \in H^1(\sO,\fw)$ a \emph{variational solution} if it is both a variational subsolution and supersolution.
\end{defn}

The first application of the weak maximum principle property, as in the case of Proposition \ref{prop:Comparison_principle_elliptic_boundary_value_problem_C2_W2d}, is to settle the question of \emph{uniqueness}.

\begin{prop}[Comparison principle and uniqueness for variational solutions]
\label{prop:ComparisonUniquenessVariationalBVProblem}
Let $\fa$ be a bilinear map on $H^1(\sO,\fw)$ obeying the weak maximum principle property on $\sO\cup\Sigma$ for $\fK$, for some $\Sigma\subseteqq\partial\sO$ and convex cone $\fK\subset H^1(\sO,\fw)$. Suppose that $u, -v\in \fK$. If $\fa(u,\cdot) \leq \fa(v,\cdot)$ and $u\leq v$ on $\partial\sO\less\Sigma$ in the sense of $H^1(\sO,\fw)$, then $u \leq v$ a.e. on $\sO$; if $\fa(u,\cdot) = \fa(v,\cdot)$ and $u = v$ on $\partial\sO\less\Sigma$ in the sense of $H^1(\sO,\fw)$, then $u = v$ a.e. on $\sO$.
\end{prop}

\begin{proof}
We have $u-v \in \fK$ and, in the case of the inequality, we have $\fa(u-v,\cdot) \leq 0$ and $u-v \leq 0$ on $\partial\sO\less\Sigma$ in the sense of $H^1(\sO,\fw)$, and thus
$$
u-v \leq 0 \quad\hbox{a.e. on }\sO,
$$
because $\fa$ has  weak maximum principle property on $\sO\cup\Sigma$ for $\fK$. In the case of the equality, then we also obtain $v-u \leq 0$ a.e. on $\sO$ and thus $v = u$ a.e. on $\sO$.
\end{proof}

We shall occasionally need the following analogues of the conditions \eqref{eq:c_nonnegative_domain} and \eqref{eq:c_positive_lower_bound_domain}, respectively; compare \cite[Equation (8.8)]{GilbargTrudinger}:
\begin{align}
\label{eq:BilinearMapNonnegLowerBound}
\fa(1,v) &\geq 0, \quad\hbox{or }
\\
\label{eq:BilinearMapPositiveLowerBound}
\tag{\ref*{eq:BilinearMapNonnegLowerBound}$'$}
\fa(1,v) &\geq (c_0,v)_{L^2(\sO,\fw)},
\end{align}
for all $v\in H^1_0(\sO\cup\Sigma,\fw)$ such that $v\geq 0$ a.e. on $\sO$, and some constant $c_0>0$. We can now proceed to give the expected \apriori estimates.

\begin{prop}[Weak maximum principle and \apriori estimates for functions in $H^1(\sO,\fw)$]
\label{prop:H1WeakMaxPrincipleAprioriEstimates}
Let $\fa$ be a bilinear map on $H^1(\sO,\fw)$ obeying the weak maximum principle property on $\sO\cup\Sigma$, for some relatively open subset $\Sigma\subseteqq\partial\sO$ and convex cone $\fK\subset H^1(\sO,\fw)$ containing the constant function $1$, and assume that \eqref{eq:BilinearMapNonnegLowerBound} holds. Suppose that $u, -v\in \fK$.
\begin{enumerate}
\item\label{item:SubsolutionfleqZero} If $\fa(u,\cdot)\leq 0$, then
$$
u\leq 0 \vee \sup_{\partial\sO\less\Sigma}u \quad\hbox{a.e. on }\sO.
$$
\item\label{item:Subsolutionfarbsign} If $\fa(u,\cdot)\leq (f,\cdot)_{L^2(\sO,\fw)}$, where $f$ has arbitrary sign but there is a constant $c_0>0$ such that $\fa$ obeys \eqref{eq:BilinearMapPositiveLowerBound}, then
$$
u\leq 0 \vee \frac{1}{c_0}\esssup_\sO f \vee \sup_{\partial\sO\less\Sigma}u \quad\hbox{a.e. on }\sO.
$$
\item\label{item:SupersolutionfgeqZero} If $\fa(v,\cdot)\geq 0$, then
$$
v\geq 0 \wedge \inf_{\partial\sO\less\Sigma}v \quad\hbox{a.e. on }\sO.
$$
\item\label{item:Supersolutionfarbsign} If $\fa(v,\cdot)\geq (f,\cdot)_{L^2(\sO,\fw)}$, where $f$ has arbitrary sign and  $\fa$ obeys \eqref{eq:BilinearMapPositiveLowerBound}, then
$$
v\geq 0 \wedge \frac{1}{c_0}\essinf_\sO f \wedge \inf_{\partial\sO\less\Sigma}v \quad\hbox{a.e. on }\sO.
$$
\item\label{item:SolutionfZero} If $u\in \fK\cap -\fK$ and $\fa(u,\cdot) = 0$, then
$$
\|u\|_{L^\infty(\sO)} \leq \|u\|_{L^\infty(\partial\sO\less\Sigma)}.
$$
\item\label{item:Solutionfarbsign} If $u\in \fK\cap -\fK$ and $\fa(u,\cdot) = (f,\cdot)_{L^2(\sO,\fw)}$, where $f$ has arbitrary sign and $\fa$ obeys \eqref{eq:BilinearMapPositiveLowerBound}, then
$$
\|u\|_{L^\infty(\sO)} \leq \frac{1}{c_0}\|f\|_{L^\infty(\sO)}\vee\|u\|_{L^\infty(\partial\sO\less\Sigma)}.
$$
\end{enumerate}
The terms $\sup_{\partial\sO\less\Sigma}u$, and $\inf_{\partial\sO\less\Sigma}v$, and $\|u\|_{L^\infty(\partial\sO\less\Sigma)}$ in the preceding items are omitted when $\Sigma=\partial\sO$.
\end{prop}

\begin{proof}
The proof follows almost the same pattern as that of Proposition \ref{prop:Elliptic_weak_maximum_principle_apriori_estimates_C2_W2d}. For Items \eqref{item:SubsolutionfleqZero} and \eqref{item:Subsolutionfarbsign}, we describe the proof when $\Sigma\subsetneqq\partial\sO$; the proof for the case $\Sigma=\partial\sO$ is the same except that a supremum of a non-negative function over $\partial\sO\less\Sigma$ is replaced by zero. When $f$ has arbitrary sign, choose
$$
M :=  0 \vee \frac{1}{c_0}\sup_\sO f \vee \sup_{\partial\sO\less\Sigma}u,
$$
while if $f\leq 0$ a.e. on $\sO$, choose
$$
M :=  0 \vee \sup_{\partial\sO\less\Sigma}u.
$$
We may assume without loss of generality that $M<\infty$. Since $M\geq u$ on $\partial\sO\less\Sigma$ (in the sense of $H^1(\sO,\fw)$) and $c_0M\geq f$ a.e. on $\sO$, then, for all $w\in H^1_0(\sO\cup\Sigma,\fw)$ with $w\geq 0$ a.e. on $\sO$, the condition \eqref{eq:BilinearMapPositiveLowerBound} gives
$$
\fa(M,w) \geq (c_0M,w)_{L^2(\sO,\fw)} \geq (f,w)_{L^2(\sO,\fw)}, \quad \forall\, w \in H^1_0(\sO\cup\Sigma,\fw), w\geq 0 \hbox{ a.e. on } \sO,
$$
when $f$ has arbitrary sign and, when $f\leq 0$, the condition \eqref{eq:BilinearMapNonnegLowerBound} gives
$$
\fa(M,w) \geq 0 \geq (f,w)_{L^2(\sO,\fw)}, \quad \forall\, w \in H^1_0(\sO\cup\Sigma,\fw), w\geq 0 \hbox{ a.e. on } \sO.
$$
Hence, $\fa(u,\cdot) \leq \fa(M,\cdot)$ and $u\leq M$ on $\partial\sO\less\Sigma$, and thus $u\leq M$ a.e. on $\sO$ by Proposition \ref{prop:ComparisonUniquenessVariationalBVProblem}, which gives Items \eqref{item:SubsolutionfleqZero} and \eqref{item:Subsolutionfarbsign}.

Items \eqref{item:SupersolutionfgeqZero} and \eqref{item:Supersolutionfarbsign} follow from Items \eqref{item:SubsolutionfleqZero} and \eqref{item:Subsolutionfarbsign} by writing $v=-u$. Item \eqref{item:SolutionfZero} follows by combining Items \eqref{item:SubsolutionfleqZero} and \eqref{item:SupersolutionfgeqZero}, while Item \eqref{item:Solutionfarbsign} follows by combining Items \eqref{item:Subsolutionfarbsign} and \eqref{item:Supersolutionfarbsign}.
\end{proof}

The \apriori estimate in Item \eqref{item:Solutionfarbsign} may be compared with \cite[Theorem 1.5.1 and 1.5.5]{Radkevich_2009a} and \cite[Lemma 2.8]{Troianiello}. We have the following analogue of Definitions \ref{defn:Weak_maximum_principle_property_elliptic_C2_W2d} and \ref{defn:WeakMaxPrinciplePropertyBilinearForm}.

\begin{defn}[Weak maximum principle property for a linear operator]
\label{defn:WeakMaxPrinciplePropertyElliptic}
Let $A: H^2(\sO,\fw) \to L^2(\sO,\fw)$ be a linear operator, let $\Sigma\subseteqq\partial\sO$, and let $\fK\subset H^2(\sO,\fw)$ be a convex cone. We say that $A$ obeys the \emph{weak maximum principle property on $\sO\cup\Sigma$ for $\fK$} if whenever $u\in \fK$ obeys
$$
\begin{cases}
Au \leq 0 \hbox{ a.e. on }\sO,
\\
u\leq 0 \hbox{ on }\partial\sO\less\Sigma\hbox{ in the sense of $H^1(\sO,\fw)$},
\end{cases}
$$
then $u \leq 0$ a.e. on $\sO$.
\end{defn}

Typically, $\fK = H^2(\sO,\fw)$ when $\sO$ is bounded or $\fK = \{u\in H^2(\sO,\fw): \esssup_\sO u < \infty\}$ when $\sO$ is unbounded, though we may also choose $\fK$ to be the set of functions $u\in H^2(\sO,\fw)$ such that $u^+$ obeys a growth condition, as in Corollary \ref{cor:Weak_maximum_principle_H2_unbounded_function}.

\begin{defn}[Integration by parts]
\label{defn:IntegrationByParts}
Let $A: H^2(\sO,\fw) \to L^2(\sO,\fw)$ be a linear operator and let $\Sigma\subseteqq\partial\sO$. We say that a bilinear map $\fa:H^1(\sO,\fw)\times H^1(\sO,\fw) \to \RR$ is associated to $A$ on $\sO\cup\Sigma$ through integration by parts if
\begin{equation}
\label{eq:IntegrationByParts}
(Au,v)_{L^2(\sO,\fw)} = \fa(u,v), \quad\forall\, u\in H^1(\sO,\fw), \quad v \in H^1_0(\sO\cup\Sigma,\fw).
\end{equation}
\end{defn}

Examples of sufficient conditions for \eqref{eq:IntegrationByParts} to hold for bilinear maps, $\fa$, and operators, $A$, are described in Section \ref{subsec:IntegrationByParts}. The following lemma, whose proof is clear, relates the weak maximum principle property on $\sO\cup\Sigma$ in Definition \ref{defn:WeakMaxPrinciplePropertyBilinearForm} to that in Definition \ref{defn:WeakMaxPrinciplePropertyElliptic}.

\begin{lem}[Relationship between weak maximum principle properties]
\label{lem:H1H2WeakMaxPrinciplePropertyRelation}
Let $A: H^2(\sO,\fw) \to L^2(\sO,\fw)$ be a linear operator, let $\Sigma\subseteqq\partial\sO$, and let $\fa:H^1(\sO,\fw)\times H^1(\sO,\fw) \to \RR$ be a bilinear map which is associated to $A$ on $\sO\cup\Sigma$ through integration by parts. Then $A$ obeys the weak maximum principle property on $\sO\cup\Sigma$ for a convex cone $\fK\subset H^2(\sO,\fw)$ if and only if $\fa$ obeys the weak maximum principle property on $\sO\cup\Sigma$ for the convex cone $\fK$.
\end{lem}

We have the following analogue of Definition \ref{defn:WeakSolution}.

\begin{defn}[Strong solution, subsolution, and supersolution]
\label{defn:StrongSolution}
Suppose $f\in L^2(\sO,\fw)$ and $g \in H^1(\sO,\fw)$. We define $u \in H^2(\sO,\fw)$ to be a \emph{strong subsolution} if
\begin{align}
\label{eq:StrongSubsolution}
Au &\leq f \quad\hbox{a.e. on } \sO,
\\
\label{eq:StrongSubsolutionBC}
u &\leq g \quad\hbox{on } \partial\sO\less\Sigma \hbox{ in the sense of } H^1(\sO,\fw).
\end{align}
We call $u \in H^2(\sO,\fw)$ a \emph{strong supersolution} if $-u$ is a strong subsolution and a \emph{strong solution} if it is both a strong subsolution and supersolution.
\end{defn}

We then have

\begin{prop}[Weak maximum principle and \apriori estimates for functions in $H^2(\sO,\fw)$]
\label{prop:H2WeakMaxPrincipleAprioriEstimates}
Let $A: H^2(\sO,\fw) \to L^2(\sO,\fw)$ be a linear operator on $\sO\cup\Sigma$ associated to a bilinear map $\fa$ on $H^1(\sO,\fw)$ through integration by parts. Assume $A$ obeys the weak maximum principle property on $\sO\cup\Sigma$ for a convex cone $\fK\subset H^2(\sO,\fw)$ containing the constant function $1$. Then the conclusions of Proposition \ref{prop:H1WeakMaxPrincipleAprioriEstimates} hold for functions $u\in H^2(\sO,\fw)$, provided the properties \eqref{eq:BilinearMapNonnegLowerBound} or \eqref{eq:BilinearMapPositiveLowerBound} for $\fa$ are replaced by the properties for $A$ that
\begin{equation}
\label{eq:EllipticOperatorNonnegLowerBound}
A1 \geq 0 \quad\hbox{a.e. on }\sO,
\end{equation}
or that there is a constant $c_0>0$ such that
\begin{equation}
\label{eq:EllipticOperatorPositiveLowerBound}
A1 \geq c_0 \quad\hbox{a.e. on }\sO,
\end{equation}
and the role of Definition \ref{defn:WeakSolution} is replaced by that of Definition \ref{defn:StrongSolution}.
\end{prop}

\begin{proof}
When $u$ is a strong subsolution (supersolution, solution), then it is necessarily a variational subsolution (supersolution, solution) using \eqref{eq:IntegrationByParts} and so the result follows immediately from Proposition \ref{prop:H1WeakMaxPrincipleAprioriEstimates}.
\end{proof}

Finally, we consider an application of the weak maximum principle property to unbounded functions on possibly unbounded open subsets, by analogy with Section \ref{subsec:Weak_maximum_principle_C2_W2d_unbounded_function}. If a bilinear map has the weak maximum principle property for subsolutions which are essentially bounded above, we may obtain an extension for subsolutions which instead obey a growth condition.

When $u\in C^\infty_0(\bar\sO)$ and $v\in C^\infty_0(\sO\cup\Sigma)$, we have $\fa(u,v) = (Au,v)_{L^2(\sO,\fw)}$. If $\varphi\in C^2(\sO)$ and $[A,\varphi]u = -B(\varphi u)$, as in \eqref{eq:First_order_operator}, and
\begin{align*}
\fa(\varphi u,v) &= (A\varphi u,v)_{L^2(\sO,\fw)}
\\
&= (\varphi Au,v)_{L^2(\sO,\fw)} + ([A,\varphi]u,v)_{L^2(\sO,\fw)}
\\
&= (Au,\varphi v)_{L^2(\sO,\fw)} - (B\varphi u,v)_{L^2(\sO,\fw)}
\\
&= \fa(u,\varphi v)_{L^2(\sO,\fw)} - (B\varphi u,v)_{L^2(\sO,\fw)},
\end{align*}
and so
\begin{equation}
\label{eq:CommutatorFormula}
\fa(\varphi u,v) + (B\varphi u,v)_{L^2(\sO,\fw)} = \fa(u,\varphi v)_{L^2(\sO,\fw)}.
\end{equation}
Let $B:H^1(\sO,\fw)\to L^2(\sO,\fw)$ be the first-order differential operator defined by \eqref{eq:CommutatorFormula} for all $u \in H^1(\sO,\fw)$ and $v \in H^1_0(\sO\cup\Sigma,\fw)$. We then have the following version of Theorem \ref{thm:Weak_maximum_principle_C2_W2d_unbounded_function} for functions in $H^1(\sO,\fw)$.

\begin{thm}[Weak maximum principle for unbounded functions in $H^1(\sO,\fw)$ on unbounded open subsets]
\label{thm:Weak_maximum_principle_H1_unbounded_function}
Let $\sO\subseteqq\RR^d$ be a possibly unbounded open subset and let $\fa$ be a bilinear map on $H^1(\sO,\fw)$. Assume that the bilinear map,
\begin{equation}
\label{eq:DefnTildeBilinearForm}
\hat\fa(u,v) := \fa(u,v) + (Bu,v)_{L^2(\sO,\fw)}, \quad\forall\, u, v \in H^1(\sO,\fw),
\end{equation}
obeys the weak maximum principle property on $\sO\cup\Sigma$, for some $\Sigma\subseteqq\partial\sO$, for functions $u\in H^1(\sO,\fw)$ which are essentially bounded above. Then, $\fa$ has the weak maximum principle property on $\sO\cup\Sigma$ for functions $u\in H^1(\sO,\fw)$ obeying the growth condition \eqref{eq:Upper_bound_subsolution} a.e. on $\sO$.
\end{thm}

\begin{proof}
Since $\fa(u,v)\leq 0$, the identities \eqref{eq:CommutatorFormula} and \eqref{eq:DefnTildeBilinearForm} imply that $\fa(\varphi u,v)\leq 0$  for all $v\in H^1_0(\sO\cup\Sigma,\fw)$, $v\geq 0$ a.e. on $\sO$. The conclusion now follows, just as in the proof of Theorem \ref{thm:Weak_maximum_principle_C2_W2d_unbounded_function}.
\end{proof}

\begin{cor}[Weak maximum principle for unbounded $H^2(\sO,\fw)$ functions on unbounded open subsets]
\label{cor:Weak_maximum_principle_H2_unbounded_function}
Let $\sO\subseteqq\RR^d$ be a possibly unbounded open subset and let $A: H^2(\sO,\fw) \to L^2(\sO,\fw)$ be a linear operator on $\sO\cup\Sigma$, for some $\Sigma\subseteqq\partial\sO$, associated to a bilinear map $\fa$ on $H^1(\sO,\fw)$ through integration by parts. Assume that $\widehat A: H^2(\sO,\fw) \to L^2(\sO,\fw)$ in \eqref{eq:Defn_hatA_operator} obeys the weak maximum principle property on $\sO\cup\Sigma$ for functions $u\in H^2(\sO,\fw)$ which are essentially bounded above. Then, $A$ has the weak maximum principle property on $\sO\cup\Sigma$ for functions $u\in H^2(\sO,\fw)$ obeying the growth condition \eqref{eq:Upper_bound_subsolution} a.e. on $\sO$.
\end{cor}

\begin{proof}
When $u\in H^2(\sO,\fw)$ is a strong subsolution, then it is necessarily a variational subsolution since $\fa(u,v) = (Au,v)_{L^2(\sO,\fw)}$ for all $v \in H^1_0(\sO\cup\Sigma,\fw)$ by \eqref{eq:IntegrationByParts} and the result follows from Theorem \ref{thm:Weak_maximum_principle_H1_unbounded_function}.
\end{proof}

\section{Applications of the weak maximum principle property to variational inequalities}
\label{sec:ApplicationsWeakMaxPrinPropertyVarIneq}
We prove uniqueness for solutions to variational inequalities defined by bilinear maps $\fa$ on $H^1(\sO,\fw)$ and obstacle problems defined by bounded, linear operators $A$ from $H^2(\sO,\fw)$ to $L^2(\sO,\fw)$, when $\fa$ or $A$ obey the weak maximum principle property (Definition \ref{defn:WeakMaxPrinciplePropertyBilinearForm} or \ref{defn:WeakMaxPrinciplePropertyElliptic}, respectively). Applications to variational inequalities are much simpler than in the case of obstacle problems and include, in Section \ref{subsec:ApplicationsWeakMaxPrinPropertyVarIneqComparison},
a comparison principle for supersolutions and uniqueness for solutions to variational inequalities (Theorem \ref{thm:WeakMaximumPrincipleVariationalInequality}) and, in Section \ref{subsec:ApplicationsWeakMaxPrinPropertyH2Comparison}, the corresponding results for $H^2(\sO,\fw)$ solutions to the obstacle problem (Theorem \ref{thm:WeakMaximumPrincipleH2ObstacleProblem}). In Section \ref{subsec:ApplicationsWeakMaxPrinPropertyVarIneqAPrioriEstimates}, we develop \apriori estimates (Proposition \ref{prop:Rodrigues}) for $H^1(\sO,\fw)$ supersolutions and solutions to variational inequalities and then the corresponding results for (Proposition \ref{prop:H2Rodrigues}) for $H^2(\sO,\fw)$ supersolutions and solutions to the obstacle problem.

\subsection{Comparison principle for $H^1$ supersolutions and uniqueness for $H^1$ solutions to variational inequalities}
\label{subsec:ApplicationsWeakMaxPrinPropertyVarIneqComparison}
We first recall our analogue \cite{Daskalopoulos_Feehan_statvarineqheston} of the standard definition \cite{Bensoussan_Lions, Friedman_1982, Kinderlehrer_Stampacchia_1980, Rodrigues_1987, Troianiello} of a solution and supersolution to a variational inequality.

\begin{defn}[Solution and supersolution to a variational inequality]
\label{defn:HomogeneousVIProblem}
Given a source functional $F\in H^{-1}(\sO\cup\Sigma,\fw)$, boundary data function $g\in H^1(\sO,\fw)$, and an obstacle function $\psi\in H^1(\sO,\fw)$ such that $\psi\leq g$ on $\partial\sO\less\Sigma$ in the sense of $H^1(\sO,\fw)$, that is,
\begin{equation}
\label{eq:ObstacleFunctionLessThanBoundaryConditionFunctionH1}
(\psi-g)^+ \in H^1_0(\sO\cup\Sigma,\fw),
\end{equation}
we say that a function $u\in H^1(\sO,\fw)$ is a \emph{solution} to a variational inequality with partial Dirichlet boundary condition along $\partial\sO\less\Sigma$ if $u = g$ on $\partial\sO\less\Sigma$ in the sense of $H^1(\sO,\fw)$, that is,
$$
u-g\in H^1_0(\sO\cup\Sigma,\fw),
$$
and
\begin{equation}
\label{eq:VIProblemHestonHomgeneous}
\begin{gathered}
u\geq\psi\hbox{ a.e. on }\sO \quad\hbox{and}\quad \fa(u,v-u) \geq F(v-u),
\\
\forall\, v\in H^1(\sO,\fw) \hbox{ with } v\geq\psi \hbox{ a.e. on }\sO  \hbox{ and } v-g\in H^1_0(\sO\cup\Sigma,\fw).
\end{gathered}
\end{equation}
We call a function $u\in H^1(\sO,\fw)$ a \emph{supersolution} if $u \geq g$ on $\partial\sO\less\Sigma$ in the sense of $H^1(\sO,\fw)$, that is,
$$
(u-g)^-\in H^1_0(\sO\cup\Sigma,\fw),
$$
and
\begin{equation}
\label{eq:VIProblemHestonHomgeneousSupersolution}
\begin{gathered}
u\geq\psi\hbox{ a.e. on }\sO \quad\hbox{and}\quad \fa(u,w) \geq F(w),
\\
\forall\, w\in H^1_0(\sO\cup\Sigma,\fw) \hbox{ with } w\geq 0 \hbox{ a.e. on }\sO.
\end{gathered}
\end{equation}
\end{defn}

If $u$ is a solution in the sense of Definition \ref{defn:HomogeneousVIProblem} then we see that it is also a supersolution by writing $v = u+w$ and observing that $v \geq u \geq \psi$ a.e. on $\sO$ and $v = g$ on $\partial\sO\less\Sigma$ if $w \in H^1_0(\sO\cup\Sigma,\fw)$ and $w\geq 0$ a.e. on $\sO$.

We now assume that the bilinear map $\fa$ has the explicit form,
\begin{equation}
\label{eq:BilinearForm}
\fa(u,v) := \int_\sO \left(a^{ij}u_{x_i}v_{x_j} + d^juv_{x_j} - b^iu_{x_i}v + cuv\right)\fw\,dx, \quad\forall\, u,v \in C^\infty_0(\bar\sO),
\end{equation}
for some $\fw$ as in Definition \ref{defn:WeightFunction} (not necessarily coinciding with the $\fw_i$  in the Definition \ref{defn:SobolevSpaces} of $H^1(\sO,\fw)$) and where the coefficients $a^{ij}$, $d^j$, $b^i$, and $c$ are measurable functions on $\sO\subset\RR^d$. We shall always require that the coefficient matrix, $(a^{ij})$, define a measurable function $a:\sO\to\sS^+(d)$, as in \eqref{eq:a_nonnegative_symmetric_matrix_valued}. We then have the

\begin{thm}[Comparison principle for supersolutions and uniqueness for solutions to the variational inequality]
\label{thm:WeakMaximumPrincipleVariationalInequality}
Let $F\in H^{-1}(\sO\cup\Sigma,\fw)$, and $g\in H^1(\sO,\fw)$, and $\psi\in H^1(\sO,\fw)$, and $\fa$ be a bilinear map of the form \eqref{eq:BilinearForm} on $H^1(\sO,\fw)$ obeying the weak maximum principle property on $\sO\cup\Sigma$, for some $\Sigma\subseteqq\partial\sO$ and convex cone $\fK\subset H^1(\sO,\fw)$, and, in addition, that
\begin{equation}
\label{eq:ZerodjCoefficients}
d^j = 0 \quad\hbox{a.e. on }\sO, \quad 1\leq j\leq d.
\end{equation}
Suppose $u_1\in \fK$ is a solution and $u_2\in -\fK$ is a supersolution to the associated variational inequality (Definition \ref{defn:HomogeneousVIProblem}). Then $u_2\geq u_1$ a.e. on $\sO$ and if $u_2$ is also a solution, then $u_2 = u_1$ a.e. on $\sO$.
\end{thm}

\begin{proof}\footnote{An alternative proof of Theorem \ref{thm:WeakMaximumPrincipleVariationalInequality} in the simpler case where the solution, supersolution, and obstacle function are continuous on $\sO$, obtained by adapting the proofs of \cite[Theorems 1.3.3 and 1.3.4]{Friedman_1982}, can be found in \cite{Feehan_maximumprinciple_v0}.}
We shall adapt the proof of \cite[Theorem 4.27]{Troianiello}. We may assume without loss of generality that $g=0$. Suppose that $u_1$ is a solution and $u_2$ is a supersolution and set $\hat u: = (u_1-u_2)^+$, so $u_1\wedge u_2 = u_1 - (u_1-u_2)^+ = u_1 - \hat u$. Observe that $\hat u \in H^1(\sO,\fw)$ and $u_1-u_2 \leq 0$ on $\partial\sO\less\Sigma$ and thus $\hat u = (u_1-u_2)^+ = 0$ on $\partial\sO\less\Sigma$, in the sense of $H^1(\sO,\fw)$, and hence $\hat u \in H^1_0(\sO,\fw)$.

We claim that $\fa(\hat u,w) \leq 0$ for all $w \in H^1_0(\sO\cup\Sigma,\fw)$ with $w \geq 0$ a.e. on $\sO$. If not, there exists $w \in C^1_0(\sO\cup\Sigma)$, $0\leq w\leq 1$, such that
\begin{equation}
\label{eq:Positivebilinearformvalue}
\fa(\hat u,w) > 0.
\end{equation}
For $\eps>0$, define
$$
w^\eps := \frac{\hat u w}{\hat u + \eps}.
$$
Observe that $w^\eps \in H^1_0(\sO\cup\Sigma,\fw)$ and
\begin{equation}
\label{eq:wepsGradient}
w^\eps_{x_j} = \frac{\hat u w_{x_j}}{\hat u + \eps} + \frac{\eps\hat u_{x_j}w}{(\hat u + \eps)^2}, \quad 1\leq j\leq d.
\end{equation}
Choose $v^\eps = u_1 - \eps w^\eps$ and observe that $v^\eps \in H^1_0(\sO\cup\Sigma,\fw)$ and
$$
v^\eps \geq u_1 - \hat u w \geq u_1 - \hat u = u_1\wedge u_2 \geq \psi \quad\hbox{a.e. on }\sO.
$$
Because $u_1$ is a solution to the variational inequality \eqref{eq:VIProblemHestonHomgeneous}, we have
$$
\fa(u_1,v-u_1) \geq F(v-u_1), \quad\forall\, v \in H^1_0(\sO\cup\Sigma,\fw) \hbox{ with } v\geq\psi\hbox{ a.e. on }\sO,
$$
and thus, choosing $v=v^\eps$ and using $v^\eps - u_1 = - \eps w^\eps$, we obtain
$$
\fa(-u_1,w^\eps) \geq -F(w^\eps).
$$
But $u_2$ is a supersolution to the variational inequality \eqref{eq:VIProblemHestonHomgeneous}, thus
$$
\fa(u_2,w^\eps) \geq F(w^\eps).
$$
Adding the preceding two inequalities yields
$$
\fa(u_2-u_1,w^\eps) \geq 0.
$$
Since $\hat u = (u_1-u_2)^+$ and $w^\eps = \hat uw/(\hat u + \eps)$, we obtain
$$
\fa(\hat u, w^\eps) \leq 0.
$$
Using \eqref{eq:wepsGradient}, the expression \eqref{eq:BilinearForm} for $\fa$ yields
\begin{align*}
\fa(\hat u, w^\eps) &= \int_\sO \left(a^{ij}\hat u_{x_i}w^\eps_{x_j} + d^j\hat uw^\eps_{x_j} - b^i\hat u_{x_i}w^\eps + c\hat uw^\eps\right)\fw\,dx
\\
&= \int_\sO \frac{\hat u}{\hat u + \eps}\left(a^{ij}\hat u_{x_i}w_{x_j} + d^j\hat uw_{x_j} - b^i\hat u_{x_i}w + c\hat uw\right)\fw\,dx
\\
&\quad + \eps\int_\sO \frac{w}{(\hat u + \eps)^2} \left(a^{ij}\hat u_{x_i} + d^j\hat u\right)\hat u_{x_j}\fw\,dx
\\
&=: I_1(\eps) + \eps I_2(\eps).
\end{align*}
Consequently, by combining the preceding inequality and identity, we obtain
\begin{equation}
\label{eq:I1I2epsleqZero}
I_1(\eps) + \eps I_2(\eps)\leq 0, \quad\forall\eps>0.
\end{equation}
By hypothesis \eqref{eq:ZerodjCoefficients}, we have $(d^j)=0$ on $\sO$ and so the non-negative definite characteristic form condition \eqref{eq:a_nonnegative_symmetric_matrix_valued} for $(a^{ij})$ implies that
$$
I_2(\eps) = \int_\sO \frac{w}{(\hat u + \eps)^2}a^{ij}\hat u_{x_i}\hat u_{x_j}\fw\,dx \geq 0.
$$
Therefore, $I_1(\eps) \leq -\eps I_2(\eps) \leq 0$ for all $\eps>0$ and because
\begin{equation}
\label{eq:I1limit}
\lim_{\eps\to 0}I_1(\eps) = \fa(\hat u,w),
\end{equation}
we obtain $\fa(\hat u,w) \leq 0$, contradicting our assumption that $\fa(\hat u,w) > 0$.

Thus, $\fa(\hat u,w)\leq 0$ for all $w\in H^1_0(\sO\cup\Sigma,\fw)$ with $w\geq 0$ a.e. on $\sO$ and so the weak maximum principle property implies that $\hat u \leq 0$ a.e. on $\sO$, that is $(u_1-u_2)^+ = 0$ a.e. on $\sO$ and so $u_1 \leq u_2$ a.e. on $\sO$. This completes the proof that if $u_1$ is a solution and $u_2$ a supersolution to the variational inequality \eqref{eq:VIProblemHestonHomgeneous}, then $u_2\geq u_1$ a.e. on $\sO$.

If $u_1, u_2$ are two solutions to \eqref{eq:VIProblemHestonHomgeneous}, then $u_1=u_2$ a.e. on $\sO$ just as before. This completes the proof.
\end{proof}

If we strengthen the non-negativity condition \eqref{eq:a_nonnegative_symmetric_matrix_valued} for the matrix $(a^{ij})$, we can allow non-zero coefficients $(d^j)$ in \eqref{eq:BilinearForm} for the statement and proof of Theorem \ref{thm:WeakMaximumPrincipleVariationalInequality}. We constrain the behavior of the coefficients $a,d$ in the expression \eqref{eq:BilinearForm} for the bilinear map $\fa$, near finite portions of $\partial\sO$ as well as spatial infinity, with the aid of the

\begin{defn}[Degeneracy coefficient]
\label{defn:DegeneracyCoefficient}
We call $\vartheta$ a \emph{degeneracy coefficient} if
\begin{equation}
\label{eq:DegeneracyCoefficientCondition}
\vartheta \in C_{\loc}(\bar\sO) \quad\hbox{and}\quad \vartheta > 0 \quad\hbox{on }\sO.
\end{equation}
\end{defn}

For our generalization of Theorem \ref{thm:WeakMaximumPrincipleVariationalInequality}, we now require that the coefficients $(a^{ij}), (d^j)$ in \eqref{eq:BilinearForm} obey
\begin{align}
\label{eq:GeneralNonDegeneracyNearBoundaryQuant}
\langle a\eta,\eta\rangle &\geq \vartheta|\eta|^2  \quad\hbox{a.e. on }\sO,
\\
\label{eq:BilinearaBound}
\langle a\eta,\eta\rangle &\leq K\vartheta|\eta|^2 \quad\hbox{a.e. on }\sO,
\\
\label{eq:BilineardBound}
|\langle d,\eta\rangle| &\leq K\vartheta|\eta| \quad\hbox{a.e. on }\sO,
\end{align}
for some positive constant, $K$, and all $\eta\in \RR^d$. When the matrix $(a^{ij})$ defines coefficients of a bilinear form \eqref{eq:BilinearForm}, our previous characterization \eqref{eq:Degeneracy_locus_elliptic} of a degeneracy locus is not convenient, so instead we shall use the

\begin{defn}[Characterization of the degeneracy locus for a bilinear map]
For a bilinear form as in \eqref{eq:BilinearForm}, with measurable coefficients $(a^{ij})$, the \emph{degeneracy locus}, $\Sigma\subseteqq\partial\sO$, for $\fa$ is given by
\begin{equation}
\label{eq:aijZeroisVarthetaZero}
\Sigma = \Int\{x\in\partial\sO: \vartheta(x) = 0\}.
\end{equation}
\end{defn}

The definitions \eqref{eq:Degeneracy_locus_elliptic} and \eqref{eq:aijZeroisVarthetaZero} coincide when $(a^{ij}) \in C_{\loc}(\partial\sO;\sS^+(d))$.

\begin{thm}[Comparison principle for supersolutions and uniqueness for solutions to the variational inequality]
\label{thm:WeakMaximumPrincipleVariationalInequalityNonzerod}
Assume the hypotheses of Theorem \ref{thm:WeakMaximumPrincipleVariationalInequality} but now require that either the coefficients $(d^j)$ in the expression \eqref{eq:BilinearForm} for the bilinear map $\fa$ obey \eqref{eq:ZerodjCoefficients} or that the coefficients $(a^{ij}),(d^j)$ in \eqref{eq:BilinearForm} obey \eqref{eq:GeneralNonDegeneracyNearBoundaryQuant}, \eqref{eq:BilinearaBound}, and \eqref{eq:BilineardBound} and that
\begin{equation}
\label{eq:VarthetaL1weight}
\vartheta \in L^1(\sO,\fw).
\end{equation}
Then the conclusions of Theorem \ref{thm:WeakMaximumPrincipleVariationalInequality} continue to hold.
\end{thm}

\begin{proof}
We need only show that the assumption \eqref{eq:Positivebilinearformvalue} still leads to a contradiction using the inequality \eqref{eq:I1I2epsleqZero} even when $(d^j)$ is non-zero. By \eqref{eq:Positivebilinearformvalue}, \eqref{eq:I1I2epsleqZero}, and \eqref{eq:I1limit}, we obtain
\begin{equation}
\label{eq:Positive_epsI2limit}
\limsup_{\eps\to 0}\eps I_2(\eps) \leq -\lim_{\eps\to 0}I_1(\eps) = -\fa(\hat u,w) < 0,
\end{equation}
and so there is an $\eps_0>0$ such that $\eps I_2(\eps) \leq 0$ and hence $I_2(\eps) \leq 0$, $\forall\, \eps\in(0,\eps_0]$. Define
$$
G_\eps(\hat u) := \frac{w^{1/2}|D\hat u|}{\hat u + \eps} \quad\hbox{a.e. on }\sO.
$$
Then, the fact that $I_2(\eps)\leq 0$, $\forall\eps\in (0,\eps_0]$, yields
\begin{align*}
\int_\sO \vartheta G_\eps^2(\hat u)\fw\,dx &= \int_\sO \vartheta \frac{w|D\hat u|^2}{(\hat u + \eps)^2}\fw\,dx
\leq
\int_\sO a^{ij}\frac{w\hat u_{x_i}\hat u_{x_j}}{(\hat u + \eps)^2}\fw\,dx  \quad\hbox{(by \eqref{eq:GeneralNonDegeneracyNearBoundaryQuant})}
\\
&\leq
-\int_\sO d^j\frac{w\hat u\hat u_{x_j}}{(\hat u + \eps)^2}\fw\,dx \quad\hbox{(since $I_2(\eps)\leq 0$)}
\\
&\leq K\int_\sO \vartheta\frac{w^{1/2}|D\hat u|}{\hat u + \eps}\fw\,dx
=
K\int_\sO \vartheta G_\eps(\hat u)\fw\,dx  \quad\hbox{(by \eqref{eq:BilineardBound})}
\\
&\leq K\left(\int_\sO \vartheta\fw\,dx\right)^{1/2}\left(\int_\sO \vartheta G_\eps^2(\hat u)\fw\,dx\right)^{1/2}
\\
&= C_0\left(\int_\sO \vartheta G_\eps^2(\hat u)\fw\,dx\right)^{1/2},
\end{align*}
where $C_0 := K\left(\int_\sO \vartheta\fw\,dx\right)^{1/2}$. But then
$$
\|\sqrt{\vartheta} G_\eps(\hat u)\|_{L^2(\sO,\fw)} \leq C_0, \quad \forall\, \eps\in (0,\eps_0],
$$
and so, $\forall\, \eps\in (0,\eps_0]$,
\begin{align*}
|I_2(\eps)| &\leq \int_\sO \frac{w}{(\hat u + \eps)^2} \left(|a^{ij}\hat u_{x_i}\hat u_{x_j}| + |d^j\hat u_{x_j}|\hat u\right)\fw\,dx
\\
&\leq K\int_\sO \vartheta\frac{w}{(\hat u + \eps)^2}\left(|D\hat u|^2 + |D\hat u|\hat u\right)|\fw\,dx \quad\hbox{(by \eqref{eq:BilinearaBound} and \eqref{eq:BilineardBound})}
\\
&\leq K\int_\sO \vartheta\frac{w|D\hat u|^2}{(\hat u + \eps)^2}\fw\,dx
+
K\left(\int_\sO\vartheta\frac{w|D\hat u|^2}{(\hat u + \eps)^2}\fw\,dx\right)^{1/2}\left(\int_\sO\vartheta\frac{w{\hat u}^2}{(\hat u + \eps)^2}\fw\,dx\right)^{1/2}
\\
&= K\|\sqrt{\vartheta} G_\eps(\hat u)\|_{L^2(\sO,\fw)}^2
+ K\|\sqrt{\vartheta} G_\eps(\hat u)\|_{L^2(\sO,\fw)}\left(\int_\sO \vartheta\fw\,dx\right)^{1/2}
\\
&\leq (K+1)C_0^2.
\end{align*}
Therefore, $|\eps I_2(\eps)|\leq (K+1)C_0^2\eps$ and so $\lim_{\eps\to 0}\eps I_2(\eps) = 0$, contradicting \eqref{eq:Positive_epsI2limit}.
\end{proof}

\subsection{Comparison principle for $H^2$ supersolutions and uniqueness for $H^2$ solutions to obstacle problems}
\label{subsec:ApplicationsWeakMaxPrinPropertyH2Comparison}
We recall our analogues of the standard definition \cite{Bensoussan_Lions, Friedman_1982, Kinderlehrer_Stampacchia_1980, Rodrigues_1987, Troianiello} of a strong solution to an obstacle problem defined in \cite{Daskalopoulos_Feehan_statvarineqheston}.

\begin{defn}[Strong solution and supersolution to an obstacle problem]
\label{defn:ObstacleHomogeneousStrong}
Given functions $f\in L^2(\sO,\fw)$, $g\in H^1(\sO,\fw)$, and $\psi\in H^1(\sO,\fw)$ obeying \eqref{eq:ObstacleFunctionLessThanBoundaryConditionFunctionH1}, we call $u\in H^2(\sO,\fw)$ a \emph{strong solution} to an obstacle problem for a linear operator, $A: H^2(\sO,\fw) \to L^2(\sO,\fw)$, with Dirichlet boundary condition along $\partial\sO\less\Sigma$, for some $\Sigma\subseteqq\partial\sO$, if
\begin{align}
\label{eq:ObstacleProblemH2}
\min\{Au-f,u-\psi\} &= 0 \quad \hbox{a.e. on }\sO,
\\
\label{eq:ObstacleHomogeneousBC}
u-g &\in H^1_0(\sO\cup\Sigma,\fw),
\end{align}
that is, $u=g$ on $\partial\sO\less\Sigma$ in the sense of $H^1(\sO,\fw)$.
We call $u\in H^2(\sO,\fw)$ a \emph{strong supersolution} if
\begin{align}
\label{eq:ObstacleProblemH2Supersolution}
\min\{Au-f,u-\psi\} &\geq 0 \quad \hbox{a.e. on }\sO,
\\
\label{eq:ObstacleHomogeneousBCSupersolution}
(u-g)^- &\in H^1_0(\sO\cup\Sigma,\fw),
\end{align}
that is, $u\geq g$ on $\partial\sO\less\Sigma$ in the sense of $H^1(\sO,\fw)$.
\end{defn}

\begin{thm}[Comparison principle and uniqueness for $H^2$ solutions to the obstacle problem]
\label{thm:WeakMaximumPrincipleH2ObstacleProblem}
Let $f\in L^2(\sO,\fw)$, and $g\in H^1(\sO,\fw)$, and $\psi\in H^1(\sO,\fw)$, and $A: H^2(\sO,\fw) \to L^2(\sO,\fw)$ be a linear operator on $\sO\cup\Sigma$ associated to a bilinear map $\fa$ of the form \eqref{eq:BilinearForm} on $H^1(\sO,\fw)$ through integration by parts, for some $\Sigma\subseteqq\partial\sO$. Assume $A$ obeys the weak maximum principle property on $\sO\cup\Sigma$ for some convex cone $\fK\subset H^2(\sO,\fw)$ (Definition \ref{defn:WeakMaxPrinciplePropertyElliptic}) and that the coefficients $(a^{ij}),(d^j)$ in the expression \eqref{eq:BilinearForm} obey the hypotheses of Theorem \ref{thm:WeakMaximumPrincipleVariationalInequalityNonzerod}. Suppose $u_1\in \fK$ is a strong solution and $u_2\in -\fK$ is a strong supersolution to the obstacle problem in the sense of Definition \ref{defn:ObstacleHomogeneousStrong}. Then $u_2 \geq u_1$ a.e. on $\sO$ and if $u_1, u_2$ are solutions, then $u_2 = u_1$ a.e on $\sO$.
\end{thm}

We shall need the following analogue of the equivalence in \cite[Equation (3.1.20)]{Bensoussan_Lions}, whose justification is identical to the proof of the corresponding \cite[Lemma 4.13]{Daskalopoulos_Feehan_statvarineqheston}.

\begin{lem}[Equivalence of variational and strong (super-)solutions]
\label{lem:VIStrongForm}
Assume the hypotheses for the operator, $A$, in Theorem \ref{thm:WeakMaximumPrincipleH2ObstacleProblem}. Let $f$, and $g$, and $\psi$ be as in Definition \ref{defn:ObstacleHomogeneousStrong} and suppose $u\in H^2(\sO,\fw)$. Then $u$ is a (super-)solution to the variational inequality in Definition \ref{defn:HomogeneousVIProblem} if and only if $u$ is a (super-)solution to the obstacle problem in Definition \ref{defn:ObstacleHomogeneousStrong}.
\end{lem}

\begin{proof}[Proof of Theorem \ref{thm:WeakMaximumPrincipleH2ObstacleProblem}]
\footnote{An alternative proof Theorem \ref{thm:WeakMaximumPrincipleH2ObstacleProblem} in the simpler case where the solution, supersolution, and obstacle function are continuous on $\sO$, obtained by adapting the proof of \cite[Theorem 1.3.4]{Friedman_1982}, can be found in \cite{Feehan_maximumprinciple_v0}.}
Since $u_1, u_2$ are a strong solution and supersolution in the sense of Definition \ref{defn:ObstacleHomogeneousStrong}, then they are necessarily a solution and supersolution to the variational inequality in the sense of Definition \ref{defn:HomogeneousVIProblem} by Lemma \ref{lem:VIStrongForm}, and so $u_1\leq u_2$ a.e. on $\sO$ by Theorem \ref{thm:WeakMaximumPrincipleVariationalInequalityNonzerod}, while $u_1 = u_2$ a.e. on $\sO$ if both $u_1$ and $u_2$ are strong solutions.
\end{proof}

\subsection{\emph{A priori} estimates for solutions and supersolutions to variational inequalities}
\label{subsec:ApplicationsWeakMaxPrinPropertyVarIneqAPrioriEstimates}
We now state weak maximum principles and corresponding \apriori estimates which extend those of non-coercive variational inequalities defined by uniformly elliptic partial differential operators, such as \cite[Theorems 4.5.4 and 4.7.4]{Rodrigues_1987} (see also \cite[Theorem 4.5.1 and Corollary 4.5.2]{Rodrigues_1987}). If $u\in H^1(\sO,\fw)$ is a supersolution to the variational inequality in Definition \ref{defn:HomogeneousVIProblem}, then $u$ necessarily obeys the inequality,
$$
u \leq \psi \quad\hbox{a.e. on }\sO,
$$
but the weak maximum principle yields additional \apriori estimates for $u$, as described below.

\begin{prop}[\emph{A priori} estimates for supersolutions and solutions to variational inequalities]
\label{prop:Rodrigues}
Let $F=(f,f^1,\ldots,f^d) \in H^{-1}(\sO\cup\Sigma,\fw)$, and $g\in H^1(\sO,\fw)$, and $\psi\in H^1(\sO,\fw)$, and $\fa$ be a bilinear map of the form \eqref{eq:BilinearForm} on $H^1(\sO,\fw)$ obeying the hypotheses of Theorem \ref{thm:WeakMaximumPrincipleVariationalInequalityNonzerod}, for some $\Sigma\subseteqq\partial\sO$ and convex cone $\fK \subset H^1(\sO,\fw)$ containing the constant function $1$, and assume that \eqref{eq:BilinearMapNonnegLowerBound} holds. Suppose that $u \in H^1(\sO,\fw)$, is a supersolution to the associated variational inequality (Definition \ref{defn:HomogeneousVIProblem}).
\begin{enumerate}
\item \label{item:VISupersolutionfgeqZero} If $F\geq 0$ and $u\in-\fK$, then
$$
u\geq 0 \wedge \inf_{\partial\sO\less\Sigma}g \quad\hbox{a.e. on }\sO.
$$
\item\label{item:VISupersolutionfarbsign} If $F=(f,0,\ldots,0)$ and $u\in-\fK$ and $f$ has arbitrary sign and $\fa$ obeys \eqref{eq:BilinearMapPositiveLowerBound}, then
$$
u\geq 0 \wedge \frac{1}{c_0}\essinf_\sO f \wedge \inf_{\partial\sO\less\Sigma}g \quad\hbox{a.e. on }\sO.
$$
\item \label{item:VISolutionfleqZero} If $F\leq 0$ and $u\in \fK$ is a solution for $F$ and $g$ and $\psi$ (Definition \ref{defn:HomogeneousVIProblem}), then
$$
u\leq 0 \vee \sup_{\partial\sO\less\Sigma}g \vee \esssup_\sO\psi \quad\hbox{a.e. on }\sO.
$$
\item\label{item:VISolutionfarbsign} If $F=(f,0,\ldots,0)$ and $f$ has arbitrary sign, $u\in \fK\cap -\fK$ is a solution for $F$ and $g$ and $\psi$, and $\fa$ obeys \eqref{eq:BilinearMapPositiveLowerBound}, then
$$
u\leq 0 \vee \frac{1}{c_0}\esssup_\sO f \vee \sup_{\partial\sO\less\Sigma}g \vee \esssup_\sO\psi \quad\hbox{a.e. on }\sO.
$$
\item \label{item:VIComparison} If $u_1$ and $u_2$ are solutions in $\fK\cap -\fK$, respectively, for $F_1\geq F_2$ and $\psi_1\geq \psi_2$ a.e. on $\sO$, and $g_1\geq g_2$ on $\partial\sO\less\Sigma$ in the sense of $H^1(\sO,\fw)$, then
$$
u_1\geq u_2 \quad\hbox{a.e. on }\sO.
$$
\item\label{item:VIStability} If $u_i\in \fK\cap -\fK$ is a solution for $F_i$, $\psi_i$, and $g_i$ with $\psi_i\leq g_i$ on $\partial\sO\less\Sigma$ in the sense of $H^1(\sO,\fw)$, and  $F_i = (f_i,0,\ldots,0)$ for $i=1,2$, and $\fa$ obeys \eqref{eq:BilinearMapPositiveLowerBound}, then
$$
\|u_1-u_2\|_{L^\infty(\sO)} \leq \frac{1}{c_0}\|f_1-f_2\|_{L^\infty(\sO)} \vee \|g_1-g_2\|_{L^\infty(\partial\sO\less\Sigma)} \vee \|\psi_1-\psi_2\|_{L^\infty(\sO)},
$$
while if $F_i=F$ for $i=1,2$ and $\fa$ obeys \eqref{eq:BilinearMapNonnegLowerBound}, then
$$
\|u_1-u_2\|_{L^\infty(\sO)} \leq \|g_1-g_2\|_{L^\infty(\partial\sO\less\Sigma)} \vee \|\psi_1-\psi_2\|_{L^\infty(\sO)}.
$$
\end{enumerate}
The terms $\sup_{\partial\sO\less\Sigma}g$, and $\inf_{\partial\sO\less\Sigma}g$, and $\|g\|_{L^\infty(\partial\sO\less\Sigma)}$, and $\|g_1-g_2\|_{L^\infty(\partial\sO\less\Sigma)}$ in the preceding items are omitted when $\Sigma=\partial\sO$.
\end{prop}

\begin{proof}
Consider Items \eqref{item:VISupersolutionfgeqZero} and \eqref{item:VISupersolutionfarbsign}. Since $u$ is a supersolution to the variational inequality in Definition \ref{defn:HomogeneousVIProblem}, then it is also a supersolution to the variational equation in Definition \ref{defn:WeakSolution} (where $\psi$ plays no role) and so Items \eqref{item:VISupersolutionfgeqZero} and \eqref{item:VISupersolutionfarbsign} here just restate Items \eqref{item:SupersolutionfgeqZero} and \eqref{item:Supersolutionfarbsign} in Proposition \ref{prop:H1WeakMaxPrincipleAprioriEstimates}.

Consider Items \eqref{item:VISolutionfleqZero} and \eqref{item:VISolutionfarbsign}. When $f\leq 0$ a.e. on $\sO$, let
$$
M := 0\vee \sup_{\partial\sO\less\Sigma}g \vee\esssup_\sO\psi,
$$
while if $f$ has arbitrary sign and $\fa$ obeys \eqref{eq:BilinearMapPositiveLowerBound}, let
$$
M := 0\vee \frac{1}{c_0}\esssup_\sO f \vee \sup_{\partial\sO\less\Sigma}g \vee\esssup_\sO\psi.
$$
We may assume without loss of generality that $M<\infty$. Then $M\geq \psi$ a.e. on $\sO$ and $M\geq g$ on $\partial\sO\less\Sigma$ in the sense of $H^1(\sO,\fw)$, while for all $w\in H^1_0(\sO\cup\Sigma,\fw)$ with $w\geq 0$ a.e. on $\sO$, we have
$$
\fa(M,w) = (cM,w)_{L^2(\sO,\fw)} \geq 0 \geq (f,w)_{L^2(\sO,\fw)},
$$
when $f\leq 0$ a.e. on $\sO$ and
$$
\fa(M,w) \geq (c_0M,w)_{L^2(\sO,\fw)} \geq (f,w)_{L^2(\sO,\fw)},
$$
when $f$ has arbitrary sign. Hence, $M$ is a supersolution and so Theorem \ref{thm:WeakMaximumPrincipleVariationalInequalityNonzerod} implies that $u\leq M$ a.e. on $\sO$, which establishes Items \eqref{item:VISolutionfleqZero} and \eqref{item:VISolutionfarbsign}.

The proofs of Items \eqref{item:VIComparison} and \eqref{item:VIStability} here are very similar to the proofs of the corresponding Items \eqref{item:Obstacle_comparison} and \eqref{item:Obstacle_stability} in Proposition \ref{prop:Elliptic_weak_maximum_principle_apriori_estimates_obstacle_problem}, except that appeals to Proposition \ref{prop:Comparison_principle_uniqueness_obstacle_problem} are replaced by appeals to Theorem \ref{thm:WeakMaximumPrincipleVariationalInequalityNonzerod}.
\end{proof}

Note that an $L^\infty$ comparison estimate for solutions $u_1, u_2$ corresponding to $f$, $g$, and obstacle functions $\psi_1, \psi_2$ is provided by \cite[Theorem 4.7.4]{Rodrigues_1987} and \cite[Theorem 3.1.10]{Bensoussan_Lions}.

\begin{prop}[Weak maximum principle and \apriori estimates for strong solutions to obstacle problems]
\label{prop:H2Rodrigues}
Let $f\in L^2(\sO,\fw)$, and $g\in H^1(\sO,\fw)$, and $A: H^2(\sO,\fw) \to L^2(\sO,\fw)$ be a linear operator on $\sO\cup\Sigma$ associated to a bilinear map $\fa$ on $H^1(\sO,\fw)$ through integration by parts. Assume $A$ obeys the weak maximum principle property on $\sO\cup\Sigma$ for some convex cone $\fK\subset H^1(\sO,\fw)$. Then the conclusions of Proposition \ref{prop:Rodrigues} hold, provided the properties \eqref{eq:BilinearMapNonnegLowerBound} or \eqref{eq:BilinearMapPositiveLowerBound} for $\fa$ are replaced by the properties \eqref{eq:EllipticOperatorNonnegLowerBound} or \eqref{eq:EllipticOperatorPositiveLowerBound} for $A$, the role of Definition \ref{defn:HomogeneousVIProblem} is replaced by that of Definition \ref{defn:ObstacleHomogeneousStrong}, and supersolution is replaced by solution.
\end{prop}

\begin{proof}
When $u\in H^2(\sO,\fw)$ is a strong (super-)solution, then it is necessarily a variational (super-)solution using \eqref{eq:IntegrationByParts} and so the result follows immediately from Proposition \ref{prop:Rodrigues}.
\end{proof}

\section{Weak maximum principle for functions in Sobolev spaces}
\label{sec:WeakMaxPrincipleSobolev}
Having considered applications of the weak maximum principle property (Definition \ref{defn:WeakMaxPrinciplePropertyBilinearForm}) to variational equations in Section \ref{sec:ApplicationsWeakMaxPrinPropertyVarEq} and variational inequalities in Section \ref{sec:ApplicationsWeakMaxPrinPropertyVarIneq}, we now establish conditions under which the bilinear map $\fa$ on $H^1(\sO,\fw)$ or the associated differential operator $A:H^2(\sO,\fw)\to L^2(\sO,\fw)$ has the weak maximum principle property on $\sO\cup\Sigma$. We begin in Section \ref{subsec:H1Prelim} with some technical preliminaries. In Section \ref{subsec:H1Bounded}, we prove a weak maximum principle for $\fa$ when $H^1(\sO,\fw)$ is defined by power weights and $\sO$ is a bounded open subset of the upper half-space, $\HH\subset\RR^d$ (Theorem \ref{thm:WeakMaximumPrincipleVarEquationBoundedDomain}); modulo a suitable weighted Sobolev inequality (Hypothesis \ref{hyp:WeightedSobolevInequality}), we then prove a more widely applicable weak maximum principle for $\fa$ when $H^1(\sO,\fw)$ is defined by general weights and $\sO$ is a bounded open subset of $\RR^d$ (Theorem \ref{thm:WeakMaximumPrincipleVarEquationBoundedDomainGeneralWeight}). In Section \ref{subsec:IntegrationByParts}, we describe an integration by parts formula relating, under suitable conditions on the weights and coefficients, the bilinear map, $\fa$, on $H^1(\sO,\fw)$ and the associated operator, $A$, on $H^2(\sO,\fw)$. Finally, in Section \ref{subsec:H1Unbounded}, we prove a weak maximum principle for bounded $H^1(\sO,\fw)$ functions on unbounded open subsets (Theorem \ref{thm:WeakMaximumPrincipleH1UnboundedDomainBoundedSolution}).

Our weak maximum principle differs in several aspects from \cite[Theorems 1.5.1 and 1.5.5]{Radkevich_2009a}, which again may appear subtle at first glance but which are still important for applications:
\begin{enumerate}
\item We use weighted Sobolev spaces adapted to the coefficients of the first and second-order derivatives in $A$ and the resulting conditions on the coefficients are weaker than those of \cite[Theorems 1.5.1 and 1.5.5]{Radkevich_2009a} and permit applications to operators such as those of Examples \ref{exmp:HestonPDE} and \ref{exmp:GeneralizedPorousMediumEquation} which are not covered by \cite[Theorems 1.5.1 and 1.5.5]{Radkevich_2009a};
\item The open subset $\sO\subset\RR^d$ is allowed to be \emph{unbounded}.
\end{enumerate}
The differences between results obtainable from our weak maximum and comparison principles and those of Fichera, Ole{\u\i}nik, and Radkevi{\v{c}} are described in more detail in Appendix \ref{sec:FicheraAndHeston} using the example of the Heston operator (Example \ref{exmp:HestonPDE}).

\subsection{Preliminaries}
\label{subsec:H1Prelim}
We let $\fw$ be as in Definition \ref{defn:WeightFunction} and $\vartheta$ be as in Definition \ref{defn:DegeneracyCoefficient} and obey \eqref{eq:VarthetaL1weight} and choose
\begin{equation}
\label{eq:H1Norm}
\|u\|_{H^1(\sO,\fw)}^2 := \int_\sO\left(\vartheta|Du|^2 + (1+\vartheta)|u|^2\right)\fw\,dx,
\end{equation}
so that $\fw_{1,0} = (1+\vartheta)\fw$ and $\fw_{1,1} = \vartheta\fw$ in the Definition \ref{defn:SobolevSpaces} of $H^1(\sO,\fw)$, while $\fw_{0,0} = \fw$ in the definition of $L^2(\sO,\fw)$.

We shall assume that the coefficients, $(a^{ij})$ and $(d^j)$,  of $\fa$ in \eqref{eq:BilinearForm} obey \eqref{eq:BilinearaBound} and \eqref{eq:BilineardBound} and, in addition, that the coefficients $(b^i)$ and $c$ obey
\begin{align}
\label{eq:BilinearbBound}
|\langle b,\eta\rangle| &\leq K\vartheta|\eta| \quad\hbox{a.e. on }\sO, \quad\forall\eta\in\RR^d,
\\
\label{eq:BilinearcBound}
|c| &\leq K(1+\vartheta) \quad\hbox{a.e on }\sO,
\end{align}
for some positive constant, $K$. It is easy to check that $\fa$ obeys the \emph{continuity estimate},
\begin{equation}
\label{eq:BilinearFormEstimate}
\fa(u,v) \leq C_1\|u\|_{H^1(\sO,\fw)}\|v\|_{H^1(\sO,\fw)}, \quad\forall\, u, v \in H^1(\sO,\fw),
\end{equation}
for some positive constant, $C_1$, when the coefficients, $(a^{ij})$, $(d^j)$, and $(b^i)$, of $\fa$ in \eqref{eq:BilinearForm} obey \eqref{eq:BilinearaBound}, \eqref{eq:BilineardBound}, \eqref{eq:BilinearbBound}, and \eqref{eq:BilinearcBound}, in which case $C_1=C_1(K)$ in \eqref{eq:BilinearFormEstimate}.

For $1\leq p<\infty$, we let $L^p(\sO,\fw)$ denote the Banach space of measurable functions, $u$, on $\sO$ such that
\begin{equation}
\label{eq:LpNorm}
\|u\|_{L^p(\sO,\fw)}^p := \int_\sO |u|^p\fw\,dx < \infty.
\end{equation}
The bilinear map $\fa$ obeys a \emph{G\r{a}rding inequality},
\begin{equation}
\label{eq:GardingEstimate}
\fa(u,u) \geq C_2\|u\|_{H^1(\sO,\fw)}^2 - C_3\|(1+\vartheta)^{1/2}u\|_{L^2(\sO,\fw)}, \quad\forall\, u\in H^1(\sO,\fw),
\end{equation}
if, in addition, we require that the coefficient matrix, $a=(a^{ij})$, obeys \eqref{eq:GeneralNonDegeneracyNearBoundaryQuant}, in which case $C_2=C_2(K)$ and $C_3=C_3(K)$.

We now specialize the Hilbert space $H^1(\sO,\fw)$ in Definition \ref{defn:SobolevSpaces} to be the completion of the vector space $C^\infty_0(\bar\sO)$ with respect to the norm \eqref{eq:H1Norm}. Given $\Sigma\subseteqq\partial\sO$, we let $H^1_0(\sO\cup\Sigma,\fw)$ be as in Definition \ref{defn:SobolevSpaces}. (We prove an analogue of the Meyers-Serrin theorem \cite{Adams_1975} for unweighted Sobolev spaces in the context of certain weighted Sobolev spaces in \cite{Daskalopoulos_Feehan_statvarineqheston}.)

\subsection{$H^1$ functions on bounded and unbounded open subsets}
\label{subsec:H1Bounded}
We first consider a special case of our desired maximum principle, analogous to \cite[Theorem 8.1]{GilbargTrudinger}.

\begin{thm}[Weak maximum principle for $H^1(\sO,\fw)$ functions]
\label{thm:WeakMaximumPrincipleVarEquationCompactDomain}
Let $\sO\subsetneqq\RR^d$ be a \emph{bounded} open subset and let $\Sigma\subseteqq\partial\sO$. Assume that the coefficients of $\fa$ in \eqref{eq:BilinearForm} are measurable on $\sO$, obey \eqref{eq:GeneralNonDegeneracyNearBoundaryQuant}, \eqref{eq:BilinearaBound}, \eqref{eq:BilineardBound}, \eqref{eq:BilinearMapNonnegLowerBound}, \eqref{eq:BilinearbBound}, \eqref{eq:BilinearcBound}, and require that $\fw$ and $\vartheta$ obey
\begin{equation}
\label{eq:SupBoundsDegenCoeffWeight}
\inf_\sO\vartheta\fw > 0 \hbox{ and }\sup_\sO\vartheta\fw < \infty.
\end{equation}
If $u \in H^1(\sO,\fw)$ obeys \eqref{eq:VariationalSubsolution} with $F=0$, then
$$
\esssup_\sO u \leq 0\vee\sup_{\partial\sO\less\Sigma} u.
$$
Moreover, $\fa$ has the \emph{weak maximum principle property on $\sO\cup\Sigma$} in the sense of Definition \ref{defn:WeakMaxPrinciplePropertyBilinearForm}.
\end{thm}

\begin{rmk}[Boundedness requirement on $\sO$]
In applications, the functions $\vartheta$ and $\fw$ would not normally obey the upper bound in \eqref{eq:SupBoundsDegenCoeffWeight} unless $\sO$ were bounded and the lower bound in \eqref{eq:SupBoundsDegenCoeffWeight} unless $a=(a^{ij})$ were uniformly elliptic on $\sO$. Moreover, boundedness of $\sO$ is implicitly used in the proof of Theorem \ref{thm:WeakMaximumPrincipleVarEquationCompactDomain} in inequalities involving the Lebesgue measures of the supports of functions and their gradients. However, unlike the proof of \cite[Theorem 8.1]{GilbargTrudinger}, our proof of Theorem \ref{thm:WeakMaximumPrincipleVarEquationCompactDomain} avoids the use of the Poincar\'e inequality, thanks to a nice observation of Camelia Pop, and hence a requirement that $\sO$ is bounded originating from the usual statements of the Poincar\'e inequality (for example, \cite[Theorem 5.6.3]{Evans}).
\end{rmk}

\begin{proof}
We essentially follow the proof of \cite[Theorem 8.1]{GilbargTrudinger}, but include the details here for later reference in our proof of Theorem \ref{thm:WeakMaximumPrincipleVarEquationBoundedDomain}. If $u \in H^1(\sO,\fw)$ and $v \in H^1_0(\sO\cup\Sigma,\fw)$, then $uv \in W^{1,1}_0(\sO\cup\Sigma,\fw)$ and $D(uv) = vDu + uDv$ by analogy with \cite[Problem 7.4]{GilbargTrudinger}, recalling that $H^1(\sO,\fw) = W^{1,2}_0(\sO,\fw)$ and $H^1_0(\sO\cup\Sigma,\fw) = W^{1,2}_0(\sO\cup\Sigma,\fw)$. By the definition \eqref{eq:BilinearForm} of the bilinear map, $\fa(u,v)$, we obtain
$$
\int_\sO\left(a^{ij}u_{x_i}v_{x_j} - (d^i + b^i)u_{x_i}v\right)\,\fw\,dx \leq -\int_\sO \left(d^i(uv)_{x_i} + cuv\right)\,\fw\,dx \leq 0,
$$
for all $v \in H^1_0(\sO\cup\Sigma,\fw)$ such that $v \geq 0$ and $uv\geq 0$ a.e. on $\sO$, where to obtain the last inequality we use \eqref{eq:BilinearMapNonnegLowerBound}. Therefore,
$$
\int_\sO a^{ij}u_{x_i}v_{x_j}\fw\,dx \leq \int_\sO (d^i + b^i)u_{x_i}v\,\fw\,dx
\leq
K\int_\sO \vartheta|Du|v\,\fw\,dx \quad\hbox{(by \eqref{eq:BilineardBound} and \eqref{eq:BilinearbBound}).}
$$
Denote
$$
l := 0\vee\sup_{\partial\sO\less\Sigma}u\geq 0,
$$
and recall our convention \eqref{eq:EssSupOverEmptyBoundary} that $0\vee\sup_{\partial\sO\less\Sigma}u = 0$ when $\partial\sO\less\Sigma = \emptyset$. We may assume without loss of generality that $l<\infty$, as otherwise there is nothing to prove. Suppose there exists a constant $k$ such that
\begin{equation}
\label{eq:Rangeofk}
l\leq k < \esssup_\sO u \leq +\infty.
\end{equation}
(If no such $k$ exists, then we are done.) Set
\begin{equation}
\label{eq:Defnv}
v := (u-k)^+,
\end{equation}
and observe that, because $u\leq k$ on $\partial\sO\less\Sigma$ in the sense of $H^1(\sO,\fw)$ (when $\partial\sO\less\Sigma$ is non-empty) and $u, k\in H^1(\sO,\fw)$, then $v \in H^1_0(\sO\cup\Sigma,\fw)$ with
$$
Dv = \begin{cases} Du &\hbox{for } u\geq k, \\ 0 &\hbox{for } u < k. \end{cases}
$$
Consequently, if $\sU$ denotes the interior of $\supp Dv \subset \supp v$,
$$
\int_\sU \vartheta|Dv|^2\fw\,dx \leq K\int_\sU \vartheta|Dv|v\,\fw\,dx,
$$
since $\langle Du, Dv\rangle = |Dv|^2$ a.e. on $\sO$ and \eqref{eq:GeneralNonDegeneracyNearBoundaryQuant} gives
$$
\int_\sO a^{ij}v_{x_i}v_{x_j}\fw\,dx \geq \int_\sU \vartheta|Dv|^2\fw\,dx.
$$
Therefore, by the Cauchy-Schwartz inequality,
\begin{equation}
\label{eq:L2DvleqL2vFullWeight}
\|\vartheta^{1/2}Dv\|_{L^2(\sU,\fw)} \leq K\|\vartheta^{1/2}v\|_{L^2(\sU,\fw)}.
\end{equation}
The inequality \eqref{eq:L2DvleqL2vFullWeight} yields
\begin{equation}
\label{eq:L2DvleqL2v}
\|Dv\|_{L^2(\sU)} \leq C_1\|v\|_{L^2(\sU)},
\end{equation}
for a positive constant
$$
C_1 := K\left(\frac{\sup_\sO\vartheta\fw}{\inf_\sO\vartheta\fw}\right)^{1/2},
$$
which is finite by \eqref{eq:SupBoundsDegenCoeffWeight}. Combining the inequality \eqref{eq:L2DvleqL2v} with the Sobolev embedding $W^{1,2}_0(\sU) \to L^q(\sU)$, $2\leq q<\infty$ if $d=2$ and $2\leq q \leq 2d/(d-2)$ if $d>2$ \cite[Theorem 5.4 (Parts I (A and B) and III)]{Adams_1975} and the fact that $v\in W^{1,2}_0(\sU)$ implies, for positive constants $C_2, C_3$ depending on $C_1, q, \sU$, that
\begin{align*}
\|v\|_{L^q(\sU)} &\leq C_2\left(\|Dv\|_{L^2(\sU)} + \|v\|_{L^2(\sU)}\right) \quad\hbox{(Sobolev embedding with $q>2$])}
\\
&\leq C_3\|v\|_{L^2(\sU)} \quad\hbox{(by \eqref{eq:L2DvleqL2v})}
\\
&\leq C_3|\sU|^{1/2-1/q}\|v\|_{L^q(\sU)}  \quad\hbox{(by \cite[Equation (7.8)]{GilbargTrudinger})},
\end{align*}
and thus, recalling that $\sU$ is the interior of $\supp Dv$,
$$
|\supp Dv|^{1/2-1/q} \geq C_3 > 0.
$$
The inequality is independent of $k < \esssup_\sO u$ and so continues to hold when we take the limit $k\to \esssup_\sO u$. Hence, we see that $u$ attains its maximum $\esssup_\sO u \leq +\infty$ on a set, $\sU$, of positive measure. If $\esssup_\sO u = +\infty$ on $\sU$, we obtain a contradiction to the fact that $u\in L^2(\sU)$, since  $u\in L^2(\sO,\fw)$; if $\esssup_\sO u < +\infty$ then, because $u$ is constant on $\sU$, one must also have $Du=0$ on $\sU$, contradicting the fact that
$$
|\sU\cap\supp Du| = |\supp Dv| > 0.
$$
This completes the proof.
\end{proof}

We now suppose $\sO\subset\HH$ and recall the following

\begin{thm}[Weighted Sobolev inequality for power weights]
\label{thm:KochWeightedSobolevInequality}
\cite[Theorem 4.2.2]{Koch}
Suppose $1\leq p\leq q<\infty$, and $s > -1/p$, and $1-d/p < \xi\leq 1$ is defined by
\begin{equation}
\label{eq:Defnalpha}
\frac{1}{p} = \frac{1}{q} + \frac{1-\xi}{d}.
\end{equation}
Then there is a positive constant $C=C(d,p,q,s)$ such that, for any $u\in L^q(\HH,x_d^s)$ with $Du \in L^p(\HH,x_d^{s+\xi};\RR^d)$, one has
\begin{equation}
\label{eq:KochWeightedSobolevInequality}
\|x_d^su\|_{L^q(\HH)} \leq C\|x_d^{s+\xi}Du\|_{L^p(\HH)}.
\end{equation}
\end{thm}

\begin{rmk}
The proof of Theorem \ref{thm:KochWeightedSobolevInequality} is based on the Hardy inequality \cite[Lemma 1.3]{KufnerOpic}.
\end{rmk}

We have the following generalization of \cite[Lemma 4.2.4]{Koch}; see \cite[Appendix B]{Feehan_maximumprinciple_v0} for a comparison between Corollary \ref{cor:PowerWeightedSobolevInequality} and a weighted Sobolev inequality due to Maz'ya \cite[Theorem 2.6.1]{Turesson_2000}.

\begin{cor}[Application of weighted Sobolev inequality for power weights]
\label{cor:PowerWeightedSobolevInequality}
Suppose $\beta>0$ and $2-d <\alpha\leq 2$. For $u\in L^2(\HH,x_d^{\beta-1})$ with $Du \in L^2(\HH,x_d^{\beta-1+\alpha};\RR^d)$, and $q\geq 2$ defined by
\begin{equation}
\label{eq:Defnq}
\frac{1}{2} = \frac{1}{q} + \frac{1-\alpha/2}{d},
\end{equation}
and $2\leq r\leq q$, one has
\begin{equation}
\label{eq:PowerWeightedSobolevInequality}
\|u\|_{L^r(\HH,x_d^{\beta-1})} \leq C\|u\|_{L^2(\HH,x_d^{\beta-1})}^\lambda\|Du\|_{L^2(\HH,x_d^{\beta-1+\alpha})}^{1-\lambda},
\end{equation}
where $\lambda\in[0,1]$ is defined by
\begin{equation}
\label{eq:DefnLambda}
\frac{1}{r} = \frac{\lambda}{2} + \frac{1-\lambda}{q}.
\end{equation}
\end{cor}

\begin{proof}
When $s=(\beta-1)/2$, so $\beta>0$ when $s>-1/2$, and $\xi=\alpha/2$, so $1-d/2 < \xi\leq 1$ when $2-d < \alpha\leq 2$, it follows from Theorem \ref{thm:KochWeightedSobolevInequality} that
$$
\|u\|_{L^q(\HH,x_d^{\beta-1})} \leq C\|Du\|_{L^2(\HH,x_d^{\beta-1+\alpha})}.
$$
Holder's inequality, in the form of \cite[Equation (7.9)]{GilbargTrudinger}, gives
$$
\|u\|_{L^r(\HH,x_d^{\beta-1})} \leq C\|u\|_{L^2(\HH,x_d^{\beta-1})}^\lambda\|u\|_{L^q(\HH,x_d^{\beta-1})}^{1-\lambda},
$$
when $\lambda\in[0,1]$ is defined by \eqref{eq:DefnLambda}. Combining the preceding two inequalities yields the result.
\end{proof}

\begin{rmk}[Weighted Poincar\'e and Sobolev inequalities]
Direct weighted analogues of the standard Sobolev inequalities, with weights given by positive powers of the distance to a boundary portion, $\Sigma\subseteqq\partial\sO$, such as \cite[Theorems 19.9 and 19.10]{KufnerOpic}, only hold for a very restrictive range of powers, even when $\sO$ is bounded. Koch provides a weighted Poincar\'e inequality on $\HH$ with a weight similar to ours \cite[Lemma 4.4.4]{Koch}, as well as certain weighted Sobolev inequalities \cite[Theorem 4.2.2 and Lemma 4.2.4]{Koch}. Adams \cite[Section 6.26]{Adams_1975} provides an unweighted Poincar\'e inequality which is valid on unbounded open subsets of `finite width', while the Gagliardo-Nirenberg-Sobolev inequality \cite[Theorem 5.6.1]{Evans}, a Poincar\'e-type inequality on $\RR^d$ and the Caffarelli-Kohn-Nirenberg inequality, another weighted Poincar\'e inequality on $\RR^d$ \cite{Dolbeault_Esteban_Loss_Tarantello_2009}, are potentially useful in this context.
\end{rmk}

With the aid of Corollary \ref{cor:PowerWeightedSobolevInequality}, as noticed by Camelia Pop, we can relax the non-degeneracy requirement \eqref{eq:SupBoundsDegenCoeffWeight} in the hypotheses of Theorem \ref{thm:WeakMaximumPrincipleVarEquationCompactDomain}.

\begin{thm}[Weak maximum principle for $H^1(\sO,\fw)$ functions on bounded open subsets of the upper half-space]
\label{thm:WeakMaximumPrincipleVarEquationBoundedDomain}
Assume the hypotheses of Theorem \ref{thm:WeakMaximumPrincipleVarEquationCompactDomain}, except that we now omit the requirement \eqref{eq:SupBoundsDegenCoeffWeight}. In addition, assume $\sO\subsetneqq\HH$ is a \emph{bounded} open subset and require that there are constants, $0<c_\vartheta<1$ and $0<c_\fw<1$, such that
\begin{align}
\label{eq:DegenCoeffEquivxdAlpha}
c_\vartheta x_d^\alpha \leq \vartheta \leq c_\vartheta^{-1} x_d^\alpha \quad\hbox{on }\sO,
\\
\label{eq:WeightEquivxdAlpha}
c_\fw x_d^{\beta-1} \leq \fw \leq c_\fw^{-1} x_d^{\beta-1} \quad\hbox{on }\sO,
\end{align}
where $\beta>0$ and $0\leq \alpha < 2$. If $u \in H^1(\sO,\fw)$ obeys \eqref{eq:VariationalSubsolution} with $f=0$, then
$$
\esssup_\sO u \leq 0\vee\sup_{\partial\sO\less\Sigma} u.
$$
Moreover, $\fa$ has the \emph{weak maximum principle property on $\sO\cup\Sigma$} in the sense of Definition \ref{defn:WeakMaxPrinciplePropertyBilinearForm}.
\end{thm}

\begin{proof}
We proceed as in the proof of Theorem \ref{thm:WeakMaximumPrincipleVarEquationCompactDomain}, except that we now appeal to Corollary \ref{cor:PowerWeightedSobolevInequality} in place of the standard Sobolev inequality \cite[Theorem 5.4]{Adams_1975}. From \eqref{eq:L2DvleqL2vFullWeight} and the fact that $v$ in \eqref{eq:Defnv} extends by zero outside $\supp v$ to an element of $H^1(\HH,x_d^{\beta-1})$ by the analogue of \cite[Lemma 3.22]{Adams_1975}, also denoted by $v$, we obtain
\begin{equation}
\label{eq:L2DvleqL2vGenWeight}
\|Dv\|_{L^2(\sU,x_d^{\beta-1+\alpha})} \leq C_1\|v\|_{L^2(\sU,x_d^{\beta-1+\alpha})},
\end{equation}
where $C_1:=(c_\vartheta c_\fw)^{-1}K$ and, as in the proof of Theorem \ref{thm:WeakMaximumPrincipleVarEquationCompactDomain}, the set $\sU\subset\sO$ denotes the interior of $\supp Dv$. Then
\begin{align*}
\|v\|_{L^r(\HH,x_d^{\beta-1})} &\leq C_2\|v\|_{L^2(\HH,x_d^{\beta-1})}^\lambda\|Dv\|_{L^2(\HH,x_d^{\beta-1+\alpha})}^{1-\lambda}
\quad\hbox{(by \eqref{eq:PowerWeightedSobolevInequality})}
\\
&\leq C_3\|v\|_{L^2(\HH,x_d^{\beta-1})}^\lambda\|v\|_{L^2(\sU,x_d^{\beta-1+\alpha})}^{1-\lambda} \quad\hbox{(by \eqref{eq:L2DvleqL2vGenWeight})}
\\
&\leq C_3\max_{x\in\sO}x_d^{(1-\lambda)\alpha/2}\|v\|_{L^2(\HH,x_d^{\beta-1})}^\lambda\|v\|_{L^2(\sU,x_d^{\beta-1})}^{1-\lambda},
\\
&\equiv C_4\|v\|_{L^2(\HH,x_d^{\beta-1})}^\lambda\|v\|_{L^2(\sU,x_d^{\beta-1})}^{1-\lambda},
\end{align*}
for a positive constant $C_4\equiv C_3\max_{x\in\sO}x_d^{(1-\lambda)\alpha/2}$ independent of the constant $k$ in \eqref{eq:Rangeofk}, where we apply Corollary \ref{cor:PowerWeightedSobolevInequality} with $r>2$ (which is possible since $\alpha<2$ and thus $q>2$) and $0<\lambda<1$. Recall from \cite[Equation (7.8)]{GilbargTrudinger} that
\begin{align*}
\|v\|_{L^2(\HH,x_d^{\beta-1})} &\leq |\supp v|_\beta^{1/2-1/r}\|v\|_{L^r(\HH,x_d^{\beta-1})},
\\
\|v\|_{L^2(\sU,x_d^{\beta-1})} &\leq |\sU|_\beta^{1/2-1/r}\|v\|_{L^r(\HH,x_d^{\beta-1})},
\end{align*}
where we denote $|\sS|_\beta := \int_\sS x_d^{\beta-1}\,dx$ for any $\beta>0$ and measurable subset $\sS\subset\HH$. Hence, the preceding inequalities give
\begin{align*}
\|v\|_{L^r(\HH,x_d^{\beta-1})} &\leq C_4\|v\|_{L^2(\HH,x_d^{\beta-1})}^\lambda\|v\|_{L^2(\sU,x_d^{\beta-1})}^{1-\lambda}
\\
&\leq C_4|\supp v|_\beta^{\lambda(1/2-1/r)}\|v\|_{L^r(\HH,x_d^{\beta-1})}^{\lambda}
|\sU|_\beta^{(1-\lambda)(1/2-1/r)}\|v\|_{L^r(\sU,x_d^{\beta-1})}^{1-\lambda}
\\
&= C_4|\supp v|_\beta^{\lambda(1/2-1/r)}|\sU|_\beta^{(1-\lambda)(1/2-1/r)}\|v\|_{L^r(\sU,x_d^{\beta-1})},
\end{align*}
and so, noting that $\sU$ is the interior of $\supp Dv$,
$$
C_4|\supp v|_\beta^{\lambda(1/2-1/r)}|\supp Dv|_\beta^{(1-\lambda)(1/2-1/r)} \geq 1.
$$
Thus, since $|\supp v|_\beta \leq |\sO|_\beta < \infty$ and $|\sO|_\beta > 0$,
$$
|\supp Dv|_\beta^{(1-\lambda)(1/2-1/r)} \geq C_4^{-1}|\sO|_\beta^{-\lambda(1/2-1/r)} > 0,
$$
recalling that $0<\lambda<1$ and $r>2$. We again obtain a contradiction, after taking the limit $k\to\esssup_\sO u$, and the result follows.
\end{proof}

Given a suitable weighted Sobolev inequality for functions on open subsets $\sO\subseteqq\RR^d$, it is a straightforward to generalize Theorem \ref{thm:WeakMaximumPrincipleVarEquationBoundedDomain} from the case where $\vartheta$ and $\fw$ obey \eqref{eq:DegenCoeffEquivxdAlpha} and \eqref{eq:WeightEquivxdAlpha}.

\begin{hyp}[Weighted Sobolev inequality]
\label{hyp:WeightedSobolevInequality}
Given an open subset $\sO\subseteqq\RR^d$, constants $1\leq p\leq q < \infty$, and functions $\vartheta, \fw \in C(\sO)$ such that $\vartheta>0$ and $\fw>0$ on $\sO$,
there is a positive constant $C=C(p,q,\vartheta,\fw)$ such that, for any $u\in L^q(\sO,\fw)$ with $Du \in L^p(\sO,\vartheta\fw;\RR^d)$, one has
\begin{equation}
\label{eq:WeightedSobolevInequality}
\|u\|_{L^q(\sO,\fw)} \leq C\|Du\|_{L^p(\sO,\vartheta\fw)}.
\end{equation}
\end{hyp}

\begin{cor}[Application of weighted Sobolev inequality]
\label{cor:WeightedSobolevInequality}
Assume Hypothesis \ref{hyp:WeightedSobolevInequality} holds with $p=2$. For $u\in L^2(\sO,\fw)$ with $Du \in L^2(\sO,\vartheta\fw;\RR^d)$ and $2\leq r\leq q$, one has
\begin{equation}
\label{eq:ProductWeightedSobolevInequality}
\|u\|_{L^r(\sO,\fw)} \leq C\|u\|_{L^2(\sO,\fw)}^\lambda\|Du\|_{L^2(\sO,\vartheta\fw)}^{1-\lambda},
\end{equation}
where $\lambda\in[0,1]$ is defined by \eqref{eq:DefnLambda}.
\end{cor}

\begin{proof}
The proof is similar to that of Corollary \ref{cor:PowerWeightedSobolevInequality}. Inequality \eqref{eq:WeightedSobolevInequality}, with $2\leq q<\infty$, gives
$$
\|u\|_{L^q(\sO,\fw)} \leq C\|Du\|_{L^2(\sO,\vartheta\fw)}.
$$
Holder's inequality, in the form of \cite[Equation (7.9)]{GilbargTrudinger}, yields
$$
\|u\|_{L^r(\sO,\fw)} \leq C\|u\|_{L^2(\sO,\fw)}^\lambda\|u\|_{L^q(\sO,\fw)}^{1-\lambda},
$$
when $\lambda\in[0,1]$ is defined by \eqref{eq:DefnLambda}. Combining the preceding two inequalities yields the result.
\end{proof}

We have the following generalization of Theorem \ref{thm:WeakMaximumPrincipleVarEquationCompactDomain}.

\begin{thm}[Weak maximum principle for $H^1(\sO,\fw)$ functions on bounded open subsets and general weights]
\label{thm:WeakMaximumPrincipleVarEquationBoundedDomainGeneralWeight}
Assume the hypotheses of Theorem \ref{thm:WeakMaximumPrincipleVarEquationCompactDomain}, except that we allow
$\sO\Subset\RR^d$ to be any \emph{bounded} open subset with $\Sigma\subseteqq\partial\sO$ and allow $\vartheta, \fw$ to be any functions obeying Hypothesis \ref{hyp:WeightedSobolevInequality} for $p=2$ and some $2<q<\infty$. If $u \in H^1(\sO,\fw)$ obeys \eqref{eq:VariationalSubsolution} with $f=0$, then
$$
\esssup_\sO u \leq 0\vee\sup_{\partial\sO\less\Sigma} u.
$$
Moreover, $\fa$ has the \emph{weak maximum principle property on $\sO\cup\Sigma$} in the sense of Definition \ref{defn:WeakMaxPrinciplePropertyBilinearForm}.
\end{thm}

\begin{proof}
The proof is almost identical to that of Theorems \ref{thm:WeakMaximumPrincipleVarEquationCompactDomain} and \ref{thm:WeakMaximumPrincipleVarEquationBoundedDomain}, except for a few minor changes which we indicate here. In place of \eqref{eq:L2DvleqL2vGenWeight}, we note that \eqref{eq:L2DvleqL2vFullWeight} may be written as
\begin{equation}
\label{eq:L2DvleqL2vVeryGenWeight}
\|Dv\|_{L^2(\sU,\vartheta\fw)} \leq K\|v\|_{L^2(\sU,\vartheta\fw)}.
\end{equation}
We now proceed as in the proof of Theorem \ref{thm:WeakMaximumPrincipleVarEquationBoundedDomain}, but apply \eqref{eq:ProductWeightedSobolevInequality} in place of \eqref{eq:PowerWeightedSobolevInequality} and choose $C_4 = C_3\max_{x\in\sO}\vartheta^{(1-\lambda)/2}$ instead of $C_3\max_{x\in\sO}x_d^{(1-\lambda)\alpha/2}$.
\end{proof}

\subsection{Integration by parts formula}
\label{subsec:IntegrationByParts}
Before proceeding to consider when the weak maximum principle holds on unbounded open subsets, we shall need to introduce an integration by parts formula and so, to accomplish this, we must impose additional conditions on the coefficients of $\fa$ beyond those stated in Section \ref{subsec:H1Prelim}. For now, we shall require that $a=(a^{ij})$ be continuous on $\bar\sO$ (not merely measurable on $\sO$) but shortly strengthen this condition further.
Let $A$ be the partial differential operator given in equivalent divergence and non-divergence forms by
\begin{equation}
\label{eq:GeneralDivergenceFormOperator}
\begin{aligned}
Au &:= -\left(a^{ij}u_{x_i} + d^ju\right)_{x_j} - \left(b^i + (\log\fw)_{x_j}a^{ij}\right)u_{x_i} + cu
\\
&= -a^{ij}u_{x_ix_j} - \left(b^i + a^{ij}_{x_j} + d^i + (\log\fw)_{x_j}a^{ij}\right)u_{x_i} + \left(c - d^j_{x_j} - (\log\fw)_{x_j}d^j\right)u
\\
&= -a^{ij}u_{x_ix_j} - \tilde b^iu_{x_i} + \tilde cu,
\end{aligned}
\end{equation}
with
\begin{align}
\label{eq:tildebcoeff}
\tilde b^i &:= b^i + a^{ij}_{x_j} + (\log\fw)_{x_j}a^{ij}, \quad 1\leq i \leq d,
\\
\tilde c &:= c - d^j_{x_j} - (\log\fw)_{x_j}d^j,
\end{align}
and where we now impose the additional regularity requirements,
\begin{align}
\label{eq:Lipschitzaij}
a^{ij} &\in C^{0,1}(\sO)\cap C_{\loc}(\bar\sO), \quad 1\leq i,j \leq d,
\\
\label{eq:Lipschitzdj}
d^j &\in C^{0,1}(\sO)\cap C_{\loc}(\bar\sO), \quad 1\leq j \leq d,
\\
\label{eq:LipschitzLogWeight}
\log\fw &\in C^{0,1}(\sO).
\end{align}
Provided we also require that $\vartheta$, $\fw$ obey
\begin{align}
\label{eq:ContinuousWeightDegenCoeffUpToBdry}
\vartheta\fw &\in C_{\loc}(\bar\sO),
\\
\label{eq:WeightDegenCoeffZeroGammaOne}
\vartheta\fw &= 0 \hbox{ on } \Sigma,
\end{align}
and that $\sO$ is an open subset for which the divergence theorem holds,
then integration by parts in \eqref{eq:BilinearForm} yields the integration by parts relation \eqref{eq:IntegrationByParts} when $u\in C^\infty_0(\bar\sO)$ and $v\in C^\infty_0(\sO\cup\Sigma)$,
\begin{equation}
\label{eq:IntegrationByPartsFormula}
\fa(u,v) = (Au,v)_{L^2(\sO,\fw)},
\end{equation}
since
\begin{align*}
\fa(u,v) &= \int_\sO \left(\left(a^{ij}u_{x_i} + d^ju\right)v_{x_j} - b^iu_{x_i}v + cuv\right)\fw\,dx \quad\hbox{(by \eqref{eq:BilinearForm})}
\\
& = \int_\sO \left(-\left(a^{ij}u_{x_i} + d^ju\right)_{x_j} - b^iu_{x_i} + cu\right)v\fw\,dx
\\
&\quad - \int_\sO \left(a^{ij}u_{x_i} + d^ju\right)v\fw_{x_j}\,dx - \int_{\partial\sO} n_j\left(a^{ij}u_{x_i} + d^ju\right)v\fw\,ds
\\
& = \int_\sO \left(-a^{ij}u_{x_ix_j} - \left(b^i+a^{ij}_{x_j}+d^i+(\log\fw)_{x_j}a^{ij}\right)u_{x_i} \right.
\\
&\qquad \left. + \left(c-d^j_{x_j}-(\log\fw)_{x_j}d^j\right)u\right)v\fw\,dx
\\
&= (Au,v)_{L^2(\sO,\fw)}, \quad\hbox{(by \eqref{eq:GeneralDivergenceFormOperator})}
\end{align*}
where $\vec n$ is the \emph{inward}-pointing normal vector field and the integral over $\partial\sO$ is zero since $v=0$ on $\partial\sO\less\Sigma$ and $\fw, \vartheta$ obey \eqref{eq:ContinuousWeightDegenCoeffUpToBdry} and \eqref{eq:WeightDegenCoeffZeroGammaOne}, and the coefficients $(a^{ij}), (d^j)$ obey \eqref{eq:GeneralNonDegeneracyNearBoundaryQuant} and \eqref{eq:BilinearaBound} on $\bar\sO$ together with \eqref{eq:Lipschitzaij} and \eqref{eq:Lipschitzdj}.

\begin{rmk}[Relaxing the conditions on $\partial\sO$]
More generally, if the divergence theorem is not assumed to hold for $\sO$, then \eqref{eq:IntegrationByPartsFormula} still holds under slightly stronger regularity assumptions on the coefficients, $a^{ij}$, and weight, $\fw$, near $\Sigma$.
\end{rmk}

Clearly, when the coefficients $a,\tilde b,\tilde c$ of $A$ obey \eqref{eq:BilinearaBound} and
\begin{align}
\label{eq:BilinearTildebBound}
|\tilde b| \leq K(1+\vartheta) \quad\hbox{a.e. on }\sO,
\\
\label{eq:BilinearTildecBound}
|\tilde c| \leq K(1+\vartheta) \quad\hbox{a.e. on }\sO,
\end{align}
there is a positive constant, $C_4=C_4(K)$, such that
\begin{equation}
\label{eq:AuH2Estimate}
\|Au\|_{L^2(\sO,\fw)} \leq C_4\|u\|_{H^2(\sO,\fw)}, \quad\forall\, u\in C^\infty_0(\bar\sO),
\end{equation}
where we set
\begin{equation}
\label{eq:H2Norm}
\|u\|_{H^2(\sO,\fw)}^2 := \int_\sO\left(\vartheta^2|D^2u|^2 + (1+\vartheta^2)\left(|Du|^2 + |u|^2\right)\right)\fw\,dx,
\end{equation}
so that $\fw_{2,0} = \fw_{2,1} = (1+\vartheta)^2\fw$ and $\fw_{2,2} = \vartheta^2\fw$ in the Definition \ref{defn:SobolevSpaces} of $H^2(\sO,\fw)$.

We specialize the Hilbert space $H^2(\sO,\fw)$ in Definition \ref{defn:SobolevSpaces} to be the completion of the vector space $C^\infty_0(\bar\sO)$ with respect to the norm \eqref{eq:H2Norm} and note that \eqref{eq:AuH2Estimate} continues to hold when $u\in H^2(\sO,\fw)$. Furthermore, the proof of \cite[Lemma 2.33]{Daskalopoulos_Feehan_statvarineqheston} (integration by parts) adapts to show that \eqref{eq:IntegrationByPartsFormula} continues to hold when $u\in H^2(\sO,\fw)$ and $v\in H^1_0(\sO\cup\Sigma,\fw)$.

\begin{rmk}[Conditions on the coefficients of $A$ and the weights]
The bounds \eqref{eq:BilinearTildebBound} and \eqref{eq:BilinearTildecBound} hold if we strengthen the conditions \eqref{eq:Lipschitzaij}, \eqref{eq:Lipschitzdj}, and \eqref{eq:LipschitzLogWeight} by requiring that $|a^{ij}_{x_j}| \leq K(1+\vartheta)$ a.e. on $\sO$, and $|d^j_{x_j}| \leq K(1+\vartheta)$ a.e. on $\sO$, and $|(\log\fw)_{x_j}| \leq K(1+\vartheta)$ a.e. on $\sO$.
\end{rmk}

\begin{exmp}[Heston operator]
\label{exmp:HestonIntegrationByParts}
We show in Appendix \ref{sec:FicheraAndHeston} that the coefficients of the Heston operator, $A$, its associated bilinear map, $\fa$, and weight function, $\fw$, defined in \cite{Daskalopoulos_Feehan_statvarineqheston}, and the degeneracy coefficient, $\vartheta$, obey the conditions described in Section \ref{subsec:IntegrationByParts}.
\end{exmp}

\subsection{$H^1$ functions on unbounded open subsets}
\label{subsec:H1Unbounded}
By adapting the proof of the maximum principle for bounded $C^2$ functions, Theorem \ref{thm:Weak_maximum_principle_C2_unbounded_opensubset}, and appealing to Theorem \ref{thm:WeakMaximumPrincipleVarEquationBoundedDomain}, instead of Theorem \ref{thm:Weak_maximum_principle_C2}, we obtain

\begin{thm}[Weak maximum principle for bounded $H^1(\sO,\fw)$ functions on unbounded open subsets]
\label{thm:WeakMaximumPrincipleH1UnboundedDomainBoundedSolution}
Let $\sO\subseteqq\RR^d$ be a possibly unbounded open subset such that the divergence theorem holds. Assume the hypotheses of Theorem \ref{thm:WeakMaximumPrincipleVarEquationBoundedDomain} (power weights and open subsets of $\HH$) or Theorem \ref{thm:WeakMaximumPrincipleVarEquationBoundedDomainGeneralWeight} (general weights and subdomains of $\RR^d$). Assume, in addition, that $\tilde c$ obeys \eqref{eq:c_positive_lower_bound_domain} a.e. on $\sO$ and that $a^{ij}, d^j, \fw, \vartheta$ obey \eqref{eq:Lipschitzaij}, \eqref{eq:Lipschitzdj}, \eqref{eq:LipschitzLogWeight}, \eqref{eq:ContinuousWeightDegenCoeffUpToBdry}, \eqref{eq:WeightDegenCoeffZeroGammaOne}, and
\begin{equation}
\label{eq:vzeroinH2}
\int_\sO(1+|x|^2)(1+\vartheta^2)\fw\,dx < \infty,
\end{equation}
and that \eqref{eq:Quadratic_growth} is obeyed a.e. on $\sO$ by $(a,\tilde b)$ in place of $(a,b)$, where $\tilde b$ is given by \eqref{eq:tildebcoeff}. Suppose $f \in L^2(\sO,\fw)$ and $\sup_\sO f < \infty$.  If $u \in H^1(\sO,\fw)$ obeys \eqref{eq:VariationalSubsolution} and \eqref{eq:VariationalSubsolutionBC} with $g=0$ (when $\partial\sO\less \Sigma$ non-empty) and $\esssup_\sO u < \infty$, then
$$
\esssup_\sO u \leq 0\vee\frac{1}{c_0}\esssup_\sO f.
$$
Moreover, $\fa$ has the \emph{weak maximum principle property on $\sO\cup\Sigma$} in the sense of Definition \ref{defn:WeakMaxPrinciplePropertyBilinearForm}.
\end{thm}

Note that the condition \eqref{eq:c_positive_lower_bound_domain} in the hypotheses of Theorem \ref{thm:WeakMaximumPrincipleH1UnboundedDomainBoundedSolution} is equivalent to \eqref{eq:BilinearMapPositiveLowerBound}.

\begin{proof}[Proof of Theorem \ref{thm:WeakMaximumPrincipleH1UnboundedDomainBoundedSolution}]
We proceed almost exactly as in the proof of Theorem \ref{thm:Weak_maximum_principle_C2_unbounded_opensubset} and choose
$$
M :=  0\vee\esssup_\sO(f+2Ku),
$$
where $K>0$ is the constant arising in the proof of Theorem \ref{thm:Weak_maximum_principle_C2_unbounded_opensubset}. Our hypotheses on $f$ and $u$ imply that $0\leq M<\infty$. For a constant $\lambda\geq 0$, set
$$
\fa_{\lambda}(u_1,u_2) := \fa(u_1,u_2) + \lambda(u_1,u_2)_{L^2(\sO,\fw)} \quad\forall\, u_1,u_2 \in H^1(\sO,\fw).
$$
Let $v_0$ be as in \eqref{eq:Defnvzero} and note that $v_0\in H^2(\sO,\fw)$ by \eqref{eq:H2Norm} and \eqref{eq:vzeroinH2}. For $\delta>0$, choose $w\in H^1(\sO,\fw)$ as in \eqref{eq:Defnw} and observe that for all $v \in H^1_0(\sO\cup\Sigma,\fw)$ with $v\geq 0$ a.e. on $\sO$,
\begin{align*}
\fa_{2K}(w,v) &= \fa_{2K}(u,v) - \delta \fa_{2K}(v_0,v) - (c_0+2K)^{-1}\fa_{2K}(M,v)
\quad\hbox{(by \eqref{eq:Defnw})}
\\
&= \fa(u,v) + 2K(u,v)_{L^2(\sO,\fw)} - \delta((A+2K)v_0,v)_{L^2(\sO,\fw)}
\\
&\quad - (c_0+2K)^{-1}((A+2K)M,v)_{L^2(\sO,\fw)}
\\
&\qquad\hbox{(since $v_0, M \in H^2(\sO,\fw)$ and applying \eqref{eq:IntegrationByPartsFormula})}
\\
&= \fa(u,v) + 2K(u,v)_{L^2(\sO,\fw)} - \delta((A+2K)v_0,v)_{L^2(\sO,\fw)}
\\
&\quad - (c_0+2K)^{-1}((c+2K)M,v)_{L^2(\sO,\fw)}
\\
&\leq (f,v)_{L^2(\sO,\fw)} + 2K(u,v)_{L^2(\sO,\fw)} - (M,v)_{L^2(\sO,\fw)}
\\
&\qquad\hbox{(by \eqref{eq:c_positive_lower_bound_domain}, \eqref{eq:Aplus2KvgeqZero}, and \eqref{eq:VariationalSubsolution})}
\\
&\leq \left(\esssup_\sO(f+2Ku)^+,v\right)_{L^2(\sO,\fw)} - (M,v)_{L^2(\sO,\fw)} = 0,
\end{align*}
where the final equality follows from the definition of $M$. We now apply Theorem \ref{thm:WeakMaximumPrincipleVarEquationBoundedDomain} or \ref{thm:WeakMaximumPrincipleVarEquationBoundedDomainGeneralWeight} (instead of Theorem \ref{thm:Weak_maximum_principle_C2}) and the remainder of the proof is the same as that of Theorem \ref{thm:Weak_maximum_principle_C2_unbounded_opensubset}.
\end{proof}

\appendix

\section{Weak maximum principle for unbounded functions}
\label{sec:WeakMaxPrincipleUnboundedFunction}
Theorems \ref{thm:Weak_maximum_principle_C2_W2d_unbounded_function} and \ref{thm:Weak_maximum_principle_H1_unbounded_function} and Corollary \ref{cor:Weak_maximum_principle_H2_unbounded_function} gave sufficient conditions describing when the weak maximum principle holds for unbounded $C^2(\sO)$, $H^1(\sO,\fw)$, or $H^2(\sO,\fw)$ functions, respectively, obeying a growth condition \eqref{eq:Upper_bound_subsolution} defined by a function $\varphi$ on the unbounded open subset, $\sO\subseteqq\RR^d$. In this section, we give examples of such growth conditions in Theorem \ref{thm:Weak_maximum_principle_C2_W2d_unbounded_function}, primarily in the case of the Heston operator, $A$.

We first compute the coefficients of the operator $B$ in \eqref{eq:First_order_operator} explicitly, when $A$ is as in \eqref{eq:Generator}. For $v\in C^2(\sO)$,
\begin{align*}
[A,\varphi]v &= A(\varphi v) - \varphi Av
\\
&= \left(-a^{ij}\varphi_{x_ix_j} - b^i\varphi_{x_i} + c\varphi\right)v
+ \left(-a^{ij}\left(\varphi_{x_i}v_{x_j}+\varphi_{x_j}v_{x_i}+\varphi v_{x_ix_j}\right)  - b^i\varphi v_{x_i}\right) - \varphi Av
\\
&= -a^{ij}\left(\varphi_{x_i}v_{x_j}+\varphi_{x_j}v_{x_i}\right) - \left(a^{ij}\varphi_{x_ix_j} + b^i\varphi_{x_i}\right)v
\\
&= -\left(a^{ij}+a^{ji}\right)\varphi_{x_j}v_{x_i} - \left(a^{ij}\varphi_{x_ix_j} + b^i\varphi_{x_i}\right)v
\\
&= -\left(a^{ij}+a^{ji}\right)\varphi^{-1}\varphi_{x_j}\left((\varphi v)_{x_i} - \varphi_{x_i}v\right)
- \left(a^{ij}\varphi_{x_ix_j} + b^i\varphi_{x_i}\right)v
\\
&= -\left(a^{ij}+a^{ji}\right)(\log\varphi)_{x_j}(\varphi v)_{x_i}
- \left(a^{ij}\varphi_{x_ix_j} + b^i\varphi_{x_i} - \left(a^{ij}+a^{ji}\right)(\log\varphi)_{x_j}\varphi_{x_i}\right)v,
\end{align*}
and therefore, since $B(\varphi v)= -[A,\varphi]v$ by \eqref{eq:First_order_operator}, we see that for all $v \in C^2(\sO)$,
\begin{equation}
\label{eq:BCoords}
\begin{aligned}
Bv \equiv f^iv_{x_i} + f^0v &\equiv \left(a^{ij}+a^{ji}\right)(\log\varphi)_{x_j}v_{x_i}
\\
&\quad + \left(a^{ij}\varphi^{-1}\varphi_{x_ix_j} + b^i(\log\varphi)_{x_i} - \left(a^{ij}+a^{ji}\right)(\log\varphi)_{x_j}(\log\varphi)_{x_i}\right)v.
\end{aligned}
\end{equation}
Next, we give some examples of choices of functions, $\varphi$.

\begin{exmp}[Exponential-affine growth]
When $\varphi$ has the form
\begin{equation}
\label{eq:ExponentialAffineVarphi}
\varphi(x) = e^{-\langle h,x\rangle}, \quad\forall\, x\in\RR^d,
\end{equation}
for a fixed vector, $h\in\RR^d$, and positive constant, $C$, the expression \eqref{eq:BCoords} for $Bv$ simplifies to
$$
Bv = -\left(a^{ij}+a^{ji}\right)h_jv_{x_i} + \left(a^{ij}h_ih_j - b^ih_i - \left(a^{ij}+a^{ji}\right)h_ih_j\right)v,
$$
and thus,
\begin{equation}
\label{eq:BCoordsExponential}
Bv = -\left(a^{ij}+a^{ji}\right)h_jv_{x_i} - \left(b^ih_i + a^{ij}h_ih_j\right)v, \quad v \in C^2(\sO).
\end{equation}
Therefore, $\widehat A = (A+B)$ is given by
\begin{equation}
\label{eq:hatAExponentialAffineVarphi}
\widehat Av = -a^{ij}v_{x_ix_j} - \left(b^i + (a^{ij} + a^{ji})h_j\right)v_{x_i} + \left(c - b^ih_i - a^{ij}h_ih_j\right)v.
\end{equation}
Hence, when $\varphi$ is as in \eqref{eq:ExponentialAffineVarphi}, it is easy to tell when $\widehat A$ obeys the hypotheses of Theorem \ref{thm:Weak_maximum_principle_C2_W2d_unbounded_function}. Indeed, it suffices to ensure that the coefficient,
$$
\hat c := c - b^ih_i - a^{ij}h_ih_j,
$$
in the expression for $\widehat A$ obeys \eqref{eq:c_positive_lower_bound_domain} for suitable $h$.
\end{exmp}

\begin{exmp}[Elliptic Heston operator and exponential-affine growth]
Suppose $h=(h_1,h_2)=(L,N)$, where $L\geq 0$ and $N\geq 0$. From the identification of the coefficients, $(a,b, c)$, for the elliptic Heston operator, $A$, in Remark \ref{rmk:EllipticHestonC2UnboundedDomain}, we see that
\begin{align*}
\hat c &= c - b^ih_i - a^{ij}h_ih_j
\\
&= r - \left(r-q-\frac{y}{2}\right)L - \kappa(\theta-y)N - \frac{y}{2}\left(L^2 + 2\varrho\sigma LN + \sigma^2 N^2\right)
\\
&= \frac{y}{2}\left(L + 2\kappa N - L^2 - 2\varrho\sigma LN - \sigma^2 N^2\right) + r - \kappa\theta N - (r-q)L.
\end{align*}
Therefore, provided the coefficients obey
\begin{equation}
\label{eq:EllipticHestonLNConditionPositivex}
L + 2\kappa N - L^2 - 2\varrho\sigma LN - \sigma^2 N^2 \geq 0 \quad\hbox{and}\quad r - \kappa\theta N - (r-q)L > 0,
\end{equation}
we see that $\hat c$ obeys \eqref{eq:c_positive_lower_bound_domain}, as desired.
\end{exmp}

\begin{exmp}[Exponential-quadratic growth]
When $\varphi$ has the form
\begin{equation}
\label{eq:ExponentialQuadraticVarphi}
\varphi(x) = e^{-L|x|^2}, \quad\forall\, x\in\RR^d,
\end{equation}
for some positive constant, $L$, the expression \eqref{eq:BCoords} for $Bv$ simplifies to give
$$
Bv = -2L\left(a^{ij}+a^{ji}\right)x_jv_{x_i} + \left(2L^2a^{ij}\delta_{ij} - 2Lb^ix_i - 4L^2\left(a^{ij}+a^{ji}\right)x_ix_j\right)v.
$$
Therefore, in this case, $\widehat A$ is given by
\begin{equation}
\label{eq:hatAExponentialQuadraticVarphi}
\begin{aligned}
\widehat Av &= -a^{ij}v_{x_ix_j} - \left(b^i + 2L\left(a^{ij}+a^{ji}\right)x_j\right)v_{x_i}
\\
&\quad + \left(c + \left(2L^2a^{ij}\delta_{ij} - 2Lb^ix_i - 4L^2\left(a^{ij}+a^{ji}\right)x_ix_j\right)\right)v.
\end{aligned}
\end{equation}
When $\varphi$ is as in \eqref{eq:ExponentialQuadraticVarphi}, one can see that $\widehat A$ will \emph{not} obey the hypotheses of Theorem \ref{thm:Weak_maximum_principle_C2_W2d_unbounded_function}, in particular the condition \eqref{eq:c_positive_lower_bound_domain}, unless $L=0$.
\end{exmp}

\section{Weak maximum principles of Fichera, Ole{\u\i}nik, and Radkevi{\v{c}} and the elliptic Heston operator}
\label{sec:FicheraAndHeston}
We compare the weak maximum principles and uniqueness theorems provided by our article with those of Fichera, Ole{\u\i}nik, and Radkevi{\v{c}} \cite{Radkevich_2009a} in the case of the elliptic Heston operator, $A$, in Example \ref{exmp:HestonPDE} on an open subset $\sO\subseteqq\HH$ and show that those of Fichera, Ole{\u\i}nik, and Radkevi{\v{c}} are strictly weaker when $0<\beta<1$.

\subsection{Verification that the Heston operator and bilinear map coefficients have the required properties}
We shall first illustrate how to choose $\fw$ so that \eqref{eq:BilinearbBound} holds for the elliptic Heston operator, $A$, on $\sO\subseteqq\HH$. Denoting $(x,y)=(x_1,x_2)$, we have
\begin{align*}
Au &= -\frac{y}{2}\left(u_{xx} + 2\varrho\sigma u_{xy} + \sigma^2 u_{yy}\right) - \left(r-q-\frac{y}{2}\right)u_x - \kappa(\theta-y)u_y + ru
\\
&= -\frac{1}{2}\left(\left(yu_x + y\varrho\sigma u_y\right)_x + \left(y\varrho\sigma u_x + y\sigma^2 u_y\right)_y\right)
\\
&\quad + \left(\frac{\varrho\sigma}{2} - \left(r-q-\frac{y}{2}\right)\right) u_x + \left(\frac{\sigma^2}{2} - \kappa(\theta-y)\right)u_y + ru
\\
&= -a^{ij}u_{x_ix_j} - \tilde b^iu_{x_i} + cu.
\end{align*}
Hence,
\begin{gather*}
a = \frac{y}{2}\begin{pmatrix}1 & \varrho\sigma \\ \varrho\sigma & \sigma^2\end{pmatrix},
\quad
\tilde b = \begin{pmatrix}-\varrho\sigma/2 + (r-q-y/2) \\ -\sigma^2/2 + \kappa(\theta-y)\end{pmatrix},
\quad c = r.
\end{gather*}
As in \cite{Daskalopoulos_Feehan_statvarineqheston}, we choose
$$
\vartheta = y \quad\hbox{and}\quad \fw = y^{\beta-1}e^{-\gamma|x|-\mu y},
$$
so that $\log\fw = (\beta-1)\log y - \gamma|x| - \mu y$ and
$$
(\log\fw)_x = -\gamma\,\sign(x) \quad\hbox{and}\quad (\log\fw)_y = (\beta-1)y^{-1} - \mu.
$$
The coefficient matrix, $a$, weight $\fw$, and degeneracy coefficient, $\vartheta$, clearly obey the regularity conditions \eqref{eq:Lipschitzaij}, \eqref{eq:LipschitzLogWeight}, \eqref{eq:ContinuousWeightDegenCoeffUpToBdry}, and \eqref{eq:WeightDegenCoeffZeroGammaOne}.

Recalling that $\beta = 2\kappa\theta/\sigma^2$ and $\mu = 2\kappa/\sigma^2$ and recalling the definition \eqref{eq:tildebcoeff} of the coefficients $\tilde b^i$, we see that
\begin{align*}
b^1 &= \tilde b^1 - \vartheta (\log\fw)_x a^{11} - \vartheta (\log\fw)_y a^{12}
\\
&= -\frac{\varrho\sigma}{2} + \left(r-q-\frac{y}{2}\right) + \frac{\gamma}{2}y\,\sign(x) - \frac{1}{2}(\beta - 1 - \mu y)\varrho\sigma
\\
&= \left(r-q-\frac{\varrho\kappa\theta}{\sigma}\right) + y\left(\frac{\gamma}{2}\sign(x) + \frac{\varrho\kappa}{\sigma} - \frac{1}{2}\right),
\end{align*}
while
\begin{align*}
b^2 &= \tilde b^2 - \vartheta (\log\fw)_x a^{21} - \vartheta (\log\fw)_y a^{22}
\\
&= -\frac{\sigma^2}{2} + \kappa(\theta-y) + \frac{\gamma}{2}y\,\sign(x)\varrho\sigma - \frac{1}{2}(\beta - 1 - \mu y)\sigma^2
\\
&= \frac{\gamma}{2}y\,\sign(x)\varrho\sigma.
\end{align*}
The resulting bilinear map agrees with that in \cite[Definition 2.2]{Daskalopoulos_Feehan_statvarineqheston}. By making use of an affine change of variables \cite[Lemma 2.2]{Daskalopoulos_Feehan_statvarineqheston}, we may assume that $r-q-\varrho\kappa\theta/\sigma = 0$ and so the expression for $b^1$ simplifies to
$$
b^1 = y\left(\frac{\gamma}{2}\sign(x) + \frac{\varrho\kappa}{\sigma} - \frac{1}{2}\right).
$$
We can easily see that the coefficients, $(a,b,c,d)$, of the bilinear map, $\fa$, associated with the elliptic Heston operator now obey the conditions \eqref{eq:BilinearaBound}, \eqref{eq:BilinearbBound}, \eqref{eq:BilinearcBound}, and \eqref{eq:GeneralNonDegeneracyNearBoundaryQuant}; note that $(d^j)=0$.

\subsection{Comparison with the weak maximum principles and uniqueness theorems of Fichera}
In the framework of Fichera (see \cite[p. 308]{Radkevich_2009a}), we let $\Sigma$ denote the subset of points $x\in\partial\sO$ where $a^{ij}(x)n_in_j=0$ (with $\vec n$ denoting the \emph{inward}-pointing unit normal vector field, as in \cite[p. 308]{Radkevich_2009a}) and the \emph{Fichera function} \cite[Equations (1.1.2) and (1.1.3)]{Radkevich_2009a} (taking into account our sign convention in \eqref{eq:Generator} for the coefficients $(a,\tilde b,c)$ of $A$) is
$$
\fb := \left(\tilde b^k - a^{kj}_{x_j}\right)n_k = \left(b^k + (\log\fw)_{x_j}a^{kj}\right)n_k.
$$
Following \cite[p. 308]{Radkevich_2009a}, we denote by $\Sigma_1\subset\Sigma$ the subset where $\fb > 0$, by $\Sigma_2\subset\Sigma$ the subset where $\fb < 0$, and by $\Sigma_0\subset\Sigma$ the subset where $\fb = 0$; the set $\partial\sO\less\Sigma$ is denoted by $\Sigma_3$. By \cite[Theorem 1.1.1]{Radkevich_2009a}, the characterization of the subsets $\Sigma, \Sigma_0, \Sigma_1, \Sigma_2, \Sigma_3$ of the boundary $\partial\sO$ remains invariant under smooth changes of the independent coordinates, $(x_1,\ldots,x_d)$. (In the work of Fichera \cite{Fichera_1960, Oleinik_Radkevic, Radkevich_2009a, Radkevich_2009b}, the boundary of the open subset $\sO\subset\RR^d$ is usually denoted by $\Sigma$ and $\Sigma^0$ is the subset of points $x\in\Sigma$ where $a^{ij}(x)n_in_j=0$.)

In our example, we have $\Sigma = \partial\sO\cap\partial\HH$ and $\vec n = (0,1)$ along $\Sigma$, so that
\begin{align*}
\fb(x_1,0) &= b^2(x_1,0) + (\log\fw)_{x_j}a^{2j} \equiv \tilde b^2(x_1,0)
\\
&= -\frac{\sigma^2}{2} + \kappa\theta = \frac{\sigma^2}{2}(\beta-1).
\end{align*}
Hence,
$$
\partial\sO\cap\partial\HH = \begin{cases}\Sigma_2 &\hbox{if }0 < \beta < 1, \\ \Sigma_1 &\hbox{if }\beta>1, \\ \Sigma_0 &\hbox{if }\beta = 1, \end{cases}
$$
while $\Sigma_3 = \HH\cap\partial\sO$. (From \cite[p. 310]{Radkevich_2009a}, one has that $\partial\sO\cap\partial\HH$ is given by $\{y=0\}$ and $-Ay = -\sigma^2/2 + \kappa(\theta-y) = \fb$, and thus $\Sigma_i' = \Sigma_i$ for $i=0,1,2$ in the notation of \cite[p. 310]{Radkevich_2009a}.)

The \emph{first boundary value problem of Fichera} \cite[Equations (1.1.4) and (1.1.5)]{Radkevich_2009a} for the operator $A$ is to find a function $u\in C^2(\sO)$ such that
$$
Au = f\hbox{ on } \sO, \quad u = g \hbox{ on } \Sigma_2\cup\Sigma_3,
$$
given a source function $f$ on $\sO$ and a boundary data function $g$ on $\Sigma_2\cup\Sigma_3$. But
$$
\Sigma_2\cup\Sigma_3 = \begin{cases}\partial\sO &\hbox{if }0<\beta<1, \\ \HH\cap\partial\sO &\hbox{if }\beta \geq 1. \end{cases}
$$
Thus, for the Heston operator, the first boundary value problem of Fichera becomes
$$
Au = f\hbox{ on } \sO, \quad u = g \hbox{ on } \begin{cases} \partial\sO &\hbox{if }0<\beta<1, \\ \HH\cap\partial\sO &\hbox{if }\beta \geq 1. \end{cases}
$$
Therefore, we see that the first boundary value problem of Fichera differs from the formulations in \cite{Daskalopoulos_Feehan_statvarineqheston, DaskalHamilton1998, Daskalopoulos_Rhee_2003, Feehan_Pop_mimickingdegen_pde} when $0<\beta<1$, where a Dirichlet boundary condition along $\partial\sO\cap\partial\HH$ is replaced by the requirement that $u$ have a regularity property, in a weighted H\"older or Sobolev sense, up to the boundary portion $\partial\sO\cap\partial\HH$ which is strictly weaker than that of the Fichera maximum principles \cite[Theorem 1.1.2 and Theorems 1.5.1 and 1.5.5]{Radkevich_2009a} for $C^2(\sO)$ or $H^1_{\loc}(\sO)$ functions, respectively. Note that
$$
\Sigma_0\cup\Sigma_1 = \begin{cases} \emptyset &\hbox{if }0<\beta<1, \\ \partial\HH\cap\partial\sO &\hbox{if }\beta \geq 1. \end{cases}
$$
In the case of $C^2(\sO)$ functions on bounded open subsets $\sO\subset\HH$, we see that the Fichera maximum principle for $C^2(\sO)$ functions \cite[Theorem 1.1.2]{Radkevich_2009a} requires that $u\in C^2(\sO\cup \Sigma_0\cup\Sigma_1)\cap C(\bar\sO)$ and $Au = f$ on $\sO\cup \Sigma_0\cup\Sigma_1$, which is \emph{stronger} than the hypothesis of our Theorem \ref{thm:Weak_maximum_principle_C2} when $\beta\geq 1$, and yields, for $r>0$,
$$
\|u\|_{C(\bar\sO)} \leq \frac{1}{r}\|f\|_{C(\bar\sO)}\vee \|g\|_{C(\Sigma_2\cup\Sigma_3)},
$$
where $\Sigma_2\cup\Sigma_3 = \HH\cap\partial\sO$ when $\beta\geq 1$ and $\Sigma_2\cup\Sigma_3 = \partial\sO$ when $0<\beta<1$. (There is a typographical error in the statement of \cite[Theorem 1.1.2]{Radkevich_2009a}, where $\Sigma_2'\cap\Sigma_3$ should be replaced by $\Sigma_2'\cup\Sigma_3$; compare \cite[Theorem 1.1.2]{Oleinik_Radkevic}.) We see that the uniqueness result, when $f=0$ on $\sO\cup \Sigma_0\cup\Sigma_1$, afforded by the Fichera maximum principle \cite[Theorem 1.1.2]{Radkevich_2009a} is \emph{weaker} than that of our Theorem \ref{thm:Weak_maximum_principle_C2} when $0<\beta < 1$, since we only require $g=0$ on $\HH\cap\partial\sO$, and \emph{not} $g=0$ on $\partial\sO$, to ensure that $u=0$ on $\bar\sO$. Indeed, the prescription of a Dirichlet boundary condition along $\partial\HH\cap\partial\sO$, when $0<\beta<1$, ensures that solutions to the first boundary value problem of Fichera are at most continuous up to $\partial\HH\cap\partial\sO$ and not smooth as in \cite{Daskalopoulos_Feehan_statvarineqheston, DaskalHamilton1998, Daskalopoulos_Rhee_2003, Feehan_Pop_mimickingdegen_pde}.

Similar remarks apply to the Fichera maximum principle for weak solutions in $L^\infty(\sO)$ \cite[Theorem 1.5.1 and 1.5.5]{Radkevich_2009a}. Furthermore, our notions of weak solution, subsolution, or supersolution differ from those of \cite[p. 318]{Radkevich_2009a}, which uses the adjoint operator $A^*$ to define these concepts for functions $u \in L^\infty(\sO)$, with a bilinear map $(u,A^*v)_{L^2(\sO)}$ and space of test functions $v\in C^2(\bar\sO)$ with
$$
v = 0 \hbox{ on } \begin{cases} \partial\sO &\hbox{if }0<\beta<1, \\ \HH\cap\partial\sO &\hbox{if }\beta \geq 1, \end{cases}
$$
and thus implies a Dirichlet boundary condition along $\partial\HH\cap\partial\sO$, when $0<\beta<1$, which is redundant in our framework of weighted H\"older or Sobolev spaces.

%
%

\bibliography{mfpde}
\bibliographystyle{amsplain}

\end{document}